\documentclass[10pt,reqno]{amsart}
\usepackage{eucal,fullpage,times,amsmath,amsthm,amssymb,mathrsfs,stmaryrd,color,enumerate,accents}
\usepackage[all]{xy}
\usepackage{url}

\newcommand{\calN}{{\mathcal{N}}}

\newcommand{\calO}{{\mathcal{O}}}
\newcommand{\calL}{{\mathcal{L}}}

\newcommand{\calA}{{\mathcal{A}}}
\newcommand{\calB}{{\mathcal{B}}}

\newcommand{\calQ}{\mathcal{Q}}

\newcommand{\calG}{\mathcal{G}}

\newcommand{\calC}{\mathcal{C}}
\newcommand{\calP}{\mathcal{P}}

\newcommand{\A}{\mathbf{A}}
\newcommand{\G}{\mathbf{G}}
\newcommand{\Z}{\mathbf{Z}}
\newcommand{\N}{\mathbf{N}}
\newcommand{\C}{\mathbf{C}}
\newcommand{\F}{\mathbf{F}}
\newcommand{\Q}{\mathbf{Q}}
\renewcommand{\P}{\mathbf{P}}

\newcommand{\Spec}{{\mathrm{Spec}}}
\newcommand{\Gal}{\mathrm{Gal}}
\newcommand{\Hom}{\mathrm{Hom}}
\newcommand{\Fun}{\mathrm{Fun}}
\newcommand{\Map}{\mathrm{Map}}

\newcommand{\Sing}{\mathrm{Sing}}
\newcommand{\Shv}{\mathrm{Shv}}
\newcommand{\GL}{\mathrm{GL}}

\newcommand{\Tor}{\mathrm{Tor}}

\newcommand{\Cartier}{\mathrm{Cartier}}
\newcommand{\Sm}{\mathrm{Sm}}

\newcommand{\Sym}{\mathrm{Sym}}
\newcommand{\Frob}{\mathrm{Frob}}
\newcommand{\Mod}{\mathrm{Mod}}
\newcommand{\Alg}{\mathrm{Alg}}
\newcommand{\Mon}{\mathrm{Mon}}
\newcommand{\Set}{\mathrm{Set}}
\newcommand{\Ab}{\mathrm{Ab}}
\newcommand{\Ann}{\mathrm{Ann}}
\newcommand{\ob}{\mathrm{ob}}

\renewcommand{\ss}{\mathrm{ss}}

\newcommand{\E}{\mathrm{E}}
\newcommand{\R}{\mathrm{R}}
\renewcommand{\L}{\mathrm{L}}

\newcommand{\Ch}{\mathrm{Ch}}
\newcommand{\Pic}{\mathrm{Pic}}
\newcommand{\et}{\mathrm{\acute{e}t}}

\newcommand{\id}{\mathrm{id}}
\newcommand{\ev}{\mathrm{ev}}

\newcommand{\Fil}{\mathrm{Fil}}

\newcommand{\Frac}{\mathrm{Frac}}
\newcommand{\gr}{\mathrm{gr}}
\newcommand{\dR}{\mathrm{dR}}
\newcommand{\ddR}{\mathrm{ddR}}
\newcommand{\red}{\mathrm{red}}

\newcommand{\perf}{\mathrm{perf}}
\renewcommand{\inf}{\mathrm{inf}}
\newcommand{\crys}{\mathrm{crys}}
\newcommand{\st}{\mathrm{st}}
\newcommand{\cl}{\mathrm{cl}}
\newcommand{\can}{\mathrm{can}}

\newcommand{\comp}{\mathcal{C}\mathrm{omp}}
\newcommand{\pre}{\mathrm{pre}}
\newcommand{\Var}{\mathcal{V}\mathrm{ar}}
\newcommand{\conj}{\mathrm{conj}}
\newcommand{\opp}{\mathrm{opp}}
\newcommand{\Free}{\mathrm{Free}}
\newcommand{\Forget}{\mathrm{Forget}}
\newcommand{\Log}{\mathrm{Log}}
\newcommand{\Nil}{\mathrm{Nil}}
\newcommand{\grp}{\mathrm{grp}}
\renewcommand{\ker}{\mathrm{ker}}
\newcommand{\cok}{\mathrm{cok}}
\newcommand{\Cone}{\mathrm{Cone}}
\newcommand{\Tot}{\mathrm{Tot}}
\newcommand{\Der}{\mathrm{Der}}
\newcommand{\Sect}{\mathrm{Sect}}
\newcommand{\val}{\mathrm{val}}

\newcommand{\fraa}{\mathfrak{a}}

\newcommand{\comment}[1]{}

\newcommand{\cosimp}[3]{\xymatrix@1{#1 \ar@<.4ex>[r] \ar@<-.4ex>[r] & {\ }#2 \ar@<0.8ex>[r] \ar[r] \ar@<-.8ex>[r] & {\ } #3 \ar@<1.2ex>[r] \ar@<.4ex>[r] \ar@<-.4ex>[r] \ar@<-1.2ex>[r] & \cdots }}
\newcommand{\colim}{\mathop{\mathrm{colim}}}

\def\rmk#1{\textcolor{red}{ #1 }}

\begin{document}
\bibliographystyle{alpha}
\newtheorem{theorem}{Theorem}[section]
\newtheorem*{theorem*}{Theorem}
\newtheorem*{condition*}{Condition}
\newtheorem*{definition*}{Definition}
\newtheorem{proposition}[theorem]{Proposition}
\newtheorem{lemma}[theorem]{Lemma}
\newtheorem{corollary}[theorem]{Corollary}
\newtheorem{claim}[theorem]{Claim}

\theoremstyle{definition}
\newtheorem{definition}[theorem]{Definition}
\newtheorem{question}[theorem]{Question}
\newtheorem{remark}[theorem]{Remark}
\newtheorem{guess}[theorem]{Guess}
\newtheorem{example}[theorem]{Example}
\newtheorem{condition}[theorem]{Condition}
\newtheorem{warning}[theorem]{Warning}
\newtheorem{notation}[theorem]{Notation}
\newtheorem{construction}[theorem]{Construction}

\title{$p$-adic derived de Rham cohomology}
\begin{abstract}
This paper studies the derived de Rham cohomology of $\F_p$ and $p$-adic schemes, and is inspired by Beilinson's work \cite{Beilinsonpadic}. Generalising work of Illusie, we construct a natural isomorphism between derived de Rham cohomology and crystalline cohomology for lci {\em maps} of such schemes, as well logarithmic variants.  These comparisons give derived de Rham descriptions of the usual period rings and related maps in $p$-adic Hodge theory. Placing these ideas in the skeleton of \cite{Beilinsonpadic} leads to a new proof of Fontaine's crystalline conjecture $C_\crys$  and Fontaine-Jannsen's semistable conjecture $C_\st$.
\end{abstract}
\author{Bhargav Bhatt}
\maketitle
\section{Introduction}

This paper grew from an attempt at lifting Beilinson's proof \cite{Beilinsonpadic} of Fontaine's $C_\dR$ conjecture in $p$-adic Hodge theory to the more refined crystalline and semistable settings. We briefly recall the surrounding picture in \S \ref{sec:bg}, and then discuss how this fits into the present paper in \S \ref{ss:results}. An actual description of the contents is available in \S \ref{sec:outline}.

\subsection{Background}
\label{sec:bg}

Let $X$ be a smooth projective variety over a characteristic $0$ field $K$. There are two Weil cohomology theories naturally associated to $X$: the de Rham cohomology  $H^*_\dR(X)$, which is a $K$-vector space equipped with the Hodge filtration, and the $p$-adic \'etale cohomology $H^*_\et(X) := H^*_\et(X_{\overline{K}},\Z_p)$, which is a $\Z_p$-module equipped with a continuous action of $\Gal(\overline{K}/K)$, for a fixed prime $p$.  These theories are often closely related:

If $K = \C$, then the classical de Rham comparison theorem identifies de Rham cohomology with Betti cohomology, and lies at the heart of Hodge theory and the theory of periods. Composition with Artin's comparison between Betti and \'etale cohomology (tensored up along some embedding $\Z_p \hookrightarrow \C$) then yields an isomorphism between de Rham and \'etale cohomologies. The key to de Rham's theorem is the following observation: the space $X(\C)$ admits sufficiently many small opens $U \subset X(\C)$ whose de Rham cohomology is trivial. This observation gives a map from $H^*_\dR(X)$ to the constant sheaf $\C$ on $X(\C)$, and thus a map of (derived) global sections
\[ \comp_\cl:H^*_\dR(X) \to H^*(X(\C),\C) \simeq H^*_\et(X) \otimes_{\Z_p} \C.\]
Having defined the map, it is easy to show that $\comp_\cl$ is an isomorphism: one can either check this locally on $X$, or simply argue that a map of Weil cohomology theories with good formal properties is automatically an isomorphism. 

Now assume that $K$ a $p$-adic local field. The analogue of the preceding complex analytic story is Fontaine's de Rham comparison conjecture $C_\dR$. Specifically,  Fontaine constructed a filtered $\Gal(\overline{K}/K)$-equivariant $K$-algebra $B_\dR$ that is complete for the filtration, and conjectured the existence of a functorial isomorphism
\[ \comp^\dR_\et:H^*_\dR(X) \otimes_K B_\dR \simeq H^*_\et(X) \otimes_{\Z_p} B_\dR \]
compatible with the tensor product filtrations and Galois actions. This statement occupies a central position in $p$-adic Hodge theory and arithmetic geometry, and has numerous applications.\footnote{The arithmetic applications are too many to list, but there are geometric applications too. For instance, the $C_\dR$ conjecture (together with some basic structure theory of $B_\dR$) implies that the Galois representation $H^*_\et(X)$ ``knows'' the Hodge numbers of $X$; this can be used to prove the birational invariance of Hodge numbers for smooth minimal models over $\C$ \cite{Itobirinv}.}  There are multiple proofs of $C_\dR$ by now (see \S \ref{ss:history}), and we briefly discuss the recent one \cite{Beilinsonpadic} as it is conceptually simple and closest to this paper. Beilinson observed that the complex analytic proof sketched above also works in the $p$-adic context provided one measures ``small opens'' using Voevodsky's $h$-topology. More precisely, he showed that (completed) de Rham cohomology sheafifies to a constant sheaf on the $h$-topology of a $p$-adic scheme; one then constructs $\comp^\dR_\et$ and shows that it is an isomorphism, just as for $\comp_\cl$ above. The two main ingredients of his proof are: de Jong's alterations theorems for constructing the desired ``small opens'' (via the $p$-divisility results of \cite{Bhattmixedcharpdiv}), and the Hodge-completed version of Illusie's derived de Rham cohomology theory for working with the de Rham cohomology of some non-smooth maps.

The Fontaine-Jannsen semistable conjecture $C_\st$ is a refinement of the $C_\dR$ conjecture takes into account the geometry of $X$ and the arithmetic of $K$ better. Let $K$ and $X$ be as above, let $K_0$ denote the maximal unramified subfield of $K$, and assume that $X$ admits a semistable model over $\calO_K$, the ring of integers of $K$. Then Kato's theory of log crystalline cohomology endows $H^*_\dR(X)$ with the following additional structures: a $K_0$-structure $H^*_\dR(X)_0$, a monodromy operator, and a Frobenius action. The semistable conjecture predicts a comparison isomorphism
\[ \comp^\st_\et:H^*_\dR(X)_0 \otimes_{K_0} B_\st \simeq H^*_\et(X) \otimes_{\Z_p} B_\st,\]
preserving all natural structures; here $B_\st$ is a filtered $\Gal(\overline{K}/K)$-equivariant $K_0$-subalgebra of $B_\dR$ that has a Frobenius action and a monodromy operator\footnote{If $X$ extends to a proper smooth $\calO_K$-scheme, then the monodromy operator on $H^*_\dR(X)_0$ is trivial, and one expects the comparison isomorphism to be defined over a smaller Galois and Frobenius equivariant filtered subalgebra $B_\crys \subset B_\st$; this is the crystalline conjecture $C_\crys$.}. This conjecture is a $p$-adic analog of Steenbrink's work \cite{SteenbrinkLMHS} on limiting mixed Hodge structures. It is also stronger than the $C_\dR$ conjecture: (a) the left hand side of $\comp^\st_\et$ (with its natural structures) recovers the $\Gal(\overline{K}/K)$-module $H^*_\et(X)[1/p]$, while the same is not true for the left hand side of $\comp^\dR_\et$, and (b) one can deduce $C_\dR$ from $C_\st$ using de Jong's theorem \cite{dJAlt}. Roughly speaking, the difference between $C_\dR$ and $C_\st$ is one of completions: the ring $B_\st$ is not complete for the Hodge filtration, so it detects more than its completed counterpart.  One major goal of this paper is give a simple conceptual proof of the $C_\st$ conjecture.

\subsection{Results}
\label{ss:results}

Our proof of $C_\st$ follows the skeleton of \cite{Beilinsonpadic} sketched above, except that we must prove non-completed analogs of all results in derived de Rham cohomology whose completed version was used in \cite{Beilinsonpadic}. In fact, this latter task takes up the bulk (see \S \ref{sec:ddr-mod-p} and \S \ref{sec:ddr-mod-p-log}) of the paper: until now (to the best of our knowledge), there were essentially no known techniques for working with the non-completed derived de Rham cohomology, e.g., one did not know a spectral sequence with computable $E_1$ terms that converged to derived de Rham cohomology. The basic observation in this paper is that Cartier theory works extremely well in the derived world in complete generality:

\begin{theorem*}[see Proposition \ref{prop:conjss}]
Let $f:X \to S$ be a morphism of $\F_p$-schemes, and let $\dR_{X/S}$ denote Illusie's derived de Rham complex. Then there exists a natural increasing bounded below separated exhaustive filtration $\Fil^\conj_\bullet$, called {\em the conjugate filtration},  of $\dR_{X/S}$ that is functorial in $f$, and has graded pieces computed by
\[ \mathrm{Cartier}_i: \gr_i^\conj(\dR_{X/S}) \simeq \wedge^i L_{X^{(1)}/S}[-i].\] 
\end{theorem*}

In particular, for any morphism $f$ as above, there is a {\em conjugate} spectral sequence that converges to derived de Rham cohomology of $f$, and has $E_1$ terms computing cohomology of the wedge powers of the (Frobenius-twisted) cotangent complex. Using this theorem, we prove several new results on derived de Rham cohomology for $p$-adic schemes. For example, we show the following non-completed version of a comparison isomorphism of Illusie:

\begin{theorem*}[see Theorem \ref{thm:ddrcryscomp}]
Let $f:X \to S$ be an lci morphism of flat $\Z/p^n$-schemes. Then there is a natural isomorphism
\[ \R f_* \dR_{X/S} \simeq \R f_* \calO_{X/S,\crys}.\]
\end{theorem*}

Here the $\calO_S$-complex on the right hand side is the relative crystalline cohomology of $f$\footnote{In \cite{BhattTorsCrys}, we use the preceding comparison theorem and the conjugate filtration on derived de Rham cohomology to show that the crystalline cohomology groups of even very mildly singular projective varieties (such as stable singular curves) are infinitely generated. In fact, we fail to find a single example of a singular projective variety with finitely generated crystalline cohomology!}.  A satisfying consequence is that divided powers, instead of being introduced {\em by fiat} as in the crystalline story, appear very naturally in derived de Rham theory: they come from the divided power operations on the homology algebra of the Eilenberg-Maclane (infinite loop) space $K(\Z,2) \simeq \C\P^\infty$.  We use this result in \S \ref{sec:ddr-period-rings} to give derived de Rham descriptions of various period rings that occur in $p$-adic Hodge theory, such as Fontaine's ring $A_\crys$:

\begin{theorem*}[see Proposition \ref{prop:acrysdefn}]
There is a natural isomorphism $A_\crys \simeq \widehat{\dR_{\overline{\Z_p}/\Z_p}}$.
\end{theorem*}

The previous isomorphism can be used to ``see'' certain natural structures on $A_\crys$. For example, Fontaine's map $\beta:\Z_p(1) \to A_\crys$ (a $p$-adic version of $2 \pi i$) is recovered as a Chern class map, see Construction \ref{cons:acryschern}; the corresponding completed picture describes $B_\dR$ as in \cite[\S 1.5]{Beilinsonpadic}. With this theory in place, in \S \ref{sec:cst}, we can show:

\begin{theorem*}[see Theorem \ref{thm:compcst}]
The $C_\crys$ and $C_\st$ conjectures are true.
\end{theorem*}

As mentioned before,  this result is not new, but our method of proof is. The difficulty, as always, lies in constructing a functorial comparison map $\comp^\st_\et$. We do so by simply repeating Beilinson's construction of $\comp^\dR_\et$ using non-completed derived de Rham cohomology instead of its completed cousin; this is a viable approach thanks to the results above. The underlying principle here may be summarised as follows: for any algebraic variety $X$ over $K$, the $A_\crys$-valued \'etale cohomology of $X_{\overline{K}}$ {\em is} the $h$-sheafification of the $p$-adic derived de Rham cohomology of any $p$-adic compactification of $X_{\overline{K}}$ (see Theorem \ref{thm:padicpoincare}). A slight difference in implementation from \cite{Beilinsonpadic} is that we must use the conjugate spectral sequence, instead of the Hodge spectral sequence, to access \cite{Bhattmixedcharpdiv}.  Once $\comp^\st_\et$ is constructed, showing isomorphy is a formal argument in chasing Chern classes (analogous to the elementary fact that a distance preserving endomorphism of a normed finite dimensional real vector space is an isomorphism).

A technical detail elided above is that the $p$-adic applications (as well as the method of proof) necessitate a theory of derived de Rham cohomology in the {\em logarithmic} context. Rudiments of this can be found in \cite{OlssonLogCot}, but, again, no non-completed results were known. Hence, in \S \ref{sec:ddr-log} and \S \ref{sec:ddr-mod-p-log}, we set up elements of ``derived logarithmic geometry'' using simplicial commutative rings and monoids (we stick to the language of model categories instead of $\infty$-categories for simplicity of exposition). In particular, Gabber's logarithmic cotangent complex from \cite[\S 8]{OlssonLogCot} appears naturally in this theory (see Remark \ref{rmk:logcotunivprop}), and one has logarithmic versions of the results mentioned above, e.g., a conjugate spectral sequence for computing log derived de Rham cohomology is constructed in Proposition \ref{prop:conjsslddr}, and a comparison isomorphism with log crystalline cohomology in the lci case (almost) is shown in Theorem \ref{thm:lddrcryscomp}.

\subsection{A brief history of the comparison theorems}
\label{ss:history}
The comparison conjectures of Fontaine and Fontaine-Jannsen are a series of increasingly stronger statements comparing the $p$-adic \'etale cohomology of varieties over $p$-adic local fields  with their de Rham cohomology (see \cite{FontaineBigRings,FontaineCst, IllusieCdrBourbaki,IllusieCrysCoh}).  These conjectures were made almost three decades ago, and have proven to be extremely influential in modern arithmetic geometry. All these conjectures have been proven now: by Faltings \cite{FaltingspHT, FaltingsCrysGal, FaltingsAEE} using almost ring theory, by Niziol \cite{NiziolCrys, NiziolST} via higher algebraic K-theory, and by Tsuji \cite{TsujiST} (building on work on Bloch-Kato \cite{BlochKatopadic}, Fontaine-Messing \cite{FontaineMessing}, Hyodo-Kato \cite{HyodoKato}, and Kato \cite{KatoST}) using the syntomic topology.  More recently, Scholze has reproven these conjectures (and more) using his language of perfectoid spaces, which can be viewed as a conceptualisation of Faltings' work. However, these proofs are technically challenging (for example, Gabber and Ramero's presentation of the almost purity theorem in \cite{FaltingsAEE} takes two books \cite{GabberRameroART,GabberRameropHT}), and it was hoped that a simpler proof could be found. Such a proof was arguably found by Beilinson in \cite{Beilinsonpadic} for the de Rham comparison conjecture $C_\dR$; the present paper extends these ideas to prove the crystalline conjecture $C_\crys$ and the semistable conjecture $C_\st$. While this paper was being prepared, Beilinson has also independently found an extension \cite{Beilinsoncrystalline} of \cite{Beilinsonpadic} to prove $C_\crys$ and $C_\st$; his new proof bypasses derived de Rham cohomology in favor of the more classical log crystalline cohomology of Kato \cite[\S 5-\S 6]{KatoLogFI}. However, both the present paper and \cite{Beilinsoncrystalline} share an essential idea: using the conjugate filtration to prove a Poincare lemma for non-completed cohomology (compare the proof of Theorem \ref{thm:padicpoincare} with \cite[\S 2.2]{Beilinsoncrystalline}).

\subsection{Outline}
\label{sec:outline}
Notation and homological conventions (especially surrounding filtration convergence issues) are discussed in \S \ref{sec:notation}.  In \S \ref{sec:ddr-generalities}, we review the definition of derived de Rham cohomology from \cite[\S VIII.2]{IllusieCC2}, and make general observations; the important points are the conjugate filtration and the transitivity properties. Specialising modulo $p^n$ in \S \ref{sec:ddr-mod-p}, we construct a map from derived de Rham cohomology to crystalline cohomology in general, and show that it is an isomorphism in the case of an lci {\em morphism} (see Theorem \ref{thm:ddrcryscomp}). The main tool here is a derived Cartier theory (Proposition \ref{prop:conjss}), together with some explicit simplicial resolutions borrowed from \cite{IyengarAQ}.

Next, logarithmic analogues of the preceding results are recorded in \S \ref{sec:ddr-log} and \S \ref{sec:ddr-mod-p-log} based on Gabber's approach to the logarithmic cotangent complex from \cite[\S 8]{OlssonLogCot}; see Theorem \ref{thm:lddrcryscomp} for the best logarithmic comparison result we show.  Along the way, rudiments of ``derived logarithmic geometry'' (with simplicial commutative monoids and rings) are set up in \S \ref{sec:ddr-scr-log} and parts of \S \ref{sec:ddr-log} as indicated in \S \ref{ss:results}\footnote{We have tried to avoid using higher categorical language in the non-logarithmic story. However, model categories (or, better, $\infty$-categories) seem necessary to cleanly present the logarithmic story, at least if one wishes to not get constantly bogged down in arguments that require showing that certain constructions are independent of choices of projective resolutions (= cofibrant replacements).}. The  $p$-adic limits of all these results are catalogued in \S \ref{sec:ddr-p-adic}.

Moving from algebraic geometry towards arithmetic, we specialise the preceding results to give derived de Rham descriptions of the $p$-adic period rings in \S \ref{sec:ddr-period-rings} (as indicated in \S \ref{ss:results}). In fact, the picture extends almost completely to {\em any} integral perfectoid algebra in the sense of Scholze \cite{ScholzePerfectoid1}, as is briefly discussed in Remark \ref{rmk:acrysperfectoid}. Using these descriptions, we prove the Fontaine-Jannsen $C_\st$-conjecture in \S \ref{sec:cst} as discussed in \S \ref{ss:results}.  The key result here is crystalline $p$-adic Poincare lemma (Theorem \ref{thm:padicpoincare}). We also briefly discuss relations with other proofs of the $p$-adic comparison theorems in Remark \ref{rmk:padiccompapproaches}.

\subsection{Notation and conventions}
\label{sec:notation}
For a ring $A$, the ring $A[x] \langle x \rangle$ (or sometimes simply $A\langle x\rangle$) is the free pd-polynomial ring in one variable $x$ over $A$; in general, we use $\langle \rangle$ to denote divided power adjunctions. Any tensor product appearing is always derived unless otherwise specified. For a $\Z_p$-algebra $A$, we let $\widehat{A} := \lim_n A/p^n$ be the $p$-adic completion of $A$ unless explicitly specified otherwise. For a complex $K$ of $\Z_p$-modules, we define the derived $p$-adic completion as $\widehat{K} := \R\lim_n K \otimes_\Z \Z/p^n$; if $K$ has $\Z_p$-flat terms, then $\widehat{K}$ is computed as the termwise $p$-adic completion of $K$. Note that the notation is inconsistent in the case that $K = A$ is $\Z_p$-algebra that is not $\Z_p$-flat, in which case we will always mean the derived completion. We also set $T_p(K) := \R\Hom_\Z(\Q_p/\Z_p,K)$; there is a natural equivalence $T_p(K) \simeq \widehat{K}[-1]$. All exterior powers that occur are derived as in \cite[\S 7]{QuillenCRCNotes}.

We often employ topological terminology when talking about complexes. A complex $K$ over an abelian category $\calA$ is called {\em connective} if $\pi_{-i}(K) = H^i(K) = 0$ for $i > 0$; it is called {\em coconnective} if the preceding vanishing holds for $i < 0$ instead. We say that $K$ is {\em eventually connective} if some shift of $K$ is connective, and similarly for {\em eventually coconnective}; these notions correspond to right-boundedness and left-boundedness in the derived category. A complex $K$ is said to be {\em $n$-connected} if $\pi_i(K) = 0$ for $i \leq n$. All these notions are compatible with the usual topological ones under the Dold-Kan correspondence, which will be used without further comment.

The symbol $\Delta$ denotes the category of simplices. For a category $\calC$, we let $s\calC$ denote the category $\Fun(\Delta^\opp,\calC)$ of simplicial objects in $\calC$; dually, we use $c\calC$ to denote the category $\Fun(\Delta,\calC)$ of cosimplicial objects. For an object $X$ in a category $\calC$, we let $\calC_{/X}$ (resp. $\calC_{X/}$) denote the category of objects of $\calC$ lying over (resp. lying under) $X$, and for a map $X \to Y$, we write $\calC_{X/ /Y}$ for $\calC_{X/} \times_{\calC} \calC_{/Y}$. 

If $P_\bullet \in sc\calA$ is a simplicial cosimplicial object in an abelian category $\calA$, then we let $|P_\bullet| \in \Ch^\bullet(\calA)$ denote the cochain complex obtained by totalising the associated double complex (via direct sums); this is a homotopy-colimit over $\Delta^\opp$ when $P_\bullet$ is viewed as defining an object of $s\Ch^\bullet(\calA)$ via the Dold-Kan correspondence. The canonical filtration on each cosimplicial object $P_n$ fits together to define an increasing bounded below separated exhaustive filtration on $|P_\bullet|$ that we call the {\em conjugate filtration} $\Fil_\bullet^\conj(|P_\bullet|)$. The associated graded piece $\gr^\conj_k(|P_\bullet|)$ may be identified as the object in $\Ch^\bullet(\calA)$ defined by $\pi^k(P_\bullet)[-k]$. Dually, if $Q^\bullet \in cs\calA$ denotes a cosimplicial simplicial object in an abelian category $\calA$, then we let $\Tot(Q^\bullet) \in \Ch_\bullet(\calA)$ be the chain complex obtained by taking a homotopy-limit over $\Delta$ of $Q^\bullet$, viewed as an object of $c\Ch_\bullet(\calA)$; the canonical filtration on each simplicial object $Q^n$ fits together to define a descreasing bounded above separated complete filtration on $\Tot(Q^\bullet)$ that we also call the {\em conjugate filtration} $\Fil_\conj^\bullet(\Tot(Q^\bullet))$. The associated graded piece $\gr_\conj^k(\Tot(Q^\bullet))$ may be identified as the object in $\Ch_\bullet(\calA)$ defined by $\pi_k(Q^\bullet)[k]$.

The following facts  will be used freely. If $A_\bullet \to B_\bullet$ is a weak equivalence of simplicial rings, and $M_\bullet$ is a simplicial $A_\bullet$-module with $M_n$ flat over $A_n$ for each $n$, then the adjunction map $M_\bullet \to M_\bullet \otimes_{A_\bullet} B_\bullet$ is an equivalence of simplicial abelian groups; see \cite[\S I.3.3.2 and Corollary I.3.3.4.6]{IllusieCC1}. A map $M \to N$ of (possibly unbounded) complexes of $\Z/p^n$-modules is a quasi-isomorphism if and only if $M \otimes_{\Z/p^n} \Z/p \to N \otimes_{\Z/p^n} \Z/p$ is so; we refer to this phenomenon as ``devissage.'' 

Let $\Set$, $\Ab$, $\Mon$, and $\Alg$ be the categories of sets, abelian groups, commutative monoids, and commutative rings respectively. There are some obvious pairs of adjunctions between these categories, and we employ the following notation to refer to these: $\Free^\Set_\Ab:\Set \to \Ab$ denotes the free abelian group functor with right adjoint $\Forget^\Ab_\Set$, while $\Free^{s\Set}_{s\Ab}:s\Set \to s\Ab$ denotes the induced functor on simplicial objects, etc. A simplicial object in a concrete category (like $\Set$, $\Ab$, $\Mon$, $\Alg$, etc) is called discrete if the underlying simplicial set is so.

Finally, many theorems in the paper are formulated and proven in ring-theoretic language. The globalisation to quasi-compact quasi-separated schemes is immediate (either via rings in a topos using \cite[\S II.2.3]{IllusieCC2}, or by Mayer-Vietoris arguments and homotopy-coherence considerations), but we ignore this issue here to improve readability.

\subsection{Acknowledgements}
The author warmly thanks Sasha Beilinson for enlightening conversations and communications. The overwhelming intellectual debt this paper owes to \cite{Beilinsonpadic} is evident. Moreover, the idea that $p$-adically completed derived de Rham cohomology (but without Hodge completion) could lead to a crystalline analogue of \cite{Beilinsonpadic} was expressed as a ``hope'' by Beilinson, and was the starting point of the author's investigations. An equally substantial yet intangible debt is owed to Quillen's manuscripts \cite{QuillenCRCNotes,QuillenCRC} for teaching the author how to work with simplicial commutative rings. Special thanks are due to Johan de Jong for numerous useful conversations, especially about homological algebra, and consistent encouragement. The author is also grateful to Jacob Lurie for conversations that clarified many homotopical aspects of this work, and to Martin Olsson for a discussion of \cite{OlssonLogCot}.

\section{The derived de Rham complex}
\label{sec:ddr-generalities}

Illusie's derived de Rham complex \cite[\S VIII.2]{IllusieCC2} is a replacement for the usual de Rham complex that works better for singular morphisms; the idea, roughly, is to replace the cotangent sheaf with the cotangent complex in the definition of the usual de Rham complex. In this section, we remind the reader of the definition, and some basic properties; these depend on a good understanding of projective resolutions of simplicial commutative rings, and a robust formal framework for these is provided by Quillen's model structure \cite{QuillenCRCNotes} on $s\Alg$ reviewed in \S \ref{sss:modelstrsalg}.

\begin{definition}
\label{defn:ddr}
Let $A \to B$ be a ring map, and let $L_{B/A}$ denote the cotangent complex. Then the derived de Rham complex of $B$ over $A$ is defined to be $|\Omega^\bullet_{P_\bullet(B/A)/A}| \in D(\Mod_A)$ where $P_\bullet(B/A) \to B$ is the canonical free resolution of $B$ as an $A$-algebra, where $\Omega^\bullet_{C/A}$ denotes the usual de Rham complex of an $A$-algebra $C$. More generally, the same definition applies when $A \to B$ is a map of simplicial commutative rings in a topos.
\end{definition}

Elaborating on Definition \ref{defn:ddr}, observe that $\Omega^\bullet_{P_\bullet(B/A)/A}$ naturally has the structure of a simplicial cochain $A$-complex. The associated total complex $| \Omega^\bullet_{P_\bullet(B/A)/A} |$ is constructed using {\em direct sums} along antidiagonals, and may be viewed as a homotopy-colimit over $\Delta^\opp$ of the $A$-cochcain complex valued functor $n \mapsto \Omega^\bullet_{P_n(B/A)/A}$; we typically picture it as a second quadrant bicomplex. This description makes it clear that  $\dR_{B/A}$ comes equipped with an $E_\infty$-algebra structure, and a decreasing separated exhaustive multiplicative Hodge filtration $\Fil^\bullet_H$.  One can show that $\dR_{B/A}$ can be defined using any free resolution $P_\bullet \to B$, and thus the functor of $\dR_{-/A}$ commutes with filtered colimits; see also \cite[\S VIII.2.1.1]{IllusieCC2}. In fact, the functor $\dR_{-/A}$ commutes with {\em arbitrary} colimits when viewed as a functor $E_\infty$-algebras,  but we do not discuss that here (see Proposition \ref{prop:ddrbasechange} though).

\begin{remark}[Lurie]
\label{rmk:lurielke}
Fix a ring $A$, and let $s\Alg_{A/}$ be the $\infty$-category of simplicial $A$-algebras. There is a natural subcategory $\Alg^\Free_{A/} \subset s\Alg_{A/}$ spanned by free simplicially constant $A$-algebras. In fact, $\Alg^\Free_{A/}$ generates $s\Alg_{A/}$ under homotopy-colimits. The functor $B \mapsto \dR_{B/A}$ on $s\Alg_{A/}$ is the left Kan extension of $F \mapsto \Omega^\bullet_{F/A}$ on $\Alg^\Free_{A/}$.
\end{remark}

An important structure present on $\dR_{B/A}$ is the conjugate filtration:

\begin{proposition}
\label{prop:conjfiltgeneral}
Let $A \to B$ be a ring map (or a map of simplicial commutative rings). Then there exists a functorial increasing bounded below separated exhaustive filtration $\Fil^\conj_\bullet$ on $\dR_{B/A}$. This filtration can be defined using the conjugate filtration on the bicomplex $\Omega^\bullet_{P_\bullet/A}$ for any free $A$-algebra resolution $P_\bullet \to A$, and is independent of the choice of $P_\bullet$. In particular, there is a convergent spectral sequence, called the {\em conjugate spectral sequence}, of the form
\[ E_1^{p,q}:H_{p+q}(\gr^\conj_p(\dR_{B/A})) \Rightarrow H_{p+q}(\dR_{B/A}) \]
that is functorial in $A \to B$ (here we follow the homological convention that $d_r$ is a map $E_r^{p,q} \to E_r^{p-r,q+r-1}$).
\end{proposition}

\begin{proof}
The filtration in question is simply the conjugate filtration on the homotopy-colimit of a simplicial cosimplicial abelian $A$-module, as explained in \S \ref{sec:notation}; we briefly reproduce the relevant arguments for the readers convenience. Let $P_\bullet \to B$ as an $A$-algebra. Then each $\Omega^\bullet_{P_n/A}$ comes equipped with the canonical filtration by cohomology sheaves. This leads to an increasing bounded below (at $0$) separated exhaustive filtration of the bicomplex $\Omega^\bullet_{P_\bullet/A}$ (filter each column by its canonical filtration). The associated graded pieces of this filtration are naturally simplicial $A$-cochain complexes, with the $i$-th one given by simplicial $A$-cochain complex defined by $n \mapsto H^i(\Omega^\bullet_{P_n/A})[-i]$. The conjugate filtration $\Fil^\conj_\bullet$ on $\dR_{B/A}$ is simply the corresponding filtration on the associated single complex $|\Omega^\bullet_{P_\bullet/A}|$. If $F_\bullet$ is a different free resolution of $B$ as an $A$-algebra, then $F_\bullet$ is homotopy equivalent to $P_\bullet$. In particular, the simplicial $A$-cochain complexes $n \mapsto H^i(\Omega^\bullet_{P_n/A})$ and $n \mapsto H^i(\Omega^\bullet_{F_n/A})$ are homotopy equivalent, which ensures that the resulting two filtrations on the associated single complex $|\Omega^\bullet_{P_\bullet/A}| \simeq \dR_{B/A} \simeq |\Omega^\bullet_{F_\bullet/A}|$ coincide. Finally, the claim about the spectral sequence is a general fact about increasing bounded below separated exhaustive filtrations on cochain complexes; see \cite[Proposition 1.2.2.14]{LurieHA} for more on this spectral sequence. 
\end{proof}

\begin{remark}
Proposition \ref{prop:conjfiltgeneral} refers to the potentially nebulous notion of filtrations on objects of the derived category $D(\Mod_A)$ of $A$-modules. To make this precise, one could work with $D(\Fun(\N,\Mod_A))$ where $\N$ is the poset of natural numbers, viewed as a category. There is a (left Quillen) homotopy-colimit functor $D(\Fun(\N,\Mod_A)) \to D(\Mod_A)$, and the first assertion of Proposition \ref{prop:conjfiltgeneral} amounts to a canonical lift $\widetilde{\dR_{B/A}}$ of $\dR_{B/A}$ along this functor (by the formula $n \mapsto |\tau_{\leq n} \Omega^*_{P_\bullet/A}|$). The graded pieces $\gr^\conj_p(\dR_{B/A})$ described above are recovered as the cone of the map $\ev_{p-1}(\widetilde{\dR_{B/A}}) \to \ev_p(\widetilde{\dR_{B/A}})$, where $\ev_n:D(\Fun(\N,\Mod_A)) \to D(\Mod_A)$ is evaluation at $n \in \N$, etc. In the sequel, this picture will be implicit in all discussions of filtered objects in derived categories.
\end{remark}

A corollary of Proposition \ref{prop:conjfiltgeneral} if that if $f:B \to C$ is a map of $A$-algebras that induces an equivalence $\gr^\conj_i(f)$ for all $i$, then it also induces an equivalence on $\dR_{B/A} \to \dR_{C/A}$, i.e., the passage from $\dR_{B/A}$ to $\oplus_p \gr^\conj_p(\dR_{B/A})$ is conservative. A consequence is that the conjugate filtration is degenerate in characteristic $0$:

\begin{corollary}
\label{cor:ddrchar0}
Let $A \to B$ be a map of $\Q$-algebras. Then $A \simeq \dR_{B/A}.$
\end{corollary}
\begin{proof}
Let $A \to P_\bullet \to B$ be a free resolution of $B$ relative to $A$. Then $\Omega^\bullet_{P_n/A} \simeq A[0]$ as polynomial algebras in characteristic $0$ have no de Rham cohomology. It follows that  $\gr^\conj_i(\dR_{B/A}) = 0$ for $i > 0$, and $\gr^\conj_0(\dR_{B/A}) = A$. The convergence of the conjugate spectral sequence then does the rest.
\end{proof}

\begin{remark}
Corollary \ref{cor:ddrchar0} renders derived de Rham theory useless in characteristic $0$. A satisfactory fix is to {\em define} $\dR_{B/A}$ as the Hodge-completed version of the complex used above; roughly speaking, this amounts to using the product totalisation instead of the direct sum totalisation when defining the derived de Rham complex. More practically, the derived Hodge-to-de-Rham spectral sequence is forced to converge, which immediately gives meaning to the resulting theory as it specialises to classical de Rham cohomology for smooth maps. This is also the variant used in \cite{Beilinsonpadic}, but is insufficient for the $p$-adic applications of \S \ref{sec:cst}. In \cite{Bhatthodgecompleteddrchar0}, we show that this Hodge-completed theory {\em always} coincides with Hartshorne's algebraic de Rham cohomology \cite{HartshorneAlgdR} for finite type maps of noetherian $\Q$-schemes (and thus with Betti cohomology over $\C$), generalising Illusie's theorem \cite[Theorem VIII.2.2.8]{IllusieCC2} from the lci case.
\end{remark}

We will see later that the conjugate filtration is quite non-trivial away from characteristic $0$, and, in fact, forms the basis of most of our computations. We end this section by discussing the behaviour under tensor products.

\begin{proposition}
\label{prop:ddrbasechange}
Let $A \to B$ and $A \to C$ be ring maps. Then we have the Kunneth formula
\[ \dR_{B \otimes_A C/A} \simeq \dR_{B/A} \otimes_A \dR_{C/A} \]
and a base change formula.
\[ \dR_{B/A} \otimes_A C \simeq \dR_{B \otimes_A C/C}, \]
where all tensor products are derived.
\end{proposition}
\begin{proof}
Both claims are clear when the algebras involved are polynomial $A$-algebras. The general case follows from this by passage to free resolutions.
\end{proof}

\section{Derived de Rham cohomology modulo $p^n$}
\label{sec:ddr-mod-p}

In this section, we investigate derived de Rham cohomology for maps of $\Z/p^n$-algebras. By an elementary devissage, almost all problems considered reduce to the case of $\F_p$-algebras. In this positive characteristic setting,  our main observation is that {\em derived} Cartier theory gives a useable description of derived de Rham cohomology, and can be effectively used to reduce questions in derived de Rham theory to questions about the cotangent complex.

\begin{notation}[Frobenius twists]
\label{not:frobtwistddr}
Let $f:A \to B$ be a map of $\F_p$-algebras. Let $\Frob_A:A \to A$ be the Frobenius morphism on $A$, and define $B^{(1)} := B \otimes_{A,\Frob} A = B \otimes_A \Frob_* A = \Frob_A^* B$ to be the Frobenius twist of $A$, viewed as a simplicial commutative ring; explicitly, if $P_\bullet \to B$ denotes a free resolution of $B$ over $A$, then $P_\bullet \otimes_A \Frob_* A$ computes $B^{(1)}$. If $\Tor_i^A(\Frob_* A, B) = 0$ for $i > 0$, then $B^{(1)}$ coincides with the usual (underived) Frobenius twist, which will be the primary case of interest to us. The following diagram and maps will be used implicity when talking about these twists:
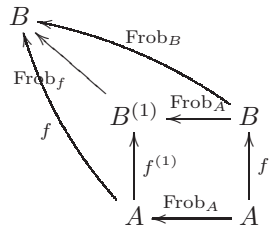
\[ \xymatrix{ B  & & \\
			  & B^{(1)} \ar[ul]^-{\Frob_f} & B \ar@/_/[ull]_-{\Frob_B} \ar[l]_-{\Frob_A} \\
			  & A \ar@/^/[uul]^-f \ar[u]_-{f^{(1)}}  & A \ar[l]_-{\Frob_A} \ar[u]_-f } \]

\end{notation}

The main reason to introduce (derived) Frobenius twists (for us) is that $\dR_{B/A}$ is naturally a complex of $B^{(1)}$-modules; this can be seen directly in the case of polynomial algebras,  and thus follows in general. 

\subsection{Review of classical Cartier theory}
\label{sss:cartierfreealg}
We briefly review the classical Cartier isomorphism in the context of {\em free} algebras; see \cite[Theorem 1.2]{DeligneIllusie} for more. 

\begin{theorem}[Classical Cartier isomorphism]
\label{thm:classicalcartier}
Let $A \to F$ be a free algebra with $A$ an $\F_p$-algebra. Then there is a canonical isomorphism of $F^{(1)}$-modules
\[ C^{-1}:\wedge^k L_{F^{(1)}/A} \simeq H^k(\Omega^\bullet_{F/A}) \]
which extends to a graded $F^{(1)}$-algebra isomorphism
\[ C^{-1}: \oplus_{k \geq 0} \wedge^k L_{F^{(1)}/A}[-k] \to \oplus_{k \geq 0} H^k(\Omega^\bullet_{F/A})[-k]. \]
\end{theorem}
\begin{proof}
To define $C^{-1}$, we may reduce to the case $A = \F_p$ via base change. Once over $\F_p$, we can pick a deformation $\widetilde{F}$ of the free algebra $F$ to $W_2$ together with a compatible lift $\widetilde{\Frob}:\widetilde{F} \to \widetilde{F}$ of Frobenius. We then define
\[ C^{-1} = \frac{\widetilde{\Frob}^*}{p}\]
in degree $1$, and extend it by taking exterior products; this makes sense because $\widetilde{\Frob}^*:\Omega^1_{\widetilde{F}/W_2} \to \Omega^1_{\widetilde{F}/W_2}$ is divisible by $p$ (as it is zero modulo $p$). One can then check using local co-ordinates that this recipe leads to the desired description of $H^p(\dR_{F/A})$. 
\end{proof}

\begin{remark}
Continuing the notation of (the proof of) Theorem \ref{thm:classicalcartier}, we note that one can do slightly better than stated: taking tensor products shows that $\widetilde{F}^*:\Omega^i_{\widetilde{F}/W_2} \to \Omega^i_{\widetilde{F}/W_2}$ is divisible by $p^i$, and hence $0$ for $i \geq 2$. It follows that the definition for $C^{-1}$ given above leads to an equivalence of {\em complexes}
\[ \oplus_{k \geq 0} \wedge^k L_{F^{(1)}/A}[-k] \stackrel{\simeq}{\to} \Omega^\bullet_{F/A}, \]
i.e., that the de Rham complex $\Omega^\bullet_{F/A}$ is {\em formal}. This decomposition depends on the choices of $\widetilde{F}$ and $\widetilde{\Frob}$, but the resulting map on cohomology is independent of these choices.
\end{remark}

\begin{remark}
Theorem \ref{thm:classicalcartier} is also true when the free algebra $F$ is replaced by any smooth $A$-algebra $B$.  A direct way to see this is to observe that both sides of the isomorphism $C^{-1}$ occurring in Theorem \ref{thm:classicalcartier} localise for the \'etale topology on $F^{(1)}$; since smooth morphisms $A \to B$ are obtained from polynomial algebras by \'etale localisation, the claim follows Zariski locally on $\Spec(B)$, and hence globally by patching. The underlying principle here of localising the de Rham cohomology on the Frobenius twist will play a prominent role in this paper (in the derived context).
\end{remark}

\subsection{Derived Cartier theory}

We begin by computing the graded pieces of the conjugate filtration in characteristic $p$.

\begin{proposition}[Derived Cartier isomorphism]
\label{prop:conjss}
Let $A \to B$ be a map of $\F_p$-algebras. Then the conjugate filtration $\Fil^\conj_\bullet$ on $\dR_{B/A}$ is $B^{(1)}$-linear, and has graded pieces computed by 
\[ \Cartier_i:\gr^\conj_i(\dR_{B/A}) \simeq \wedge^i L_{B^{(1)}/A}[-i].\]
In particular, the conjugate spectral sequence takes the form
\[ E_1^{p,q}: H_{2p+q}(\wedge^p L_{B^{(1)}/A}) \Rightarrow H_{p+q}(\dR_{B/A}). \]
\end{proposition}
\begin{proof} 
Let $P_\bullet \to B$ be the canonical free resolution of $B$ over $A$ by free $A$-algebras. The associated graded pieces $\gr^\conj_i(\dR_{B/A})$ are given by totalisations of the simplicial $A$-cochain complexes determined by $n \mapsto H^i(\Omega^\bullet_{P_n/A})[-i]$. 
By Theorem \ref{thm:classicalcartier}, one has $H^i(\Omega^\bullet_{P_n/A})[-i] \simeq \Omega^i_{P_n^{(1)}/A}$, and hence
\[ \gr^\conj_i(\dR_{B/A}) \simeq \wedge^i L_{B^{(1)}/A}[-i].\]
The rest follows formally.
\end{proof}

Before discussing applications, we make a definition.

\begin{definition}
A map $A \to B$ of $\F_p$-algebras is called {\em relatively perfect} if $B^{(1)} \to B$ is an equivalence; the same definition applies to simplicial commutative $\F_p$-algebras as well. A map $A \to B$ of $\Z/p^n$-algebras is called {\em relatively perfect modulo $p$} if $A \otimes_{\Z/p^n} \F_p \to B \otimes_{\Z/p^n} \F_p$ is relatively perfect; similarly for $\Z_p$-algebras.
\end{definition}

\begin{example} 
Any \'etale map is relatively perfect, and any map between perfect $\F_p$-algebras is relatively perfect.  By base change, the map $\Z_p \to W(R)$ is relatively perfect for any perfect $\F_p$-algebra $R$.
\end{example}

The connection between the preceding definition and de Rham theory is:

\begin{corollary}
\label{cor:ddretaleinv}
Let $A \to B$ be a map of $\Z/p^n$-algebras that is relatively perfect modulo $p$.  Then $L_{B/A} \simeq 0$, and $\dR_{B/A} \simeq B$. 
\end{corollary}
\begin{proof}
By devissage, we may immediately reduce to the case that $A$ is an $\F_p$-algebra, and $A \to B$ is relatively perfect. We first show that $L_{B/A} \simeq 0$. Indeed, for any $A$-algebra $B$,  the $A$-algebra map $B^{(1)} \to B$ induces the $0$ map $L_{B^{(1)}/A} \to L_{B/A}$ (resolve $B$ by free $A$-algebras, and use that Frobenius on the terms of the free resolution lifts Frobenius on $B$). Thus, if $B^{(1)} \to B$ is an isomorphism, then $L_{B/A}$ and $L_{B^{(1)}/A}$ must both be $0$. The conjugate filtration on $\dR_{B/A}$ is therefore trivial in degree $> 0$, so one obtains $\dR_{B/A} \simeq B^{(1)} \simeq B$, where the second equality follows from the relative perfectness.
\end{proof}

\begin{question} 
\label{ques:etaleinv}
What is an example of an $\F_p$-algebra map $A \to B$ with $L_{B/A} \simeq 0$ but $B^{(1)} \to B$ not an isomorphism? For $A = \F_p$ itself, we are asking for $\F_p$-algebras $B$ with $L_{B/\F_p} = 0$ that are not perfect; note that such algebras have to be discrete. It is conceivable that such examples do not exist, but we do not have a proof (except when $A \to B$ is finitely presented). This question also arose in Scholze's work \cite{ScholzePerfectoid1} on perfectoid spaces.
\end{question}

We use the derived Cartier isomorphism to show that derived de Rham cohomology coincides with classical de Rham cohomology for smooth maps:

\begin{corollary}
\label{cor:ddrindepsmooth}
Let $A \to B$ be a map of $\Z/p^n$-algebras, and let $P_\bullet \to B$ be a resolution of $B$ by {\em smooth} $A$-algebras (not necessarily free). Then there is a natural equivalence
\[ \dR_{B/A} \simeq |\Omega^\bullet_{P_\bullet/A}|. \]
In particular, if $A \to B$ is a smooth map of $\Z/p^n$-algebras, then $\dR_{B/A} \simeq \Omega^\bullet_{B/A}$.
\end{corollary}
\begin{proof}
To see this, let $Q_\bullet \to B$ be a resolution of $B$ by free $A$-algebras. By the cofibrancy of $Q_\bullet$, we can pick a map $Q_\bullet \to P_\bullet$ lying over $B$ (and hence an equivalence). This defines a map $\dR_{B/A} \to |\Omega^\bullet_{P_\bullet/A}|$. To check that this map is an equivalence, we may reduce to the case that $A$ and $B$ are both $\F_p$-algebras by devissage and and base change. In this case, the claim follows from the convergence of the conjugate spectral sequence and the fact that $L_{B^{(1)}/A}$ can be computed using $Q_\bullet^{(1)}$ or $P_\bullet^{(1)}$ (see \cite[Proposition III.3.1.2]{IllusieCC1}).
\end{proof}

\begin{question}
Observe that the proof above also shows that when $A \to B$ is smooth map of $\F_p$-algebras, then the conjugate filtration on $\dR_{B/A}$ coincides with the {\em canonical} filtration. What can be said modulo $p^n$?
\end{question}

\begin{remark}
Note that Corollary \ref{cor:ddrindepsmooth} is completely false in characteristic $0$. By Corollary \ref{cor:ddrchar0}, one has $\dR_{B/A} \simeq A$ whenever $\Q \subset A$. On the other hand, if $A \to B$ is smooth, then $B$ itself provides a smooth resolution of $B$ in the category of $A$-algebras, and the resulting de Rham cohomology groups are the usual de Rham cohomology groups of $A \to B$ (by\cite{GrothendieckAlgdR}) which need not be concentrated in degree $0$. For example, $\Q \to \Q[x,x^{-1}]$ has a one-dimensional (usual) de Rham cohomology group of degree $1$ (with generator $\frac{dx}{x}$), but no derived de Rham cohomology.
\end{remark}

Using the conjugate filtration, we can prove a connectivity estimate for derived de Rham cohomology:

\begin{corollary}
\label{cor:connectivity}
Let $A \to B$ be a map of $\Z/p^n$-algebras such that $\Omega^1_{B/A}$ is generated by $r$ elements for some $r \in \Z_{\geq 0}$. Then $\dR_{B/A}$ is $(-r-1)$-connected. 
\end{corollary}
\begin{proof}
To see this, first note that by devissage, we may reduce to the case that both $A$ and $B$ are $\F_p$-algebras. In this case, via the conjugate spectral sequence, it suffices to check that $\wedge^n L_{B^{(1)}/A}$ is $(n-r-1)$-connected for each $n$. By base change from $B$, note that a choice of generators of $\Omega^1_{B/A}$ defines a triangle of $B^{(1)}$-modules
\[ F \to L_{B^{(1)}/A} \to Q \]
with $F$ a free module of rank $r$, and $Q$ a connected $B^{(1)}$-module. The claim now follows by filtering wedge powers of $L_{B^{(1)}/A}$ using the preceding triangle, and noting that $\wedge^a F = 0$ for $a > r$, while $\wedge^b Q$ is $(b-1)$-connected.
\end{proof}

Next, we show that derived de Rham cohomology localises for the \'etale topology; note that there is no analogous description for usual de Rham cohomology in characteristic $0$.

\begin{corollary}
\label{cor:etalelocalisation}
Let $A \to B \to C$ maps of $\F_p$-algebras, and assume that $B \to C$ is \'etale (or simply that $B \to C$ is flat with $L_{C/B} = 0$). Then 
\[ \dR_{B/A} \otimes_{B^{(1)}} C^{(1)} \simeq \dR_{C/A} \quad \textrm{and} \quad H^i(\dR_{B/A}) \otimes_{B^{(1)}} C^{(1)} \simeq H^i(\dR_{C/A}),\]
where all Frobenius twists are computed relative to $A$.
\end{corollary}
\begin{proof}
The first statement implies the second by taking cohomology and using:  $\dR_{B/A}$ is a complex of $B^{(1)}$-modules while $B^{(1)} \to C^{(1)}$ is flat since it is a base change of $B \to C$ along $B \to B^{(1)}$. For the first, note that there is indeed a natural map $\dR_{B/A} \otimes_{B^{(1)}} C^{(1)} \to \dR_{C/A}$. The claim now follows by computing both sides using the conjugate spectral sequence and noting that $\wedge^q L_{B^{(1)}/A} \otimes_{B^{(1)}} C^{(1)} \simeq \wedge^q L_{C^{(1)}/A}$ since $L_{C^{(1)}/B^{(1)}} = \Frob_A^* L_{C/B} = 0$.
\end{proof}

Next, we relate the first differential of the conjugate spectral sequence (or, rather, the first extension determined by the conjugate filtration) to a liftability obstruction.  Let $f:A \to B$ be a map of $\F_p$-algebras. Then one has an exact triangle
\[ \gr^\conj_{q-1}(\dR_{B/A}) \to \Fil^\conj_q (\dR_{B/A})/ \Fil^\conj_{q-2}(\dR_{B/A}) \to \gr^\conj_q(\dR_{B/A}) \]
By Proposition \ref{prop:conjss}, we have $\gr^\conj_q(\dR_{B/A}) \simeq \wedge^q L_{B^{(1)}/A}[-q]$. The above triangle thus determines a map
\[ \ob_q:\wedge^q L_{B^{(1)}/A} \to \wedge^{q-1} L_{B^{(1)}/A}[2]\] 
We can relate $\ob_1$ to a geometric invariant of $B$ as follows:

\begin{proposition}
\label{prop:conjssfirstd}
In the preceding setup, assume that a lift $\widetilde{A}$ of $A$ to $\Z/p^2$ has been specified. Then the map $\ob_1$ coincides with the obstruction to lifting $B^{(1)}$ to $\widetilde{A}$ when viewed as a point of $\Map(L_{B^{(1)}/A},B^{(1)}[2])$. 
\end{proposition}

\begin{proof}[Sketch of proof]
We first construct $\ob_1$ explicitly. Fix a free resolution $P_\bullet \to B$, and let $\tau_{\leq 1} \Omega^\bullet_{P_n/A}$ denote the two-term cochain complex $P_n \to Z^1(\Omega^1_{P_n/A})$; the association  $n \mapsto \tau_{\leq 1} \Omega^\bullet_{P_n/A}$ defines a simplicial cochain complex totalising to $\Fil^\conj_1 \dR_{B/A}$. Identifying the cohomology of $\Omega^\bullet_{P_n/A}$ via the Cartier isomorphism then gives a exact triangle of simplicial cochain complexes
\[ P_\bullet^{(1)} \to \tau_{\leq 1} \Omega^\bullet_{P_\bullet/A} \to L_{P_\bullet^{(1)}/A}[-1].\]
Taking a homotopy-colimit and identifying the terms then gives an exact triangle of $B^{(1)}$-modules
\[ B^{(1)} \to K \to L_{B^{(1)}/A}[-1].\]
The boundary map $L_{B^{(1)}/A}[-1] \to B^{(1)}[1]$ for this triangle realises $\ob_1$.  To see the connection with liftability, observe that the boundary map $L_{P_n^{(1)}/A}[-1] \to P_n^{(1)}[1]$ defines a point of $\Map(L_{P_n^{(1)}/A},P_n^{(1)}[2])$ that is {\em canonically} identified with the point defining the obstruction to lifting $P_n^{(1)}$ to $\widetilde{A}$, i.e., with the map 
\[ L_{P_n^{(1)}/A}[-1] \stackrel{a_n}{\to} L_{A/\Z_p} \otimes_A P_n^{(1)} \stackrel{b_n}{\to} P_n^{(1)}[1]\] 
where $a_n$ is the Kodaira-Spencer map for $A \to P_n^{(1)}$ and $b_n$ is the derivation classifying the square-zero extension $\widetilde{A} \to A$ pulled back to $P_n^{(1)}$; see \cite[Theorem 3.5]{DeligneIllusie}. Taking homotopy-colimits then shows that the point $\ob_1 \in \Map(L_{B^{(1)}/A},B^{(1)}[2])$ constructed above also coincides with the map
\[ L_{B^{(1)}/A}[-1] \stackrel{a}{\to} L_{A/\Z_p} \otimes_A B^{(1)} \stackrel{b}{\to} B^{(1)}[1], \]
where $a = |a_\bullet|$ is the Kodaira-Spencer map for $A \to B^{(1)}$, while $b = |b_\bullet|$ is the derivation classifying the square-zero extension $\widetilde{A} \to A$ pulled back to $B^{(1)}$; the claim follows.
\end{proof}

\begin{remark}
As mentioned in \cite{IllusieCC2ERR}, there is a mistake in \cite[\S VIII.2.1.4]{IllusieCC2} where it is asserted that for any algebra map $A \to B$, there is a natural isomorphism $\oplus_p \wedge^p L_{B^{(1)}/A}[-p] \simeq \dR_{B/A}$ rather than simply an isomorphism of the graded pieces; we thank Beilinson for pointing out \cite{IllusieCC2ERR} to us. Based on Proposition \ref{prop:conjssfirstd}, a non-liftable (to $W_2$) singularity gives an explicit counterexample to the direct sum decomposition. A particularly simple example, due to Berthelot and Ogus, is $A = \F_p$ and $B = \F_p[x_1,\dots,x_6](x_i^p,x_1x_2 + x_3x_4 + x_5x_6)$.
\end{remark}

The classical Cartier isomorphism has an important extension \cite[Remark 5.5.1]{IllusieFrobHodge}: the description of the cohomology of the de Rham complex of a smooth morphism in terms of differentials on the Frobenius twists lifts to a description of the entire de Rham complex in the presence of $\Z/p^2$-lift of everything in sight, including Frobenius.  We show next that a similar picture is valid in the derived context:

\begin{proposition}[Liftable Cartier isomorphism]
\label{prop:derivedcartier}
Let $\widetilde{A} \to \widetilde{B}$ be a map of flat $\Z/p^2$-algebras such that there exist compatible endomorphisms $\widetilde{F_A}$ and $\widetilde{F_B}$ lifting the Frobenius endomorphisms of $A = \widetilde{A} \otimes_{\Z/p^2} \F_p$ and $B = \widetilde{B} \otimes_{\Z/p^2} \F_p$. Then there exists an equivalence of algebras
\[ \Cartier^{-1}:\oplus_{k \geq 0} \wedge^k L_{B^{(1)}/A}[-k]  \simeq \dR_{B/A}  \]
splitting the conjugate filtration from Proposition \ref{prop:conjss}.
\end{proposition}

\begin{proof}
Our proof uses the model structure on simplicial commutative rings due to Quillen \cite{QuillenHA}, see \S \ref{sss:modelstrsalg}. The liftability assumption on Frobenius shows that if $\widetilde{P_\bullet} \to \widetilde{B}$ denotes a free $\widetilde{A}$-algebra resolution of $\widetilde{B}$, then there exists a map $\widetilde{h}: \widetilde{P_\bullet} \to \widetilde{P_\bullet}$ which is compatible with $\widetilde{F_A}$ and $\widetilde{F_B}$ up to homotopy, and has a modulo $p$ reduction that is homotopic to the Frobenius endomorphism of $P_\bullet := \widetilde{P_\bullet} \otimes_{\Z/p^2} \Z/p$. Now set $\widetilde{P_\bullet}^{(1)} := \widetilde{P_\bullet} \otimes_{\widetilde{A},\widetilde{F_A}} \widetilde{A}$, and let $\widetilde{g}:\widetilde{P_\bullet}^{(1)} \to \widetilde{P_\bullet}$ denote the induced $\widetilde{A}$-algebra map. Observe that $\widetilde{P_\bullet}^{(1)}$ is cofibrant as an $\widetilde{A}$-algebra (as it is the base change of the cofibrant $\widetilde{A}$-algebra $\widetilde{P_\bullet}$ along some map $\widetilde{A} \to \widetilde{A}$), and the reduction modulo $p$ map $\widetilde{P_\bullet} \to P_\bullet$ is a fibration (since it is so as a map of simplicial abelian groups). Using general model categorical principles (more precisely, the ``covering homotopy theorem,'' see \cite[Chapter 1, page 1.7, Corollary]{QuillenHA}), we may replace $\widetilde{g}$ with a homotopic $\widetilde{A}$-algebra map to ensure that the modulo $p$ reduction of $\widetilde{g}$ is {\em equal} to the relative Frobenius map $P_\bullet^{(1)} \to P_\bullet$. With this choice, the induced map 
\[ \Omega^1(\widetilde{g}^*):  \Omega^1_{\widetilde{P_\bullet}^{(1)}/\widetilde{A}} \to \Omega^1_{\widetilde{P_\bullet}/\widetilde{A}} \]
reduces to the $0$ map modulo $p$. Taking wedge powers and using Lemma \ref{lem:nullmodp}, all the induced maps 
\[ \Omega^k(\widetilde{g}^*):  \Omega^k_{\widetilde{P_\bullet}^{(1)}/\widetilde{A}} \to \Omega^k_{\widetilde{P_\bullet}/\widetilde{A}} \]
are $0$ for $k > 1$. In particular, there exist well-defined maps 
\[ \frac{1}{p} \cdot \Omega^k(\widetilde{g}^*): \Omega^k_{P_\bullet^{(1)}/A}  \to  \Omega^k_{P_\bullet/A},\]
all $0$ for $k > 1$, with the property that the square 
\[ \xymatrix{ \Omega^1_{P_\bullet^{(1)}/A} \ar[rr]^-{\frac{1}{p} \cdot \Omega^1(\widetilde{g}^*)} \ar[d]^d & & \Omega^1_{P_\bullet/A} \ar[d]^d \\
			  \Omega^2_{P_\bullet^{(1)}/A} \ar[rr]^-{\frac{1}{p} \cdot \Omega^2(\widetilde{g}^*)}  & & \Omega^2_{P_\bullet/A} } \]		
commutes. Since the bottom map is $0$, there is a well-defined map of double complexes
\[ \Omega^1_{P_\bullet^{(1)}/A}[-1] \to \Omega^\bullet_{P_\bullet/A} \]
which totalises to give a map
\[ L_{B^{(1)}/A}[-1] \to \dR_{B/A}.\]
The Cartier isomorphism in the smooth case shows that the preceding morphism splits the conjugate filtration in degree $1$. We leave it to the reader to check that taking wedge powers and using the algebra structure on $\dR_{B/A}$ now defines the desired isomorphism
\[ \Cartier^{-1}: \oplus_{k \geq 0} \wedge^k L_{B^{(1)}/A} [-k] \to \dR_{B/A}. \qedhere\]
\end{proof}

The following lemma used in the proof of Proposition \ref{prop:derivedcartier}.

\begin{lemma}
\label{lem:nullmodp}
Let $R$ be a flat $\Z/p^2$-algebra. Let $f:K_1 \to K_2$ be a map of simplicial $R$-modules. Assume that $f \otimes_R \F_p$ is $0$ as a map of complexes. If $K_2$ has projective terms, then $\wedge^k f = 0$ for $k > 1$.
\end{lemma}
\begin{proof}[Sketch of proof]
The assumption that $f$ is $0$ modulo $p$ implies that $f$ factors as a map
\[ K_1 \to p \cdot K_2 \stackrel{i}{\hookrightarrow} K_2.\]
Moreover, since $K_2$ is projective, the derived exterior powers of $f$ are computed in the naive sense, without any cofibrant replacement of the source. Thus, $\wedge^k(f)$ factors through $\wedge^k(i)$. However, it is clear $\wedge^k(i) = 0$ for $k > 1$, and so the claim follows.
\end{proof}

\begin{remark}
The proof of Proposition \ref{prop:derivedcartier} made use of certain choices, but the end result is independent of these choices. If we have a different free resolution $\widetilde{P_\bullet}' \to \widetilde{B}$ and a different lift $\widetilde{g}'$ of Frobenius on $\widetilde{P_\bullet}'$ compatible with the chosen lift on $B$, then one can still run the same argument to get a decomposition of $\dR_{B/A}$. The resulting map $\wedge^k L_{B^{(1)}/A}[-k] \to \dR_{B/A}$ is homotopic to the one constructed in the proof above, since the lifts $\widetilde{g}$ and $\widetilde{g}'$ are homotopic as maps of $\widetilde{A}$-algebras, i.e., we may choose an $\widetilde{A}$-algebra equivalence $\widetilde{P_\bullet} \to \widetilde{P_\bullet}'$ lying over $\widetilde{B}$ that commutes with $\widetilde{g}$ and $\widetilde{g}'$, up to specified homotopy.
\end{remark}

\begin{remark}
Some homological analysis in the proof of Proposition \ref{prop:derivedcartier} becomes simpler if we specify lifts to $\Z_p$ instead $\Z/p^2$. Indeed, once $\Z_p$-lifts have been specified, the resulting map on forms (the analogue of the map labelled $\Omega^1(\widetilde{g}^*)$ above) is divisible by $p$ {\em as a map}, and hence the maps $\Omega^k(\widetilde{g}^*)$ will be divisible by $p^k$ as maps for all $k$, without modifying the original choice of $\widetilde{g}$ as we did at the start of the proof.
\end{remark}

Using Proposition \ref{prop:derivedcartier}, we can give an explicit example of a morphism of $\F_p$-algebras whose derived de Rham cohomology is not left-bounded. In particular, this shows that derived de Rham cohomology cannot arise as the cohomology of a sheaf of rings on a topos. In future work \cite{BhattDerCrys}, we will construct a {\em derived crystalline site} which will be a simplicially ringed $\infty$-topos functorially attached to a morphism $f:X \to S$ of schemes, and show that the cohomology of the structure sheaf on this topos is canonically isomorphic to derived de Rham cohomology.

\begin{example}[Non-coconnectivity of derived de Rham cohomology]
Let $A$ be a $\F_p$-algebra with the following two properties: (a) the cotangent complex $L_{A/\F_p}$ is unbounded on the left (i.e., the singularity $\Spec(A)$ is not lci), (b) the algebra $A$ admits a lift to $\Z/p^2$ along with a lift of the Frobenius map. For example, we can take $A = \F_p[x,y]/(x^2,xy,y^2)$ with the obvious lift (same equations), and obvious Frobenius lift (raise to the $p$-th power on the variables). Then the derived Cartier isomorphism shows that
\[ \dR_{A/\F_p} \simeq \oplus_{i \geq 0} \wedge^i L_{A/\F_p}[-i]. \]
In particular, the complex $\dR_{A/\F_p}$ is unbounded on the left. 
\end{example}

Next, we discuss the transitivity properties for derived de Rham cohomology. Our treatment here is unsatisfactory as we do not develop the language of coefficients in this paper. 

\begin{proposition}
\label{prop:conjfiltddrcomposite}
Let $A \to B \to C$ be a composite of maps of $\F_p$-algebras. Then $\dR_{C/A}$ admits an increasing bounded below separated exhaustive filtration with graded pieces of the form
\[ \dR_{B/A} \otimes_{\Frob_A^* B} \Frob_A^* \big( \wedge^n L_{C/B} [-n] \big),\]
where the second factor on the right hand side is the base change of $\wedge^n L_{C/B}[-n]$, viewed as an $B$-module, along the map $\Frob_A:B \to \Frob_A^* B$.
\end{proposition}
\begin{proof}
Let $P_\bullet \to B$ be a polynomial $A$-algebra resolution of $B$, and let $Q_\bullet \to P_\bullet$ be a termwise-polynomial $P_\bullet$-algebra resolution of $C$. Then $\dR_{C/A} \simeq | \Omega^\bullet_{Q_\bullet/A} |$. For each $k \in \Delta^\opp$, transitivity for de Rham cohomology (along smooth morphisms, see \cite[\S 3]{KatzNilp}) endows each complex $\Omega^\bullet_{Q_k/R}$ with an increasing bounded below separated exhaustive filtration $\Fil_\bullet$ given by the (usual) de Rham complexes $\Omega^\bullet_{P_k/A}(\tau_{\leq n} \Omega^\bullet_{Q_k/P_k})$, where $\tau_{\leq n} \Omega^\bullet_{Q_k/P_k})$ is the canonical trunction in degrees $\leq n$ of $\Omega^\bullet_{Q_k/P_k}$ equipped with the Gauss-Manin connection for the composite $A \to P_k \to Q_k$. The graded pieces of this filtration are then computed to be the (usual) de Rham complexes $\Omega^\bullet_{P_k/A}(H^n(\Omega^\bullet_{Q_k/P_k})[-n])$. By the classical Cartier isomorphism, the group $H^n(\Omega^\bullet_{Q_k/P_k})$ is computed as $\Frob_{P_k}^* \Omega^n_{Q_k/P_k}$, and the Gauss-Manin connection coincides with the induced Frobenius descent connection; see also \cite[Theorem 5.10]{KatzNilp}. Lemma \ref{lem:frobdescent} below then gives an identification
\[ \Omega^\bullet_{P_k/A}(H^n(\Omega^\bullet_{Q_k/P_k})[-n]) \simeq \Omega^\bullet_{P_k/A} \otimes_{\Frob_A^* P_k} \big(\Frob_A^* \Omega^n_{Q_k/P_k}[-n]) \simeq \Omega^\bullet_{P_k/A} \otimes_{\Frob_A^* P_k} \Frob_A^* \big(\wedge^n L_{Q_k/P_k}[-n]\big).\]
The desired claim now follows by taking a homotopy-colimit over $k \in \Delta^\opp$.
\end{proof}

\begin{remark}
Let $A \stackrel{f}{\to} B \stackrel{g}{\to} C$ be two composable maps of $\Z/p^n$-algebras. Proposition \ref{prop:conjfiltddrcomposite} is a shadow of an isomorphism $\dR_f(\dR_g) \simeq \dR_{g \circ f}$; we do not develop the language here to make sense of the left hand side, but simply point out that in the case that $f$ and $g$ are both smooth, this is the transitivity isomorphism for crystalline cohomology using Berthelot's comparison theorem between de Rham and crystalline cohomology (and Corollary \ref{cor:ddrindepsmooth}). A similarly satisfactory explanation in general will be given in \cite{BhattDerCrys}.
\end{remark}

The following general fact about Frobenius descent connections was used in Proposition \ref{prop:conjfiltddrcomposite}.

\begin{lemma}
\label{lem:frobdescent}
Let $f:A \to B$ be a map of $\F_p$-algebras that exhibits $B$ as a polynomial $A$-algebra. Let $M$ be a $B^{(1)}$-module. Then the de Rham cohomology of the Frobenius descent connection on $\Frob_f^* M$ takes the shape:
\[ \Omega^\bullet_{B/A}(\Frob_f^* M) \simeq \Omega^\bullet_{B/A} \otimes_{B^{(1)}} M.\]
\end{lemma}
\begin{proof}
This lemma is essentially tautological as the connection on $\Frob_f^* M$ is defined to be the first differential in the complex appearing on the right above.
\end{proof}

\subsection{Connection with crystalline cohomology}

Classical crystalline cohomology is very closely related to de Rham cohomology modulo $p^n$: the two theories coincide for smooth morphisms. We will show that there exists an equally tight connection classical crystalline cohomology and derived de Rham cohomology: the two theories coincide for {\em lci} morphisms. In future work \cite{BhattDerCrys}, we enhance this result by constructing {\em derived crystalline cohomology} that always coincides with derived de Rham cohomology, and also with the classical crystalline cohomology for lci maps.

We start off by constructing a natural transformation from derived de Rham cohomology to crystalline cohomology. For simplicity of notation, we restrict ourselves to the affine case.

\begin{proposition}
\label{prop:ddrcryscompmap}
Let $f:A \to B$ be a map of $\Z/p^n$-algebras. Then there is a natural map of Hodge-filtered $E_\infty$-algebras
\[ \comp_{B/A}:\dR_{B/A} \to \R\Gamma( (B/A)_\crys,\calO_\crys) \]
that is functorial in $A \to B$, and agrees with the one coming from \cite[Theorem IV.2.3.2]{BerthelotCCLNM} when $A \to B$ is smooth (via Corollary \ref{cor:ddrindepsmooth}).
\end{proposition}
We remind the reader that the right hand side is the crystalline cohomology of $\Spec(B) \to \Spec(A)$, and is defined using nilpotent thickenings of $B$ relative to $A$ (as $\Z/p^n$-algebras) equipped with a pd-structure on the ideal of definition compatible with the pd-structure on $(p)$; see \cite[Chapter IV]{BerthelotCCLNM}.

\begin{proof}
Let  $P_\bullet \to B$ be a free simplicial resolution of $B$ over $A$. For each $k \geq 0$, the map $P_k \to B$ is a surjective map from a free $A$-algebra onto $B$; let $I_k \subset P_k$ be the kernel of this map. Since we are working over $\Z/p^n$, it follows from Berthelot's theorem (see \cite[Theorem V.2.3.2]{BerthelotCCLNM}) that we have a filtered quasi-isomorphism
\[ \Omega^\bullet_{P_k(B/A)/A} \otimes_{P_k} D_{P_k}(I_k)  \stackrel{\simeq}{\to} \R\Gamma( (B/A)_\crys, \calO_\crys) \]
where $D_{P_k}(I_k)$  denotes the pd-envelope of the ideal $I_k$ compatible with the standard divided powers on $p$. As $k$ varies, we obtain a map of simplicial cochain complexes
\begin{equation}
\label{eq:dr-crys-comp}
\comp^\bullet_{B/A}:\Omega^\bullet_{P_\bullet/A} \to \Omega^\bullet_{P_\bullet/A} \otimes_{P_\bullet} D_{P_\bullet}(I_\bullet)
\end{equation}
By Berthelot's theorem, the right hand simplicial object is quasi-isomorphic to the constant simplicial cochain complex on the crystalline cohomology of $B$ relative to $A$. More precisely, the natural map
\[ \Omega^\bullet_{P_0/A} \otimes_{P_0} D_{P_0}(I_0) \to |\Omega^\bullet_{P_\bullet/A} \otimes_{P_\bullet} D_{P_\bullet}(\ker(P_\bullet \to B))|\]
is an equivalence with both sides computing the crystalline cohomology of $A \to B$. The map $\comp^\bullet_{B/A}$ then defines a map
\[ \comp_{B/A}:\dR_{B/A} \to \R\Gamma( (B/A)_\crys,\calO_\crys ) \]
in the derived category. This morphism respects the Hodge filtration $\Fil^\bullet_H$ as the map \eqref{eq:dr-crys-comp} does so. It is clear from the construction and the proof of Corollary \ref{cor:ddrindepsmooth} that this map agrees with the classical one when $A \to B$ is smooth.
\end{proof}

\begin{remark}[Lurie]
\label{rmk:lkedrcrys}
Remark \ref{rmk:lurielke} can be used to give an alternate construction of the map $\comp_{B/A}$. For $F \in \Alg_{A/}^\Free$, Berthelot's theorem \cite[Theorem V.2.3.2]{BerthelotCCLNM} gives a natural map $\Omega^\bullet_{F/A} \to \R\Gamma_\crys( (F/A)_\crys,\calO_\crys)$. If one defines the crystalline cohomology of $B \in s\Alg_{A/}$ as that of $\pi_0(B)$, then general properties of left Kan extensions give a map $\dR_{B/A} \to \R\Gamma_\crys( (B/A)_\crys,\calO_\crys)$ for any $B \in s\Alg_{A/}$. 
\end{remark}

The goal of this section is to prove the following theorem:

\begin{theorem}
\label{thm:ddrcryscomp}
Let $f:R \to B$ be a map of flat $\Z/p^n$-algebras for some $n > 0$. Assume that $f$ is lci. Then the map $\comp_{B/R}$ from Proposition \ref{prop:ddrcryscompmap} is an isomorphism.
\end{theorem}

\begin{remark}
We suspect Theorem \ref{thm:ddrcryscomp} is true without the flatness condition on $R$ (or, equivalently, on $B$ since $f$ has finite $\Tor$-dimension). However, we do not pursue this question here.
\end{remark}

Our strategy for proving Theorem \ref{thm:ddrcryscomp} is to first deal with the special case that $B = R/(f)$ for some regular element $f \in R$, and then build the general case from this one using products, Berthelot's comparison theorem, and Corollary \ref{cor:ddrindepsmooth}. The following special case therefore forms the heart of the proof:

\begin{lemma}
\label{lem:ddrcryscompbase}
Let $A \to B$ be the map $\F_p[x] \stackrel{x \mapsto 0}{\to} \F_p$. Then $\comp_{B/A}$ from Proposition \ref{prop:ddrcryscompmap} is an isomorphism.
\end{lemma}

The idea of the proof of Lemma \ref{lem:ddrcryscompbase} is very simple. The liftable Cartier isomorphism from Proposition \ref{prop:derivedcartier} lets one explicitly compute $\dR_{B/A}$, while the crystalline cohomology can be explicitly computed by Berthelot's theorem \cite[Theorem IV.2.3.2]{BerthelotCCLNM}; both sides turn out to be isomorphic to $\F_p \langle x \rangle$, the free pd-algebra in one variable over $\F_p$. Checking that $\comp_{B/A}$ is an isomorphism takes a little tracing through definitions, leading to a slightly long proof.

\begin{proof}
We fix the $\Z/p^2$-lift $\widetilde{A} := \Z/p^2[x] \to \Z/p^2 =: \widetilde{B}$ of $A \to B$ together with the lifts of Frobenius determined by the identity on $\widetilde{B}$ and $x \mapsto x^p$ on $\widetilde{A}$. Using the liftable Cartier isomorphism, the preceding choice gives a presentation
\[ \dR_{B/A} \simeq \oplus_{i \geq 0} \wedge^i L_{B^{(1)}/A}[-i]\]
as algebras. Let $I = (x) \subset A$ denote the ideal defining $A \to B$. One easily computes that $B^{(1)} \simeq \F_p[x]/(x^p)$ as an $A$-algebra, and hence 
\[ a:L_{B^{(1)}/A} \simeq I^p/I^{2p}[1] \simeq B^{(1)} \cdot y [1]\] 
is free of rank $1$ on a generator $y$ in degree $1$ that we choose to correspond to $x^p \in I^p/I^{2p}$ under the isomorphism $a$. Computing derived exterior powers then gives a presentation
\begin{equation}
\label{eq:dr-lci-comp}
\dR_{B/A} \simeq \oplus_{i \geq 0} \F_p[x]/(x^p) \cdot \gamma_i(y)
\end{equation}
as an algebra. On the other hand, since the crystalline cohomology of $A \to B$ is given by $\F_p \langle x \rangle$ (by \cite[Theorem V.2.3.2]{BerthelotCCLNM}, for example), we have a presentation
\begin{equation}
\label{eq:crys-lci-comp}
\R\Gamma( (B/A)_\crys, \calO_\crys) \simeq \F_p \langle x \rangle \simeq \oplus_{i \geq 0} \F_p[x]/(x^p) \cdot \gamma_{ip}(x)
\end{equation}
as algebras. We will show that the map $\comp_{B/A}$ respects the direct sum decompositions appearing in formulas \eqref{eq:dr-lci-comp} and \eqref{eq:crys-lci-comp}, and induces an isomorphism on each summand; the idea is to first understand the image of $y$, and then its divided powers.

\begin{claim}
\label{claim:comp-dr-crys}
The map $\comp_{B/A}$ sends $y$ to $- \gamma_p(x) \in \F_p \langle x \rangle$.
\end{claim}
\begin{proof}
First, we make the derived Cartier isomorphism explicit by choosing particularly nice free resolutions and Frobenius lifts as follows. Let $\widetilde{P_\bullet} \to \Z/p^2$ be the bar resolution of $\Z/p^2$ as a $\Z/p^2[x]$-algebra as described in, say, \cite[Construction 4.13]{IyengarAQ}; see Remark \ref{rmk:barresolution} for a more functorial description. The first few (augmented) terms are:
\[ \xymatrix{\Big(\dots \ar@<0.4ex>[r] \ar[r] \ar@<-0.4ex>[r] & \Z/p^2[x,t] \ar@<.4ex>[r] \ar@<-.4ex>[r] & \Z/p^2[x]\Big) \ar[r]^{\simeq} & \Z/p^2 } \]
where the two $\Z/p^2[x]$-algebra maps from $\Z/p^2[x,t] \to \Z/p^2[x]$ are given by $t \mapsto x$ and $t \mapsto 0$ respectively. This resolution has the property that the terms $\widetilde{P_n}$ are polynomial algebras $\Z/p^2[x][X_n]$ over a set $X_n$ with $n$ elements, and the simplicial $\Z/p^2[x]$-algebra map $\Z/p^2[x][X_n] \to \Z/p^2[x][X_m]$ lying over a map $[m] \to [n] \in \Map(\Delta)$ is induced by a map of sets $X_n \to X_m \cup \{x,0\}$. In particular, the map $\widetilde{\Frob}_n:\Z/p^2[x][X_n] \to \Z/p^2[x][X_n]$ defined by $\widetilde{\Frob}_n(x) = x^p$ and $\widetilde{\Frob}_n(x_i) = x_i^p$ for each $x_i \in X_n$ defines an endomorphism $\widetilde{\Frob}:\widetilde{P_\bullet} \to \widetilde{P_\bullet}$ of $\widetilde{P_\bullet}$ which visibly lifts Frobenius modulo $p$, and also lies over the chosen Frobenius endomorphism of $\Z/p^2[x]$. Set $P_\bullet = \widetilde{P_\bullet} \otimes_{\Z/p^2} \F_p$. We will use the free resolution $P_\bullet \to \F_p$ together with the lift $\widetilde{P_\bullet}$ and the Frobenius endomorphism described above in order to understand the derived Cartier isomorphism and its composition with $\comp_{B/A}$.

The element $dt \in \Omega^1_{\F_p[x,t]/\F_p[x]}$ realises a generator of $H^{-1}(L_{\F_p/\F_p[x]})$ when we use $\Omega^1_{P_\bullet/\F_p[x]}$ to calculate $L_{\F_p/\F_p[x]}$. The Frobenius pullback of this class then determines a generator of $H^{-1}(L_{(\F_p[x]/(x^p))/\F_p[x]})$ which coincides with $y$; this can be easily checked. Chasing through the definition of the Cartier isomorphism, we find that the image of $y$ in $\dR_{B/A}$ is given by $t^{p-1} dt \in \Omega^1_{\F_p[x,t]/\F_p[x]}$ when the latter group is viewed as a subgroup of the group of $0$-cocyles in $\dR_{B/A} = | \Omega^\bullet_{P_\bullet/A} |$. On the other hand, after adjoining divided powers of $t$ and $x$ (i.e., moving to the crystalline side following the comparison recipe from \S \ref{prop:ddrcryscompmap}), the element $t^{p-1} dt \in \Omega^1_{\F_p[x,t]/\F_p[x]} \otimes_{\F_p[x,t]} \F_p \langle x,t \rangle$ may be written as 
\[ t^{p-1} dt = d_v\Big((p-1)! \gamma_p(t)\Big) = (p-1)! \cdot d_v(\gamma_p(t)) = d_v(\gamma_p(t)),\] 
where $d_v:\F_p \langle x,t \rangle \to \Omega^1_{\F_p[x,t]/\F_p[x]} \otimes_{\F_p[x,t]} \F_p \langle x,t \rangle$ is the vertical differential in the first column of 
\[ |\Omega^\bullet_{P_\bullet/\F_p[x]} \otimes_{P_\bullet} D_{\ker(P_\bullet \to \F_p)}(P_\bullet)|, \]
the bicomplex computing the crystalline cohomology of $\F_p[x] \to \F_p$ via the resolution $P_\bullet$. Since the sum of the vertical and horizontal differentials is $0$ in cohomology, it follows that the image of $t^{p-1} dt$ in crystalline cohomology coincides with the element 
\[ -d_H(\gamma_p(t)) \in \F_p \langle x \rangle.\] 
The horizontal differential $d_H:\F_p \langle x,t \rangle \to \F_p \langle x \rangle$ is the difference of the $\F_p[x]$-algebra maps obtained by sending $t$ to $x$ and $0$ respectively, and so we have $-d_H ( \gamma_p(t) ) = -  \gamma_p(x)$, as claimed.
\end{proof}

By Claim \ref{claim:comp-dr-crys}, the map $\comp_{B/A}$ induces an isomorphism of the first two summands appearing in formulas \eqref{eq:dr-lci-comp} and \eqref{eq:crys-lci-comp}. The rest follows by simply observing that 
\[ \gamma_k ( -  \gamma_p(x) ) = \frac{ (-1)^k \cdot x^{kp}}{k!\cdot (p!)^k} = \gamma_{kp}(x) \cdot u, \]
where $u$ is a {\em unit} in $\F_p$; the point is that $n_k := \frac{ (kp)!}{k! p^k}$ is an integer, and $n_k$ and $n_{k+1}$ differ multiplicatively by a unit modulo $p$.  Thus, $\comp_{B/A}$ induces isomorphisms on all the summands, as desired.
\end{proof}

\begin{remark}
\label{rmk:barresolution}
The bar resolution used in Lemma \ref{lem:ddrcryscompbase} may be described more functorially (in the language of \S \ref{ss:modelstrsmon}) as follows. Let $\calC = \E\N$ be the category with object set $\N$, and a unique morphism $n_1 \to n_2$ if $n_2 - n_1 \geq 0$. The monoid law on $\N$ makes $\calC$ a strict symmetric monoidal category, and the obvious map on object sets defines a strict symmetric monoidal functor $\N \to \calC$. Passing to nerves gives a map $\N \to N(\calC)$ of simplicial commutative monoids. In co-ordinates, $N(\calC)_k  = \N^{k+1}$ with the identification sending $(n_0,\dots,n_k) \in \N^{k+1}$ to the $k$-simplex of $\calC$ with vertices $n_0$, $n_0 + n_1$, \dots , $n_0 + n_1 + \dots + n_k$. The map $\N \to N(\calC)_k = \N^{k+1}$ is simply $n \mapsto (n,0,0,\dots,0)$, i.e., map $n \in \N$ to the constant $k$-simplex based at $n \in \calC$.  The category $\calC$ has an initial object, so the augmentation $N(\calC) \to \ast$ is a weak equivalence (see \cite[Corollary 2, page 84]{QuillenAlgKtheory1}). Moreover, the explicit description makes it clear that $\N \to N(\calC)$ is a termwise free $\N$-algebra (using \cite[Remark 4, page 4.11]{QuillenHA}, one can also check that $N(\calC)$ is cofibrant in $s\Mon_{\N/}$). By Proposition \ref{prop:modelstrsalgsmon}, the induced free algebra map $\Z[\N] \to \Z[N(\calC)]$ in $s\Alg_{\Z[\N]/}$ is a free $\Z[\N]$-algebra resolution of $\Z[\N] \stackrel{\N \mapsto 1}{\to} \Z$. Identifying $\Z[\N] \simeq \Z[t]$ via $1 \in \N \mapsto t$ gives an explicit free $\Z[t]$-algebra resolution of $\Z[t] \stackrel{t \mapsto 1}{\to} \Z$. The reduction modulo $p^2$ of this resolution (up to a change of variables $t \mapsto t+1$) was used in Lemma \ref{lem:ddrcryscompbase}; the Frobenius lift is induced by the multiplication by $p$ map on monoids.
\end{remark}

\begin{remark}
The proof of Lemma \ref{lem:ddrcryscompbase} ``explains'' the two algebra isomorphisms 
\[\Big(\F_p[x]/(x^p)\Big)\langle y \rangle \simeq \oplus_{i \geq 0} \F_p[x]/(x^p) \cdot \gamma_{ip}(x) \simeq \F_p \langle x \rangle, \]
where the first one maps $\gamma_k(y)$ to $\gamma_{kp}(x)$. Indeed, the first one is the usual splitting of the pd-filtration on a divided power polynomial algebra, while the second one arises by the splitting the conjugate filtration on $\dR_{\F_p/\F_p[x]}$ provided by $\Z/p^2$-lifts of Frobenius. Iterating this procedure with $\F_p$ replaced by $\F_p[x]/(x^p)$  gives an isomorphism\footnote{This isomorphism seems to be known to the experts, and was discovered by the author in conversation with Andrew Snowden.} of algebras
\[ \F_p[x_0,x_1,x_2,\dots]/(x_0^p,x_1^p,x_2^p\dots) \simeq \F_p \langle x \rangle\]
defined via $x_i \mapsto (\gamma_p \circ \gamma_p \circ \dots \circ \gamma_p)(x)$, where the composition is $i$-fold with $\gamma_0$ being the identity. We do not know a derived de Rham interpretation of this isomorphism.
\end{remark}

\begin{remark}
A theorem of Illusie \cite[Corollary VIII.2.2.8]{IllusieCC2} shows that Hodge-completed derived de Rham cohomology always agrees with Hodge-completed crystalline cohomology. Lemma \ref{lem:ddrcryscompbase} can therefore be regarded as decompleted version of this theorem. The difference between the Hodge-completed and the non-completed theories is, however, rather large: the latter is degenerate in characteristic $0$ by Corollary \ref{cor:ddrchar0}, while the former is not.
\end{remark}

\begin{question} 
In the presentation \eqref{eq:crys-lci-comp}, the Hodge filtration $\Fil_H^\bullet$ coincides with divided-power filtration on the right, while in the presentation \eqref{eq:dr-lci-comp}, the conjugate filtration is realised by setting $\Fil^\conj_n$ to be the first $n$ pieces of the direct sum decomposition on the right. Thus, we have $\Fil^\conj_i \cap \Fil_H^{ip} \simeq \gr^\conj_i \simeq \oplus_{j=ip}^{ip+p-1} \gr^j_H$. What can be said about the relative positions of the Hodge and conjugate filtrations in general? 
\end{question}

\begin{remark}
\label{rmk:altproofddrcrys}
A slightly less computational proof of the main result of \S \ref{lem:ddrcryscompbase} can be given as follows; it comes at the expense of more careful bookkeeping of homotopies, an issue we largely eschew below. Illusie's theorem for the Hodge-completed comparison isomorphism \cite[Corollary VIII.2.2.8]{IllusieCC2} gives a canonical isomorphism 
\[ \dR_{\F_p/\F_p[x]}/\Fil^k_H \simeq \F_p \langle x \rangle / \Fil^k_H\]
for all $k$. In particular, there is a canonical equivalence of exact triangles
\[ \xymatrix{ \Fil^k_H/\Fil^{k+2j}_H(\dR_{\F_p/\F_p[x]}) \ar[r] \ar[d]^\simeq & \Fil^k_H/\Fil^{k+j}_H(\dR_{\F_p/\F_p[x]}) \ar[r]^-{\delta^\dR_{k,j}} \ar[d]^\simeq & \Fil^{k+j}_H / \Fil^{k+2j}_H(\dR_{\F_p/\F_p[x]})[1] \ar[d]^\simeq \\
 \Fil^k_H/\Fil^{k+2j}_H(\F_p\langle x\rangle) \ar[r] & \Fil^k_H/\Fil^{k+j}_H(\F_p\langle x\rangle) \ar[r]^-{\delta^\crys_{k,j}} & \Fil^{k+j}_H / \Fil^{k+2j}_H(\F_p \langle x \rangle)  [1], }\]
for all values of $k$ and $j$. Now we claim:

\begin{claim} 
The map $\delta^\crys_{jp,p}$ is naturally equivalent to $0$, for all values $j$ as a map of complexes of $\F_p[x]$-modules.
\end{claim}

\begin{proof}
Identifying terms explicitly, the claim amounts to showing that the short exact sequence
\[ \xymatrix{ 1 \ar[r] &  \langle \gamma_{(j+1)p}(x) \rangle / \langle \gamma_{(j+2)p}(x) \rangle \ar[r] \ar[d]^\simeq & \langle \gamma_{jp}(x) \rangle / \langle \gamma_{(j+2)p}(x) \rangle \ar[r] \ar[d]^\simeq &  \langle \gamma_{jp}(x) \rangle / \langle \gamma_{(j+1)p}(x) \rangle \ar[r] \ar[d]^\simeq &  1\\
			1 \ar[r] & \F_p[x]/(x^p) \cdot \gamma_{j+1}(p) \ar[r] & \Fil^{jp}(\F_p \langle x \rangle)/\Fil^{(j+2)p}(\F_p \langle x \rangle) \ar[r] & \F_p[x]/(x^p) \cdot \gamma_{jp}(x) \ar[r] & 1 }
 \]
is split in the category of $\F_p[x]$-modules. This follows from fact that $\gamma_{jp}(x) \in \Fil^{jp}(\F_p \langle x \rangle) / \Fil^{(j+2)p}(\F_p \langle x\rangle)$ is killed by $x^p$ (since $x^p \cdot \gamma_{jp}(x) = (j+1) \cdot p \cdot \gamma_{(j+1)p}(x) = 0$), and hence defines a splitting of the surjection above in the category of $\F_p[x]$-modules.
\end{proof}

We remark that there is some ambiguity in the choice of splittings used above, but this will cancel itself out at the end. We will use this information as follows. First, we partially totalise the (say canonical) bicomplex computing $\dR_{\F_p/\F_p[x]}$, i.e., we totalise rows $0$ through $p-1$, rows $p$ through $2p-1$, etc.; the result is still a bicomplex whose associated single complex computes $\dR_{\F_p/\F_p[x]}$. Moreover, the rows are of a very specific form: the $j$-th row is naturally quasi-isomorphic to 
\[ K_j := \Fil^{jp}_H(\dR_{\F_p/\F_p[x]})/\Fil^{(j+1)p}_H(\dR_{\F_p/\F_p[x]})[j],\] 
(and hence a perfect complex of $\F_p[x]$-modules), and the differential 
\[ K_j \to K_{j+1} \]
is identified with $\delta^{\dR}_{jp,p}$, which is itself isomorphic to $\delta^\crys_{jp,p}$, and hence equivalent to $0$ in a manner prescribed as above. We leave it to the reader to check that any such bicomplex is canonically split, i.e., we have a canonical equivalence
\[ \oplus_j K_j[-j] \simeq |K_\bullet|.\]
Putting it all together, we obtain an equivalence 
\[ \dR_{\F_p/\F_p[x]} \simeq \oplus_{j \in \Z_{\geq 0}} \Fil^{jp}_H/\Fil^{(j+1)p}_H(\dR_{\F_p/\F_p[x]}) \simeq \oplus_{j \in \Z_{\geq 0}} \Fil^{jp}_H/\Fil^{(j+1)p}_H(\R\Gamma_\crys( (\F_p/\F_p[x]),\calO_\crys)) \simeq \F_p \langle x \rangle,\]
where the first map is non-canonical and constructed using the above splittings, the second map comes from Illusie's theorem, and the last map comes from explicit construction; the choices that go into constructing the last map are exactly the ones that go into making the first map as well.
\end{remark}

Next, we show that pd-envelopes behave well under taking suitable tensor products; this is useful in passing from the situation handled in Lemma \ref{lem:ddrcryscompbase} to  complete intersections of higher codimension.

\begin{lemma}
\label{lem:kunnethforpdenv}
Let $A \to B = A/I$ be a quotient map of $\F_p$-algebras. Assume that $I$ is generated by a regular sequence $f_1,\dots,f_r$ of length $r$. Then one has
\[ D_A(I) \simeq A \langle x_1,\dots,x_r \rangle/ (x_1 - f_1,\dots,x_r - f_r) \simeq \otimes_i A \langle x \rangle/(x - f_i) \simeq \otimes_i D_A(f_i) \]
where all tensor products are derived.
\end{lemma}
\begin{proof}
The elements $f_i$ define the Koszul presentation 
\[ \wedge^2 F \to F \to I \to 0 \]
where $F = \oplus_{i = 1}^r A \cdot x_i$ is a free module of rank $r$, the map $F \to I$ is given by $x_i \mapsto f_i$, and the map $\wedge^2 F \to F$ is the usual Koszul differential determined by $x_i \wedge x_j \mapsto f_j x_i - f_i x_j$. 
Exactness properties of $\Gamma^n$ (see \cite[Corollary (A.5)]{BOGUSNotes}) give a presentation
\[ \oplus_{i = 1}^n  (\Gamma^i_A(\wedge^2 F) \otimes_A \Gamma^{n-i}_A(F)) \to \Gamma^n_A(F) \to \Gamma^n_A(I) \to 0. \]
This leads to an algebra presentation
\[ \Gamma^*_A(I) \simeq \Gamma_A^*(F) / (\langle f_j x_i - f_i x_j \rangle_{i,j}).  \]
The pd-envelope $D_A(I)$ is then obtained by the formula
\[ D_A(I) = \Gamma_A^*(I)/(x_1 - f_1,\dots,x_r - f_r) \simeq \Gamma_A^*(F)/ (\langle f_j x_i - f_i x_j \rangle_{i,j}, x_1 - f_1,\dots,x_r - f_r). \]
To simplify this, we observe that for each pair $i,j$ and each positive integer $n$, we have 
\[ \gamma_n(f_j x_i - f_i x_j) \in (x_i - f_i, x_j - f_j) \subset  (x_1 - f_1,\dots,x_r - f_r)\]
in the algebra $\Gamma_A^*(F)$. Indeed, this follows by expanding the left hand side modulo the ideal $(x_i - f_i,x_j - f_j)$, and using the equality $\gamma_n(x_i) \cdot x_j^n = \gamma_n(x_i x_j) = x_i^n \cdot \gamma_n(x_j)$. Thus, we can simplify the preceding presentation of $D_A(I)$ to write
\[ D_A(I) = \Gamma_A^*(F) / (x_1 - f_1,\dots,x_r - f_r) \simeq A \langle x_1,\dots,x_r \rangle / (x_1 - f_1,\dots,x_r - f_r). \]
Using the regularity of each $f_i$, the same reasoning also shows that
\[ D_A(f_i)  \simeq A \langle x_i \rangle / (x_i - f_i).\]
It remains to check that
\[ \otimes_{i = 1}^r A \langle x_i \rangle / (x_i - f_i) \simeq A \langle x_1,\dots,x_r \rangle / (x_1 - f_1,\dots,x_r - f_r). \]
The natural map from the left hand side to the right hand side naturally realises the latter as $\pi_0$ of the former. Hence, it suffices to check that the left hand side is discrete. The regularity of $f_i$ implies the regularity of $x_i - f_i \in A \langle x_i \rangle$. Hence, we have a chain of isomorphisms (with derived tensor products)
\begin{eqnarray*} 
A \langle x_i \rangle / (x_i - f_i) &=& A \otimes_{\F_p[t]} \big(\F_p[t] \langle x_i \rangle \stackrel{x_i-t}{\to} \F_p[t]\langle x_i \rangle\big) \quad \textrm{via} \quad t \mapsto f_i \\
&=&   A \otimes_{\F_p[t]} \F_p \langle t \rangle   \\
&=& A \otimes_{\F_p[t]} \Big( \oplus_{j \in \Z_{\geq 0}} \F_p[t]/(t^p) \cdot \gamma_{jp}(t) \Big) \\
&=&  \oplus_{j \in \Z_{\geq 0}} A/(f_i)^p \cdot \gamma_{jp}(f_i),
\end{eqnarray*}
where the last equality uses the regularity of $f_i^p \in A$.  In particular, each ring $A \langle x_i \rangle / (x_i - f_i)$ is a free module over $A/(f_i^p)$. The desired discreteness now follows by commuting the tensor product with direct sums, and using that $f_1^p,\dots,f_r^p$ is a regular sequence since $f_1,\dots,f_r$ is so.
\end{proof}

We need some some base change properties of crystalline cohomology. First, we deal with pd-envelopes.

\begin{lemma}
\label{lem:basechangepdenv}
Let $A \to B = A/I$ be a quotient map of flat $\Z/p^n$-algebras. Assume that $I$ is generated by a regular sequence. Then the pd-envelope $D_A(I)$ (compatible with divided powers on $p$) is $\Z/p^n$-flat, and its formation commutes with reduction modulo $p$, i.e., $D_A(I) \otimes_{\Z/p^n} \F_p \simeq D_{A/p}(I + (p)/(p))$.
\end{lemma}
\begin{proof}
Let $I = (f_1,\dots,f_r)$ be generated by the displayed regular sequence. The proof of Lemma \ref{lem:kunnethforpdenv} shows that
\[ D_A(I) \simeq A \langle x_1,\dots,x_r \rangle / (x_1 - f_1,\dots,x_r - f_r). \]
Thus, to show flatness over $\Z/p^n$, we may reduce to the case $r = 1$, i.e., $I = (f)$ for some regular element $f$. We need to check that $A \langle x \rangle / (x - f)$ is $\Z/p^n$-flat. The regularity of $f$ and the $\Z/p^n$-flatness of $A$ imply that  $\Big(A \langle x \rangle \stackrel{x-f}{\to} A \langle x \rangle\Big)$ is a $\Z/p^n$-flat resolution of $D_A(I)$. Since any $\Z/p^n$-module with a finite flat resolution is flat, $D_A(I)$ is also $\Z/p^n$-flat. For base change, it now suffices to show that if $f \in A$ is regular, then so is its image in $A/p$, i.e., the sequence
\[ 0 \to A/p \stackrel{f}{\to} A/p \to (A/f)/p \to 0\]
is exact. Since $A$ is $\Z/p^n$-flat, the regularity of $f$ shows that $A/f$ has finite flat dimension over $\Z/p^n$, and hence is flat as above. The desired exactness then follows from the vanishing of $\Tor^{\Z/p^n}_1(A/f,\F_p)$.
\end{proof}

Next, we show that the formation of crystalline cohomology often commutes with reduction modulo $p$.

\begin{lemma}
\label{lem:basechangecrys}
Let $A \to B$ be an lci map of flat $\Z/p^n$-algebras. Then the formation of $R\Gamma( (B/A)_\crys,\calO)$ commutes with $- \otimes_{\Z/p^n} \F_p$.
\end{lemma}
\begin{proof}
Choose a factorisation $A \to F \to B$ with $A \to F$ a polynomial algebra, and $F \to B$ an lci quotient by a regular sequence; the reduction modulo $p$ of this factorisation defines a similar presentation of $B/p$ by free $A/p$-algebras. Then by Berthelot's theorem \cite[Theorem IV.2.3.2]{BerthelotCCLNM}, $R\Gamma( (B/A)_\crys, \calO)$ is computed by the de Rham complex $\Omega^\bullet_{F/A} \otimes_A D_A(I)$, and similarly modulo $p$. By Lemma \ref{lem:basechangepdenv} and the freeness over $A$ of each $\Omega^i_{F/A}$, the formation of this complex commutes with reduction modulo $p$, so the claim follows.
\end{proof}

The preceding few lemmas  and Lemma \ref{lem:ddrcryscompbase} combine to show more instances of Theorem \ref{thm:ddrcryscomp}.

\begin{corollary}
\label{cor:ddrcryscompquot}
Let $A \to B = A/I$ be a quotient map of flat $\Z/p^n$-algebras. Assume that $I$ is generated by a regular sequence. Then the map $\comp_{B/A}$ from Proposition \ref{prop:ddrcryscompmap} is an isomorphism. 
\end{corollary}
\begin{proof}
We want to show that $\dR_{B/A} \to R\Gamma( (B/A)_\crys,\calO)$ is an isomorphism. Since the formation of both sides commutes with derived base change (by Proposition \ref{prop:ddrbasechange} and Lemma \ref{lem:basechangepdenv}), we may reduce (by devissage) to the case that $A$ and $B$ are $\F_p$-algebras, and $I = (f_1,\dots,f_r)$ is generated by a regular sequence. The target is computed as $D_A(I)$ by Berthelot's theorem \cite[Theorem IV.2.3.2]{BerthelotCCLNM}. Proposition \ref{prop:ddrbasechange} and Lemma \ref{lem:kunnethforpdenv} then immediately reduce us to the case $r = 1$, i.e. $I = (f)$ for some regular element $f \in A$. In this case, the target is $D_A(f) \simeq A \langle x \rangle / (x - f)$. To compute the source, observe that we have a commutative square
\[ \xymatrix{ \F_p [t] \ar[r]^{t \mapsto f} \ar[d]^{t \mapsto 0} & A \ar[d] \\
			   \F_p \ar[r] & A/f. } \]
This square can be checked to be a (derived) pushout using the resolution of $\F_p$ given by multiplication by $t$ on $\F_p[t]$. By base change in derived de Rham cohomology and Lemma \ref{lem:ddrcryscompbase}, we obtain
\[ \dR_{B/A} \simeq \F_p \langle t \rangle \otimes_{\F_p[t]} A \simeq \big( \F_p[t] \langle x \rangle \stackrel{x - t}{\to} \F_p[t] \langle x \rangle \big) \otimes_{\F_p[t]} A \simeq \big(A \langle x \rangle \stackrel{x - f}{\to} A \langle x \rangle\big) \simeq A \langle x \rangle/(x-f) \simeq D_A(f), \]
as desired; here the second-to-last isomorphism comes from the regularity of $x-f \in A \langle x \rangle$ which, in turn, comes from the regularity of $f \in A$.
\end{proof}

The next lemma proves a Tor-independence result for lci quotients, and is here for psychological comfort.

\begin{lemma}
\label{lem:lcitorindep}
Let $A \to B = A/I$ be a quotient map of $\F_p$-algebras. Assume that $I$ is generated by a regular sequence. Then $B^{(1)}$ is discrete, i.e., $\Frob_{\ast} A$ and $B$ are Tor-independent over $A$.
\end{lemma}
\begin{proof}
The assumption implies that $B \simeq \otimes_i A/(f_i)$; here the tensor product is derived and relative to $A$, and $I = (f_1,\dots,f_r)$ with the $f_i$'s spanning a regular sequence. The desired Tor-independence follows from the following sequence of canonical isomorphisms:
\begin{eqnarray*}
\Frob_\ast A \otimes_A B &=&  \big(\Frob_\ast A \otimes_A A/(f_1)\big) \otimes_{\Frob_\ast A} \dots \otimes_{\Frob_\ast A} \big(\Frob_{\ast} A \otimes_A A/(f_r) \big) \\
						&=&  \Frob_\ast (A/f_1^p) \otimes_{\Frob_\ast A} \dots \otimes_{\Frob_\ast A} \Frob_\ast (A/f_r^p) \\
							 &=&  \Frob_\ast \big( (A/(f_1^p) \otimes_A \dots \otimes_A A/(f_r^p) \big) \\
						&=&  \Frob_\ast \big(A/(f_1^p,\dots,f_r^p)\big).
\end{eqnarray*}
Here all tensor products are derived, and the last equality uses the regularity of the sequence $(f_1^p,\dots,f_r^p)$.
\end{proof}

Next, we discuss the conjugate filtration on pd-envelopes of ideals generated by regular sequences; this is pure algebra, but will correspond to the conjugate filtration on derived de Rham cohomology once Theorem \ref{thm:ddrcryscomp} is shown.

\begin{lemma}
\label{lem:conjfiltpdenv}
Let $A \to B = A/I$ be a quotient map of $\F_p$-algebras. Assume that $I$ is generated by a regular sequence. Then $D_A(I)$ admits a natural increasing bounded below (at $0$) separated exhaustive filtration $\Fil^\conj_\bullet$ by $B^{(1)}$-submodules, with graded pieces given by
\[ \gr^\conj_i(D_A(I)) \simeq \Gamma^i_{B^{(1)}}( \pi_0(I^{[p]} \otimes_A B^{(1)}) ) \simeq \Frob_A^* \Big(\Gamma^i_B(I/I^2)\Big),\]
 where $I^{[p]} = (f_1^p,\dots,f_r^p)$ denotes the Frobenius-twisted ideal. 
\end{lemma}

Note that the natural map $A \to D_A(I)$ sends $I^{[p]}$ to $0$ as $f^p = p \cdot \gamma_p(f) = 0$ for any $f \in I$, so $D_A(I)$ may be viewed as a $B^{(1)}$-algebra. By Lemma \ref{lem:lcitorindep}, the algebra $B^{(1)}$ is also discrete, and the $B^{(1)}$-module $\pi_0(I^{[p]} \otimes_A B^{(1)})$ is a locally free module of rank $r$, where $r$ is the length of a regular sequence generating $I$. Moreover, this $B^{(1)}$-module can be identified with the pushout of $I/I^2$ along $\Frob_A:B \to B^{(1)}$, which explains the last equality above.

\begin{proof}
The filtration can be defined by setting $\Fil^\conj_n(D_A(I))$ to be the $B^{(1)}$-submodule of $D_A(I)$ generated by $\gamma_{kp}(f)$ for $f \in I$ and $k \leq n$.  To compute this filtration, observe that if $I = (f)$ with $f \in A$ regular, then, as in the proof of Lemma \ref{lem:kunnethforpdenv}, one has 
\[ D_A(I) \simeq \oplus_{i \in \Z_{\geq 0}} A/(f^p) \cdot \gamma_{ip}(x). \]
Under this isomorphism, one has $\Fil^\conj_n(D_A(I)) \simeq \oplus_{i=0}^n A/(f^p) \cdot \gamma_{np}(f)$, i.e., the conjugate filtration coincides with the evident filtration by the number of factors on the direct sum decomposition above. The claim about associated graded pieces is clear in this case well.  The general case follows from this special case and Lemma \ref{lem:kunnethforpdenv}.
\end{proof}

\begin{remark}
\label{rmk:conjfiltpdenvaltdesc}
If $A \to B = A/I$ is a quotient by an ideal $I$ generated by a regular sequence, then one has $L_{B/A} \simeq I/I^2[1]$, and $L_{B^{(1)}/A} \simeq \pi_0(I^{[p]} \otimes_A B^{(1)})[1]$ by Lemma \ref{lem:lcitorindep}. Thus, the graded pieces of the conjugate filtration appearing in Lemma \ref{lem:conjfiltpdenv} may be rewritten as
\[ \gr^\conj_i(D_A(I)) \simeq \Gamma^i_{B^{(1)}}( \pi_0(I^{[p]} \otimes_A B^{(1)}) ) \simeq \wedge^i L_{B^{(1)}/A}[-i], \]
which brings it much closer to the derived de Rham theory by Proposition \ref{prop:conjss}. In \cite{BhattDerCrys}, we will define a notion of a ``derived pd-envelope'' $\L D_A(I)$ of an arbitrary ideal $I \subset A$ in such a way that the analogue of the previous statement is true without the assumption that $I$ is generated by a regular sequence.
\end{remark}

The conjugate filtration introduced in Lemma \ref{lem:conjfiltpdenv} respects the Gauss-Manin connection if the base comes equipped with derivations. The following lemma identifies the induced connection on the graded pieces.

\begin{lemma}
\label{lem:conjfiltpdenvconn}
Let $A \to B = A/I$ be a quotient map of $\F_p$-algebras. Assume that $I$ is generated by a regular sequene. Let $R \to A$ be another map of $\F_p$-algebras. Then the conjugate filtration $\Fil^\conj_\bullet$ from Lemma \ref{lem:conjfiltpdenv} on $D_A(I)$ is compatible with the natural $R$-linear connection $D_A(I) \to D_A(I) \otimes_A \Omega^1_{A/R}$. The induced connection on $\gr^\conj_n(D_A(I)) \simeq \Frob_A^* \Gamma^n_B(I/I^2)$ coincides with the Frobenius descent connection. 
\end{lemma}
\begin{proof}
The first claim follows directly from the description of the conjugate filtration given in the proof of Lemma \ref{lem:conjfiltpdenv}. For the second part, we first explain what the Frobenius descent connection is. The natural Frobenii on $A$ and $R$ define a diagram of simplicial commutative rings
\[ \xymatrix{ R \ar[r]^f \ar[d]^-{\Frob_R} & A \ar[r]^g \ar[d]^-{\Frob_R} & B \ar[d]^-{\Frob_R} \\
			  R \ar[r]^a \ar[rd]^f & \Frob_R^* A \ar[r]^b \ar[d]^-{\Frob_f} & \Frob_R^* B \ar[d]^-{c} \\
								& A \ar[r]^d \ar[rd]^g & \Frob_A^* B \ar[d]^-{\Frob_g} \\
								&					&   B. } \]
All squares here are cartesian. Now the free $\Frob_A^* B$-module $\gr^\conj_n(D_A(I))$ is identified with $\wedge^n L_d[-n]$ by Remark \ref{rmk:conjfiltpdenvaltdesc}. In particular, as an $A$-module, this module is the pullback along $\Frob_f$ of the $\Frob_R^* A$-module $\wedge^n L_b [-n]$, viewed as an $\Frob_R^* A$-module via restriction of scalars along $b$. For any $\Frob_R^* A$-module $M$, the pullback $\Frob_f^* M$ acquires a connection relative to $R$, which is called the Frobenius descent connection. We leave it to the reader to check that this Frobenius descent connection coincides with the standard one on (conjugate graded pieces of) $D_A(I)$.
\end{proof}

The de Rham cohomology of a module equipped with a connection coming from Frobenius descent takes a particularly nice form, and this leads to a tractable description of the de Rham complex associated to the Gauss-Manin connection acting on the conjugate filtration.

\begin{lemma}
\label{lem:conjfiltpdenvdr}
Let $R$, $A$ and $B$ be as in Lemma \ref{lem:conjfiltpdenvconn}. If the map $R \to A$ exhibits $A$ as a free $R$-algebra, then one has an identification of de Rham complexes
\[ \dR_{A/R}(\gr^\conj_n(D_A(I))) \simeq \dR_{A/R} \otimes_{\Frob_R^* A} \Frob_R^* \big(\Gamma^n_B(I/I^2)\big), \]
where the second factor on the right hand side is the base change of $\Gamma^n_B(I/I^2)$, viewed as an $A$-module, along the map $\Frob_R:A \to \Frob_R^* A$.
\end{lemma}
\begin{proof}
This lemma follows from Lemma \ref{lem:conjfiltpdenvconn} and Lemma \ref{lem:frobdescent}.
 \end{proof}

We now have enough tools to finish proving Theorem \ref{thm:ddrcryscomp}.

\begin{proof}[Proof of Theorem \ref{thm:ddrcryscomp}]
Let $R \to A \to B$ be a composite map of flat $\Z/p^n$-algebras, with $R \to A$ a free $R$-algebra, and $A \to B = A/I$ a quotient map with $I \subset A$ an ideal generated by a regular sequence.  We want to show that $\comp_{B/R}:\dR_{B/R} \to R\Gamma( (B/R)_\crys,\calO)$ is an isomorphism. Since the formation of either side commutes with base change (by Proposition \ref{prop:ddrbasechange} and Lemma \ref{lem:basechangecrys}), we may reduce (by devissage) to the case $n = 1$, i.e., we may assume that all algebras in sight are $\F_p$-algebras. By Proposition \ref{prop:conjfiltddrcomposite}, $\dR_{B/R}$ admits an increasing bounded below separated exhaustive filtration with graded pieces:
\begin{equation}
\label{eq:compthm1}
\Omega^\bullet_{A/R}(\gr^\conj_n (\dR_{B/A})) \simeq \dR_{A/R} \otimes_{\Frob_R^* A} \Frob_R^* \big( \wedge^n L_{B/A} [-n] \big).
\end{equation}
Transitivity for crystalline cohomology together with Lemma \ref{lem:conjfiltpdenvdr} show that $R\Gamma( (B/R)_\crys,\calO)$ admits an increasing bounded below separated exhaustive filtration with terms given by
\begin{equation}
\label{eq:compthm2}
 \Omega^\bullet_{A/R}(\gr^\conj_n(D_A(I))) \simeq \dR_{A/R} \otimes_{\Frob_R^* A} \Frob_R^* \big(\Gamma^n_B(I/I^2)\big) \simeq \dR_{A/R} \otimes_{\Frob_R^* A} \Frob_R^* \big( \wedge^n L_{B/A} [-n] \big),
\end{equation}
where the last equality uses Remark \ref{rmk:conjfiltpdenvaltdesc}. We leave it to the reader to check that $\comp_{B/A}$ respects both these filtrations, and induces the identity isomorphism between \eqref{eq:compthm1} and \eqref{eq:compthm2}.
\end{proof}

\begin{remark}
The identification of crystalline and derived de Rham cohomology provided in Theorem \ref{thm:ddrcryscomp} answers \cite[Question VIII.2.2.8.2]{IllusieCC2} in the case of $\Z/p^n$-algebras. The case of characteristic $0$ has a negative answer by \S \ref{cor:ddrchar0}, and hence this seems like the best possible answer.
\end{remark}

A consequence of Theorem \ref{thm:ddrcryscomp} and the Frobenius action on crystalline cohomology is the Frobenius action on $\dR_{A/(\Z/p^n)}$ for flat lci $\Z/p^n$-algebras $A$. In fact, this is a completely general phenomenon:

\begin{proposition}
\label{prop:frobactionddr}
Let $A$ be a $\Z/p^n$-algebra. Then $\dR_{A/(\Z/p^n)}$ has a canonical Frobenius action commuting with the Frobenius on $\R\Gamma( (B/A)_\crys,\calO_\crys)$ under $\comp_{B/A}$.
\end{proposition}
\begin{proof}
Let $P_\bullet \to A$ be a free $(\Z/p^n)$-algebra resolution of $A$. Then $\Omega^\bullet_{P_m/(\Z/p^n)}$ has a natural Frobenius action coming from the isomorphism of $\Omega^\bullet_{P_m/(\Z/p^n)}$ with the crystalline cohomology of $\Z/p^m \to P_m/p$ (compatible with divided powers on $p$). Since the Frobenius action on crystalline cohomology is functorial, Frobenius also acts on the bicomplex $\Omega^\bullet_{P_\bullet/(\Z/p^n)}$, and hence on $\dR_{A/(\Z/p^n)}$. The compatibility with $\comp_{B/A}$ is clear from construction.
\end{proof}

\begin{remark}
It seems possible to use Mazur's theorem (or, rather, Ogus's generalization of it) to explicitly characterise the ``image'' of the Frobenius map defined above: it is the homotopy colimit over $m \in \Delta^\opp$ of the complexes $\L \eta \Omega^\bullet_{P_m/(\Z/p^n)}$, where $\L \eta$ denotes the cogauge used in Ogus's theorem. However, this does not seem very useful as derived de Rham cohomology tends to be unbounded outside the smooth case.
\end{remark}

\section{Some simplicial algebra}
\label{sec:ddr-scr-log}

The purpose of this section is to record some basic notions in simplicial algebra. In \S \ref{sss:modelstrsalg}, we review the usual model structures on simplicial sets, abelian groups, and commutative rings that are used in practice to defined derived functors. In \S \ref{ss:modelstrsmon}, we extend these ideas to simplicial commutative monoids.  This material will be used in \S \ref{sec:logalg} to set up some basic formalism for derived logarithmic geometry.

\subsection{Review of some standard model structures}
\label{sss:modelstrsalg}

We simply collect (with references) some of Quillen's results from \cite{QuillenHA}. All model structures we consider are {\em closed}, so we will not use this adjective. We refer the reader to the end of \S \ref{sec:notation} for our conventions concerning simplicial sets, simplicial rings, etc.

\subsubsection*{Simplicial sets and abelian groups} The category $s\Set$ is always equipped with the model structure where weak equivalences are the usual ones (defined by passage to geometric realisations), and fibrations are Kan fibrations. Similarly, we equip $s\Ab$ with the model structure where weak equivalences (resp. fibrations) are the maps which induce weak equivalences (resp. fibrations) of underlying simplicial sets. In particular, $\Forget^{s\Ab}_{s\Set}$ is a right Quillen functor with left adjoint given by $\Free^{s\Set}_{s\Ab}$. A good reference for these model structures is \cite{QuillenHA}. We follow here the convention that $(|-|,\Sing(-))$ denotes the usual adjunction between $s\Set$ and topological spaces.

\subsubsection*{Simplicial commutative rings} 
The category $\Alg$ has finite limits, all filtered colimits, and  enough projectives (given by retracts of free algebras $\Free^\Set_\Alg(X) \simeq \Z[X]$, since effective epimorphisms are just surjective maps). Hence, by Quillen's theorem \cite[Chapter 2, \S 4, Theorem 4]{QuillenHA},  we can equip the category $s\Alg$ with a model structure where fibrations (resp. weak equivalences) are those maps $A_\bullet \to B_\bullet$ such that for every projective $P \in \Alg$, the induced map $\Hom_{\Alg}(P,A_\bullet) \to \Hom_{\Alg}(P,B_\bullet)$ is a fibration (resp. weak equivalence). Note that a projective $P$ is a retract of a free algebra $\Free^\Set_\Alg(X)$ for some set $X$, and that for a set $X$, we have
\[ \Hom_{\Alg}(\Free^\Set_\Alg(X),A) \simeq \Hom_{\Set}(X,A) \simeq A^X.   \]
Thus, a fibration (resp. weak equivalence) $A_\bullet \to B_\bullet$ in $s\Alg$ is precisely a map such that for any set $X$, the map $A_\bullet^X \to B_\bullet^X$ of simplicial sets is a fibration (resp. weak equivalence). In particular, $\Free^{s\Set}_{s\Alg}$ is a left Quillen functor with right adjoint $\Forget^{s\Alg}_{s\Set}$, and similarly for the pair $(\Free^{s\Ab}_{s\Alg},\Forget^{s\Alg}_{s\Set})$. In fact, we have:

\begin{proposition}
A map $A_\bullet \to B_\bullet$ in $s\Alg$ is a fibration (resp. weak equivalence) if and only if it is so as a map of simplicial sets.
\end{proposition}
\begin{proof}
For fibrations, this follows because an arbitrary product of fibrations (in any model category) is always a fibration. For weak equivalences, note that the simplicial set underlying any object of $s\Alg$ is automatically Kan fibrant as it is a simplicial abelian group (see \cite[Chapter 2, \S 3, Corollary to Proposition 1, page 3.8]{QuillenHA}), and hence fibrant-cofibrant since all simplicial sets are cofibrant (see \cite[Chapter 2, \S 3, page 3.15, Proposition 2]{QuillenHA}). Thus, a map $A_\bullet \to B_\bullet$ in $s\Alg$ that induces a weak equivalence on underlying simplicial sets actually induces a homotopy equivalence on underlying simplicial sets. The claim now follows from the fact that homotopy equivalences are closed under arbitrary products, and the fact that $\Forget^\Alg_\Set$ commutes with products.
\end{proof}

\subsection{Model structures on simplicial commutative monoids} 
\label{ss:modelstrsmon}

Quillen's theorem used in \S \ref{sss:modelstrsalg} also leads to a model structure on $s\Mon$, and we summarise the result as:

\begin{proposition}
The category $s\Mon$ admits a model structure with a map $f:M_\bullet \to N_\bullet$ being a (trivial) fibration if and only if it is so as a map in $s\Set$.
\end{proposition}
\begin{proof}
The category $\Mon$ has finite limits, all filtered colimits, and enough projectives (given by retracts of free monoids $\Free^\Set_\Mon(X) \simeq \N^{(X)} := \oplus_{x \in X} \N \cdot x$, since effective epimorphisms are just surjective maps). By Quillen's theorem \cite[Chapter 2, \S 4, Theorem 4]{QuillenHA}, there is a model structure on $s\Mon$ with a map $f:M_\bullet \to N_\bullet$ being a (trivial) fibration if and only if the associated map $\Hom(\N^{(X)},M_\bullet) \to \Hom(\N^{(X)},N_\bullet)$ of simplicial sets is a trivial fibration for any set $X$ (as any projective is a retract of one of the form $\N^{(X)}$). By adjunction, this last map may be identified with the map $(M_\bullet)^X \to (N_\bullet)^X$. Hence, it suffices to show that a map $M_\bullet \to N_\bullet$ in $s\Set$ is a (trivial) fibration if and only if $(M_\bullet)^X \to (N_\bullet)^X$ is so for any set $X$. This follows from axiom \underline{SM7} of \cite[Chapter 2, \S 2, Definition 2]{QuillenHA} (applied with $A = \emptyset$) and \cite[Chapter 2, \S 3, Theorem 3]{QuillenHA}. 
\end{proof}

There is a forgetful functor $\Forget^{\Ab}_\Mon:\Ab \to \Mon$ which is a right adjoint with left adjoint given by $M \to M^\grp$, the group completion functor, denoted $(-)^\grp$ in the sequel. These functors interact well with the model structures:

\begin{proposition}[Olsson]
\label{prop:modelstrgrpcompl}
The adjoint pair $( (-)^\grp, \Forget^{s\Ab}_{s\Mon})$ is a Quillen adjunction. Moreover, if $P_\bullet \to M$ is an equivalence in $s\Mon$ with $M$ with discrete, then $P_\bullet^\grp \to M^\grp$ is also an equivalence.
\end{proposition}

\begin{proof}
The first part is clear as (trivial) fibrations in $s\Mon$ and $s\Ab$ are defined by passing to underlying simplicial sets; the second part is \cite[Theorem A.5]{OlssonLogCot}.
\end{proof}

Regarding a commutative ring as a commutative monoid under multiplication defines a forgetful functor $\Forget^{\Alg}_\Mon:\Alg \to \Mon$ with left adjoint $\Free^\Mon_\Alg$, and similar simplicial functors. As in the case of abelian groups, one has:

\begin{proposition}
\label{prop:modelstrsalgsmon}
The adjoint pair $(\Free^{s\Mon}_{s\Alg}, \Forget^{s\Alg}_{s\Mon})$ is a Quillen adjunction. Moreover, $\Free^{s\Mon}_{s\Alg}$ preserves all weak equivalences.
\end{proposition}
\begin{proof}
The first part is clear as (trivial) fibrations in both $s\Mon$ and $s\Alg$ are defined by passing to $s\Set$.  For the second part, note that if $M_\bullet \in s\Mon$, then 
\[ \Forget^{s\Alg}_{s\Ab} \circ \Free^{s\Mon}_{s\Alg}(M_\bullet) \simeq \Free^{s\Set}_{s\Ab} \circ \Forget^{s\Mon}_{s\Set} (M_\bullet) \simeq \Z M_\bullet, \]
i.e., the abelian group underlying the free algebra on a monoid $M$ is the same as the free abelian group on the set underlying $M$. Now if $f:M_\bullet \to N_\bullet$ is a weak equivalence in $s\Mon$, then the map $f$ is also a weak equivalence when regarded as a map of simplicial sets. Ken Brown's lemma (which ensures that a left Quillen functor preserves all weak equivalences between cofibrant objects) and the cofibrancy of all simplicial sets then show that the induced map  $\Free^{s\Set}_{s\Ab}(f) = \Z f: \Z M_\bullet \to \Z N_\bullet$ is a weak equivalence of simplicial abelian groups (and hence underlying simplicial sets). The claim now follows from the description of weak equivalences in $s\Alg$.
\end{proof}

The next few lemmas prove easy properties about simplicial commutative monoids. First, we relate a simplicial monoid to its singular complex.

\begin{lemma}
\label{lem:singmon}
For any object $M_\bullet \in s\Mon$, the singular complex $\Sing(|M_\bullet|)$ acquires the structure of a simplicial commutative monoid. The natural map $M_\bullet \to \Sing(|M_\bullet|)$ is a weak equivalence of simplicial commutative monoids.
\end{lemma}
\begin{proof}
The geometric realisation functor $|-|$ commutes with finite products of simplicial sets, so the multiplication map $M_\bullet \times M_\bullet \to M_\bullet$ defines the structure of commutative monoid on $|M_\bullet|$. The singular complex functor, by virtue of being a right adjoint, also commutes with finite products, so $\Sing(|M_\bullet|)$ becomes a simplicial commutative monoid. It is clear that the map $M_\bullet \to \Sing(|M_\bullet|)$ is a map of simplicial commutative monoids. Moreover, the map $|M_\bullet| \to |\Sing(|M_\bullet|)|$ is a weak equivalence (which is true for any simplicial set), so the last claim follows.
\end{proof}

Next, we relate a simplicial monoid to its set of connected components.

\begin{lemma}
	\label{lem:pi0mon}
Let $M_\bullet \in s\Mon$. Then $\pi_0(M_\bullet)$ (computed on the underlying simplicial set) has the natural structure of a commutative monoid. The map $M_\bullet \to \pi_0(M_\bullet)$ is the universal map from $M_\bullet$ to a simplically constant object of $s\Mon$. Moreover, $M_\bullet$ is discrete if and only if $M_\bullet \to \pi_0(M_\bullet)$ is a weak equivalence.
\end{lemma}
\begin{proof}
The multiplication map $M_\bullet \times M_\bullet \to M_\bullet$ defines the multiplication on $\pi_0(M_\bullet)$ as $\pi_0(-)$ commutes with products of simplicial sets. The universal property comes directly from that of $\pi_0$ of any simplicial set. The last claim is true for any simplicial set.
\end{proof}

Recall that an object in $s\Mon$ or $s\Alg$ is called {\em discrete} if the underlying simplicial set has a discrete geometric realization. We show next that $\Free^{s\Mon}_{s\Alg}$ preserves and reflects discreteness:

\begin{proposition}
\label{prop:mondiscrete}
An object $M_\bullet \in s\Mon$ is discrete if and only if $\Free^{s\Mon}_{s\Alg}(M_\bullet)$ is discrete. 
\end{proposition}
\begin{proof}
The forward direction follows from Proposition \ref{prop:modelstrsalgsmon} applied to the map $M_\bullet \to \pi_0(M_\bullet)$ using Lemma \ref{lem:pi0mon}. For the converse, note that $|M_\bullet|$ is a topological space with an abelian fundamental group (since $M_\bullet$ is commutative monoid, the space $|M_\bullet|$ is a commutative $H$-space) for any base point, and that the singular chain complex $\Z\Sing(|M_\bullet|)$ of $|M_\bullet|$ is equivalent to the $\Z M_\bullet \simeq \Forget^{s\Alg}_{s\Ab} \circ \Free^{s\Mon}_{s\Alg}(M_\bullet)$ by Lemma \ref{lem:singmon}. If the latter is discrete, then each connected component of $|M_\bullet|$ has no homology. Since the fundamental group is abelian, each component is therefore contractible (by Hurewicz). The claim now follows.
\end{proof}

In preparation for discussing flat morphisms of log schemes, we make the following definition:

\begin{definition}
\label{defn:monoidflat}
A map $h:M \to N$ of monoids is {\em flat} if for all maps $M \to M'$, the natural map $M' \sqcup^h_M N \to M' \sqcup_M N$ is an equivalence or, equivalently, if $M' \sqcup_M^h N$ is discrete. Here $M' \sqcup^h_M N$ is the {\em homotopy-pushout} of $M \to N$ along $M \to M'$, defined by taking a cofibrant replacement for $M \to M'$ and applying the naive pushout.
\end{definition}

The definition given above is a general definition in model category, and specialises to the case of flatness in the case of $s\Alg$, which explains the nomenclature. Our main observation is:

\begin{proposition}
\label{prop:monoidflat}
A map $h:M \to N$ in $\Mon$ is flat if $\Free^\Mon_\Alg(M) \to \Free^\Mon_\Alg(N)$ is flat in $\Alg$.
\end{proposition}
\begin{proof}
By Proposition \ref{prop:modelstrsalgsmon}, the left derived functor $\L\Free^{s\Mon}_{s\Alg}$ coincides (up to equivalence) with the naive functor $\Free^{s\Mon}_{s\Alg}$. Hence, since the former preserves homotopy-colimits, we can write
\[ \Free^{s\Mon}_{s\Alg}(N \sqcup^h_M M') \simeq \Free^\Mon_\Alg(N) \otimes_{\Free^\Mon_\Alg(M)}^L \Free^\Mon_\Alg(M'). \]
By assumption, the right hand side is discrete, and hence so is the left hand side. Proposition \ref{prop:mondiscrete} then shows that $N \sqcup^h_M M'$ is discrete, as desired.
\end{proof}

\begin{example}
\label{ex:flatmonoid}
An integral homomorphism of integral monoids is flat by \cite[Proposition 4.1]{KatoLogFI}, and can therefore be used to compute homotopy pushouts.
\end{example}

\section{The category $\Log\Alg^\pre$ of prelog rings and its model structure} 
\label{sec:logalg}

Our goal in this section is to define the basic object of logarithmic algebraic geometry: a prelog ring. We define this next, and introduce a model structure on simplicial prelog rings immediately after; this model structure will replace the usual model structure on $s\Alg$ in logarithmic version of the cotangent complex and the derived de Rham complex.

\begin{definition}
Let $\Log\Alg^\pre$ be the category of maps $\alpha:M \to A$ with $M$ a monoid, $A$ an algebra, and $\alpha$ a monoid homomorphism where $A$ is regarded as a monoid via multiplication; objects of this category are often called {\em prelog rings}. For an object $P \in \Log\Alg^\pre$, we often write $P_\Alg$ and $P_\Mon$ for the rings and monoids appearing in $P$. Given a ring $A$, we often use $A$ to denote the prelog ring $\alpha:0 \to A$.
\end{definition}

\begin{remark}
As the notation suggests, a prelog ring is a weaker version of the notion of a log ring. More precisely, a prelog ring $\alpha:M \to A$ is called a {\em log ring} if $\alpha^{-1}(A^*) \to A^*$ is an isomorphism, at least after sheafification for some topology (typically \'etale) on $\Spec(A)$. It turns out that it is much easier to develop the basic theory of the cotangent complex (see \cite[\S 8]{OlssonLogCot}) with prelog rings, so we focus on these, and only discuss genuine log rings occasionally.
\end{remark}

The association $P \mapsto (P_\Mon,P_\Alg)$ defines forgetful functors
\[ \Forget^{\Log\Alg^\pre}_{\Set \times \Set}:\Log\Alg^\pre \to \Set \times \Set \quad \textrm{and} \quad \Forget^{\Log\Alg^\pre}_{\Mon \times \Alg}:\Log\Alg^\pre \to \Mon \times \Alg.\]
Both functors $\Forget^{\Log\Alg^\pre}_{\Set \times \Set}$ and $\Forget^{\Log\Alg^\pre}_{\Mon \times \Alg}$ admit left adjoints defined by
\[ \Free^{\Set \times \Set}_{\Log\Alg^\pre}(X,Y) := (\N^{(X)} \to \Free^\Set_\Alg(X \sqcup Y) \simeq \Free^\Mon_\Alg(\N^{(X)}) \otimes_\Z \Free^\Set_\Alg(Y))\]
and
\[ \Free^{\Mon \times \Alg}_{\Log\Alg^\pre}(N,B) := (N \to \Free^\Mon_\Alg(N) \otimes_\Z B).\]
Here $\N^{(X)}$ denotes the free monoid on a set $X$, i.e., a {\em direct sum} of copies of $\N$ indexed by $X$.  Using these functors, one can construct a model structure on $s\Log\Alg^\pre$:

\begin{proposition}
\label{prop:modelstrlogalg}
The category $s\Log\Alg^\pre$ admits a simplicial model structure with (trivial) fibrations being those maps $P \to Q$ which induce a (trivial) fibration after application of $\Forget^{s\Log\Alg^\pre}_{s\Set \times s\Set}$. Under this model structure, both $\Forget^{s\Log\Alg^\pre}_{s\Set \times s\Set}$ and $\Forget^{s\Log\Alg^\pre}_{s\Mon \times s\Alg}$ are right Quillen functors with left adjoints $\Free^{s\Set \times s\Set}_{\Log\Alg^\pre}$ and $\Free^{s\Mon \times s\Alg}_{s\Log\Alg^\pre}$ respectively.
\end{proposition}
\begin{proof}
The category $\Log\Alg^\pre$ has all small limits and colimits; the formation of limits commutes with both forgetful functors mentioned above, while the formation of colimits commutes with $\Forget^{\Log\Alg^\pre}_{\Mon \times \Alg}$. Since effective epimorphisms in $\Log\Alg^\pre$ are exactly the maps which induce surjections on underlying sets, one can check (using adjunction) that the objects $\Free^{\Set \times \Set}_{\Log\Alg^\pre}(X,Y)$ are projective, and that 
\[ \Free^{\Set \times \Set}_{\Log\Alg^\pre}(1,1) \simeq \Free^{\Set \times \Set}_{\Log\Alg^\pre}( (1,\emptyset) \sqcup (\emptyset,1) )   \simeq \Free^{\Set \times \Set}_{\Log\Alg^\pre}(1,\emptyset) \sqcup \Free^{\Set \times \Set}_{\Log\Alg^\pre}(\emptyset,1) \] 
generates the category $\Log\Alg^\pre$, i.e., every object admits an effective epimorphism from a coproduct of copies of $\Free^{\Set \times \Set}_{\Log\Alg^\pre}(1,1)$. Quillen's theorem \cite[Chapter 2, \S 4, Theorem 4]{QuillenHA} then shows that $s\Log\Alg^\pre$ has a simplicial model structure with (trivial) fibrations being defined by applying $\Forget^{s\Log\Alg^\pre}_{s\Set \times s\Set}$; note that $P_\bullet \to Q_\bullet$ is a (trivial) fibration if and only if $P_{\bullet,\Alg} \to Q_{\bullet,\Alg}$ and $P_{\bullet,\Mon} \to Q_{\bullet,\Mon}$ are (trivial) fibrations in $s\Alg$ and $s\Mon$ respectively, so we have checked all claims.
\end{proof}

\begin{remark}
\label{rmk:cofiblogalg}
The cofibrations in $s\Log\Alg^\pre$ can be described explicitly as follows (see \cite[Chapter 2, \S 4, page 4.11, Remark 4]{QuillenHA}): a free cofibration is a map $(M \to A) \to P_\bullet$ with each $P_n \simeq \Free^{\Set \times \Set}_{\Log\Alg^\pre_{(M \to A)/}}(X,Y)$ for suitable sets $X$ and $Y$ with the additional property that all degeneracies are induced from $\Set \times \Set$, and a general cofibration is a retract of a free one.
\end{remark}

Proposition \ref{prop:modelstrlogalg} implies that the formation of homotopy-limits commutes with the right derived functors of the forgetful functor $\Forget^{s\Log\Alg^\pre}_{s\Mon \times s\Alg}$; it turns out that the same is true for homotopy-colimits:

\begin{proposition}
\label{prop:modelstrlogalg2}
The functor $\Forget^{s\Log\Alg^\pre}_{s\Mon \times s\Alg}$ is a left Quillen functor.
\end{proposition}
\begin{proof}
We first observe that $\Forget^{\Log\Alg^\pre}_{\Mon \times \Alg}$ is a left adjoint functor with right adjoint $\Nil^{\Mon \times \Alg}_{\Log\Alg^\pre}$ given by $(N,B) \mapsto (N \times B \to B)$. The resulting simplicial functor $\Nil^{s\Mon \times s\Alg}_{s\Log\Alg^\pre}$  preserves (trivial) fibrations since (trivial) fibrations are defined  in $s\Mon$ and $s\Alg$ by passing to $s\Set$, and similarly for $s\Log\Alg^\pre$. Hence, $\Nil^{s\Mon \times s\Alg}_{s\Log\Alg^\pre}$ is a right Quillen functor with left Quillen adjoint given  by $\Forget^{s\Log\Alg^\pre}_{s\Mon \times s\Alg}$.
\end{proof}

Next, we define the prelog avatar of the canonical free resolution:

\begin{definition}
For a map $(M \to A) \to (N \to B)$ in $\Log\Alg^\pre$, and let $P_{(M \to A)}(N \to B)$ be the simplicial object in $s\Log\Alg^\pre_{(M \to A)/}$ built using the adjunction $(\Free^{\Set \times \Set}_{\Log\Alg^\pre_{(M \to A)/}},\Forget^{\Log\Alg^\pre_{(M \to A)/}}_{\Set \times \Set})$ applied to the object $(N \to B)$; the counit defines an augmentation $P_{(M \to A)}(N \to B) \to (N \to B)$, and we call this the {\em canonical free resolution} of the $(M \to A)$-algebra $(N \to B)$. In general, any trivial fibration $P_\bullet \to (N \to B)$ with $P_\bullet$ cofibrant in $s\Log\Alg^\pre_{(M \to A)/}$ will be called a {\em projective resolution} of $(N \to B)$ as an $(M \to A)$-algebra; the same conventions apply for a morphism in an arbitrary model category.
\end{definition}

One can check that the canonical free resolution is indeed a projective resolution, and any two projective resolutions are homotopy equivalent (see \cite[Chapter 1, \S 1, Lemma 7]{QuillenHA}). Moreover, there is a tight connection between projective resolutions of prelog rings and those of the underlying monoids and algebras:

\begin{proposition}
Given a map $(M \to A) \to (N \to B)$ in $\Log\Alg^\pre$ together with a projective resolution $P_\bullet \to (N \to B)$, the maps $P_{\bullet,\Mon} \to N$ and $P_{\bullet,\Alg} \to B$ are projective resolutions in $s\Mon$ and $s\Alg$ respectively.
\end{proposition}
\begin{proof}
This follows from the fact that $\Forget^{s\Log\Alg^\pre}_{s\Mon \times s\Alg}$ preserves trivial fibrations (since it is a right Quillen functor by Proposition \ref{prop:modelstrlogalg}) and cofibrantions (since it is a left Quillen functor by Proposition \ref{prop:modelstrlogalg2}).
\end{proof}

\section{Logarithmic derived de Rham cohomology}
\label{sec:ddr-log}

Our goal in this section is to use the formalism of \S \ref{sec:logalg} to define the logarithmic version of Illusie's derived de Rham cohomology. First, we recall the key non-derived players:

\begin{definition}
\label{defn:logkahlerdiff}
Let $f:(M \to A) \to (N \to B)$ be a map in $\Log\Alg^\pre$. The $B$-module of {\em logarithmic Kahler differentials} is defined as
\[ \Omega^1_f := \Omega^1_{(N \to B)/(M \to A)} := \Big(\Omega^1_{B/A} \oplus \big(\cok(M \to N)^\grp \otimes_{\Z} B\big) \Big) / \Big((d \beta(n), 0) - (0, n \otimes \beta(n)) \Big) \]
where $\beta:N \to B$ is the structure map; see \cite[\S 1.7]{KatoLogFI}. The monoid map $d\log:N \to \Omega^1_f$ is defined by $n \mapsto d\log(n) := (0,n \otimes 1)$. The derivation $B \to \Omega^1_{B/A}$ defines by composition an $A$-linear derivation $B \to \Omega^1_f$, and we use $\Omega^\bullet_f$ to denote the corresponding complex, called the {\em logarithmic de Rham complex}.
\end{definition}

Note that $\Omega^\bullet_f$ comes equipped with a multiplication, and a descending Hodge filtration. Essentially by construction, there is a natural multiplicative filtered map $\Omega^\bullet_{B/A} \to \Omega^\bullet_f$. Moreover, an easy computation shows that $d\log(n)$ is closed, and hence $d\log$ lifts to a map $N \to \Omega^\bullet_f[1]$, also denoted $d\log$.

\begin{example}
\label{ex:freelogalg}
Let $(M \to A) \to T(X,Y) := \Free^{\Set \times \Set}_{\Log\Alg^\pre_{(M \to A)/}}(X,Y)$ be a free $(M \to A)$-algebra. Then we have
\[ \Omega^1_{T(X,Y)/(M \to A)} \simeq \Free^\Set_{\Alg_{A/}}(X \sqcup Y) \otimes_\Z (\oplus_{x \in X} \Z d\log(x) \oplus_{y \in Y} \Z dy) \]
where $d\log(x)$ and $dy$ are formal symbols; see also \cite[\S 8.4]{OlssonLogCot}. 
\end{example}

The logarithmic cotangent complex is defined by mimicking the construction of the usual cotangent complex using the canonical free resolutions in $\Log\Alg^\pre$ instead of $\Alg$. More precisely,

\begin{definition}[Gabber]
\label{defn:logcotdef}
Let $f:(M \to A) \to (N \to B)$ be a map of prelog rings, and let $P_\bullet \to (N \to B)$ be the canonical free resolution of $f$ in $s\Log\Alg^\pre$.  For each $n \in \Delta$, the prelog $(M \to A)$-algebra $P_n$ has a module $\Omega^1_{P_n/(M \to A)}$ of Kahler differentials as defined in \S \ref{defn:logkahlerdiff}, and as $n$ varies, these fit together to define a simplicial  $P_{\bullet,\Alg}$-module $\Omega^1_{P_\bullet/(M \to A)}$. The log cotangent complex of $f$ is defined to be the corresponding $B$-module, i.e., we have
\[ L_f := \Omega^1_{P_\bullet/(M \to A)} \otimes_{P_{\bullet,\Alg}} B. \]
The maps $d\log:P_{n,\Mon} \to \Omega^1_{P_n/(M \to A)}$  for each $P_n$ fit together to give a map $d\log:P_{\bullet,\Mon} \to L_f$, and hence a map $d\log:N \simeq |P_{\bullet,\Mon}| \to L_f$ in the homotopy category of $s\Mon$. 
\end{definition}

Definition \ref{defn:logcotdef} generalises in the obvious manner to all maps in $s\Log\Alg^\pre$, and the complex $L_f$ can be calculated using any projective resolution as these are all homotopy equivalent. 

\begin{remark}
Gabber's cotangent complex complex $L_f$ from Definition \ref{defn:logcotdef} is denoted $L^G_f$ in \cite[\S 8]{OlssonLogCot}. The same paper \cite{OlssonLogCot} also introduces a different version of the cotangent complex for a morphism of fine log schemes using Olsson's stack-theoretic reformulation of the logarithmic theory \cite{OlssonLogAlgStacks}. The resulting two complexes agree for integral morphisms (\cite[Corollary 8.29]{OlssonLogCot}) and always in small cohomological degrees (\cite[Theorem 8.27]{OlssonLogCot}); a key difference is that Gabber's complex is not necessarily discrete for log smooth maps, while Olsson's is. We will consistently use Gabber's complex for two reasons: (a) Gabber's theory has better functoriality properties (like the transitivity triangle \cite[Theorem 8.14]{OlssonLogCot}), (b) Gabber's theory applies to arbitrary morphisms, while Olsson's theory imposes strong finiteness conditions that will be unavailable to us.
\end{remark}

We have the following compatibility between $L_{f_\Alg}$ and $L_f$.

\begin{proposition}
\label{prop:logcotstrict}
Let $f:(M \to A) \to (N \to B)$ be a map of prelog rings. There is a natural map $L_{f_\Alg} \to L_f$ that is an isomorphism when $M = N$.
\end{proposition}
\begin{proof}
Let $P_\bullet \to (N \to B)$ be a projective $(M \to A)$-algebra resolution. The functor $\Forget^{s\Log\Alg^\pre}_{s\Mon \times s\Alg}$ preserves projective resolutions by Proposition \ref{prop:modelstrlogalg2}, so the natural map from usual Kahler differentials to the logarithmic version defines the desired map $L_{f_\Alg} \to L_f$. When $M = N$, a projective $A$-algebra resolution $Q_\bullet \to B$ defines a projective $(M \to A)$-algebra resolution $(M \to Q_\bullet) \to (M \to B)$ by Remark \ref{rmk:cofiblogalg}. One then checks directly that $\Omega^1_{Q_\bullet/A} \simeq \Omega^1_{(M \to Q_\bullet)/(M \to A)}$, proving the second claim (or one may use \cite[Lemma 8.17]{OlssonLogCot}).
\end{proof}

The usual cotangent complex can be characterised by its functor of points: for a map $f:A \to B$ of commutative rings, $\Hom(L_f,M)$ classifies $A$-linear derivations $B \to M$ for any complex $M$ of $B$-modules. The next remark gives a similar description in the logarithmic context, and was discovered in conversation with Lurie.

\begin{remark}[Lurie]
\label{rmk:logcotunivprop}
Fix a map $f:(M \to A) \to (N \to B)$ in $s\Log\Alg^\pre$. Then the construction of $L_f$ given above can be characterised by an intrinsic description of its functor of points, analogous to the picture for the usual cotangent complex, as follows. For any simplicial $B$-module $P$, there is a natural equivalence
\begin{equation}
\label{eq:univproplogcot}
\Map_{s\Mod_B}(L_f,P) \simeq  \Sect_{(M \to A)}(N \oplus P \to B \oplus P, N \to B) =:   \Der_{(M \to A)}( (N \to B), P).
\end{equation}
Let us explain briefly what this means. The term on the left is the space of maps $L_f \to P$ in the simplicial model category $s\Mod_B$ given the usual (projective) model structure; the resulting space is homotopy equivalent to $\tau_{\leq 0} \R\Hom_B(L_f,P)$. The middle term is the space of sections of the projection map 
\[ (N \oplus P \to B \oplus P) \to (N \to B)\] 
in the simplicial model category $s\Log\Alg^\pre_{(M \to A)/}$. Here $B \oplus P$ is the trivial square-zero extension of $B$  by $P$, $N \oplus P$ is the trivial square-zero extension of $N$ by $P$ in $\Mon$ with the binary operation given by $(n_1,p_1) \cdot (n_2,p_2) = (n_1 n_2,p_1 + p_2)$, and the structure morphism $N \oplus P \to B \oplus P$ is defined by 
\[ (n,p) \mapsto (\alpha(n),0) \cdot (1,p) = (\alpha(n), \alpha(n) \cdot p)\]
where $\alpha:N \to B$ is the structure map. A section of the projection map is explicitly computed by first replacing $(N \to B)$ by a cofibrant $(M \to A)$-algebra, and then producing a section over the pullback; the $B$-module structure on $\Sect_{(M \to A)}(N \oplus P \to B \oplus P, N \to P)$ is induced by that of $P$.  The universal $(M \to A)$-linear section of $(N \oplus L_f \to B \oplus L_f) \to (N \to B)$ inducing equivalence \eqref{eq:univproplogcot} is determined by the standard derivation $d:B \to L_f$ and the map $d\log:N \to L_f$; both these maps implicitly use a cofibrant replacement. When $M = N$, the equivalence \eqref{eq:univproplogcot} recovers the fact (see \cite[Proposition II.1.2.6.7]{IllusieCC1}) that the cotangent complex of a ring map classifies derivations in the derived category. Another illustrative case is when $P$ is discrete. Here we find that the space $\Map_{s\Mod_B}(L_f,P)$ is also discrete, and $\pi_0(\Map_{s\Mod_B}(L_f,P))$ can be described as the set of pairs $(\lambda,d)$ where $\lambda:N \to P$ is a monoid map that kills the image of $M \to N$, and $d:B \to P$ is an $A$-linear derivation such that $\alpha(n) \cdot \lambda(n) = d(\alpha(n))$. Note that $\lambda$ factors through $N \to N/M \to \pi_0( (N/M)^\grp \otimes_\Z B)$ since $P$ is an abelian group, admits a $B$-action, and is discrete.  Hence, this description identifies $\pi_0(L_f)$ with the sheaf $\Omega^1_f$ from \cite[\S 1.7]{KatoLogFI} or Definition \ref{defn:logcotdef}.
\end{remark}

\begin{remark}
Let $f:(M \to A) \to (N \to B)$ be a map in $\Log\Alg^\pre$. A natural question is to ask for a conceptual description of the cokernel $\calQ$ of $L_{f_\Alg} \to L_f$. Using Remark \ref{rmk:logcotunivprop}, one can interpret $\Map_{s\Mod_B}(\calQ,P)$ as the space of (monoid) maps $\lambda:N \to P$ together with nullhomotopies of the induced maps $M \to P$ and $N \stackrel{\Delta}{\to} N \times N \stackrel{\alpha,\lambda}{\to} B \times P \stackrel{\mathrm{act}}{\to} P$. We do not know if there is a better description.
\end{remark}

\begin{definition}
\label{defn:logddr}
Let $f:(M \to A) \to (N \to B)$ be a map in $\Log\Alg^\pre$, and let $P_\bullet \to (N \to B)$ be the canonical free resolution of $(N \to B)$ as an $(M \to A)$-algebra.  The {\em logarithmic derived de Rham cohomology of $f$}, denoted either $\dR_f$ or $\dR_{(N \to B)/(M \to A)}$, is the total complex associated to bicomplex $\Omega^\bullet_{P_\bullet/(M \to A)}$. The maps $P_{n,\Mon} \to \Omega^\bullet_{P_n/(M \to A)}[1]$ fit together to define a map $d\log: N \simeq |P_{\bullet,\Mon}| \to \dR_f[1]$ in the homotopy category of $s\Mon$.
\end{definition}

Elaborating on Definition \ref{defn:logddr}, the complex $\dR_f$ is naturally the simple complex associated to a simplicial $A$-cochain complex $n \mapsto \Omega^\bullet_{P_n/(M \to A)}$ for $n \in \Delta^\opp$. This definition makes it clear that $\dR_f$ is naturally an $E_\infty$-$A$-algebra equipped with a decreasing and separated Hodge filtration $\Fil^\bullet_H$. Moreover, it can be computed using any projective resolution as in the non-logarithmic case. There are also comparison maps $\Omega^\bullet_{P_{n,\Alg}/A} \to \Omega^\bullet_{P_n/(M \to A)}$ for each $n \in \Delta^\opp$ which fit together to define a natural map $\dR_{f_\Alg} \to \dR_f$. Finally, we have a conjugate filtration:
 
\begin{proposition}
\label{prop:conjfiltgenerallddr}
Let $f:(M \to A) \to (N \to B)$ be a map of prelog rings (or a map in $s\Log\Alg^\pre$). Then there exists a functorial increasing bounded below separated exhaustive filtration $\Fil^\conj_\bullet$ on $\dR_f$. This filtration can be defined using the conjugate filtration on the bicomplex $\Omega^\bullet_{P_\bullet/(M \to A)}$ for any projective $(M \to A)$-algebra resolution $P_\bullet \to (N \to B)$, and is independent of the choice of $P_\bullet$. In particular, there is a convergent spectral sequence, called the {\em conjugate spectral sequence}, of the form
\[ E_1^{p,q}:H_{p+q}(\gr^\conj_p(\dR_f)) \Rightarrow H_{p+q}(\dR_f) \]
that is functorial in $f$ (here we follow the homological convention that $d_r$ is a map $E_r^{p,q} \to E_r^{p-r,q+r-1}$).
\end{proposition}
\begin{proof}
This is proven like Proposition \ref{prop:conjfiltgeneral}.
\end{proof}

\begin{remark}
\label{rmk:lkeldr}
Let $\Log\Alg^{\pre,\Free}_{(M \to A)/}$ be the full subcategory of $\Log\Alg^\pre_{(M \to A)/}$ spanned by free prelog algebras, i.e., prelog algebras of the form $\Free^{\Set \times \Set}_{\Log\Alg^\pre_{(M \to A)/}}(X,Y)$.  Then $\Log\Alg^{\pre,\Free}$ generates $s\Log\Alg^{\pre}$ under homotopy-colimits (as an $\infty$-category). Moreover, the functor $(N \to B) \mapsto \dR_{(N \to B)/(M \to A)}$ on $s\Log\Alg^\pre_{(M \to A)/}$ is the left Kan extension of the functor $(N \to B) \mapsto \Omega^\bullet_{(N \to B)/(M \to A)}$ on $\Log\Alg^{\pre,\Free}_{(M \to A)/}$.
\end{remark}

The first basic result about logarithmic derived de Rham cohomology is that it agrees with the non-logarithmic analogue for strict maps:

\begin{proposition}
\label{prop:logddrstrict}
Let $f:(M \to A) \to (N \to B)$ be a map of prelog rings. Then the natural map $\dR_{f_\Alg} \to \dR_f$ is an isomorphism when $M = N$.
\end{proposition}
\begin{proof}
One can use the proof of Proposition \ref{prop:logcotstrict}.
\end{proof}

Next, we show some formal properties for tensor product behaviour:

\begin{proposition}
\label{prop:logddrproduct}
Let $f_i:(M \to A) \to (N_i \to B_i)$ for $i = 1,2$ be two $(M \to A)$-algebras, and let $f:(M \to A) \to (N \to B) \simeq (N_1 \to B_1) \sqcup^h_{(M \to A)} (N_1 \to B_2)$ be their homotopy cofibre product. Then the natural map defines an equivalence
\[ \dR_{f_1} \otimes^L_A \dR_{f_2} \simeq \dR_f. \]
If we use $g_1:(N_2 \to B_2) \to (N \to B)$ to denote induced map, then the natural map defines an equivalence
\[ \dR_{f_1} \otimes^L_A B_2 \simeq \dR_{g_1}.\]
\end{proposition}
\begin{proof}
Using the fact that a cofibre product of cofibrant replacements of each $f_i$ defines a cofibrant replacement for $f$, we reduce to the case that each $f_i$ is free. In this case, both claims follow from computing differential forms.
\end{proof}

In Corollary \ref{cor:ddrchar0}, we saw that derived de Rham cohomology is degenerate in characteristic $0$. The logarithmic theory is only slightly less degenerate: it sees the monoid, but misses the algebras completely. 

\begin{proposition}
\label{prop:logddrchar0}
Let $f:(M \to A) \to (N \to B)$ be a map in $\Log\Alg^\pre_{\Q/}$. Then 
\[ \gr_i^\conj(\dR_f) \simeq \wedge^i( \Cone(M^\grp \to N^\grp) \otimes_\Z A)[-i], \]
where all operations (taking  exterior powers, tensor products, and group completion) are understood to be derived.
\end{proposition}
We remark that the derived group completion agrees with the naive group completion by Proposition \ref{prop:modelstrgrpcompl}.
\begin{proof}
Let $f:(M \to A) \to T(X,Y)$ be a free map as in Example \ref{ex:freelogalg}. Then one can show that
\[ \oplus_i H^i(\Omega^\bullet_f)[-i] \simeq \wedge^i H^1(\Omega^\bullet_f) \simeq \oplus_i \wedge^i (\Z^{(X)} \otimes_\Z A)[-i] \simeq \wedge^i ((T(X,Y)_\Mon/M)^\grp \otimes_\Z A) [-i]  ,  \]
where the generators in $H^1(\dR_f) \simeq \Z^{(X)}$ correspond to $d\log(x) \in \Omega^1_f$ for $x \in X$. This computation can be carried out by reduction to the case that $X \sqcup Y $ is finite by passage to filtered colimits, then by reduction to $A = \Q$ and then $A = \C$ by base change, and then by using the logarithmic Poincare lemma to reduce to the computation of the Betti cohomology of a torus with character lattice $\Z^{(X)}$ (up to some $\A^n$ factors); we leave the details to the reader. Now in general, for any map $f:(M \to A) \to (N \to B)$,  let $(M \to A) \stackrel{a}{\to} P_\bullet \stackrel{b}{\to} (N \to B)$ be a projective resolution. Then using the preceding calculation, we find that
\[ \gr_i^\conj(\dR_f) \simeq |\wedge^i (P_{\bullet,\Mon}^\grp/M^\grp \otimes_\Z A)|[-i]. \]
By Proposition \ref{prop:modelstrgrpcompl}, the map $P_{\bullet,\Mon}^\grp \to N^\grp$ is an equivalence. Moreover, since $(-)^\grp$ is a left Quillen functor, $M^\grp \to P_{\bullet,\Mon}^\grp/M^\grp$ is a projective resolution of $N^\grp$ in $s\Ab_{M^\grp/}$; the claim now follows.
\end{proof}

Let us give an example of Proposition \ref{prop:logddrchar0}.

\begin{example}
Let $f: (\N^k \to \Q[\N^k]) \to (1 \to \Q)$ be the map on prelog rings associated to the monoid map $\N^k \to 1$; geometrically, this is the inclusion of $(1,1,\dots,1)$ in $\A^k_{\Q}$ with the log structure defined by the co-ordinate hyperplanes. Let $A = \Q[\N^k] = \Gamma(\A^k_{\Q},\calO)$. Then Proposition \ref{prop:logddrchar0} shows that
\[ \gr^\conj_i(\dR_f) \simeq \wedge^i_A (A^k[1]) [-i] \simeq \Gamma^i_A(A^k) \simeq \Sym^i_{A} (A^k). \]
In particular, $\dR_f$ is an ordinary commutative ring with an increasing separated bounded below exhaustive filtration whose associated graded coincides with the associated graded of $\Sym_{A}^\ast (A^k)$ for the degree filtration; I do not know whether $\dR_f$ itself can be identified with $\Sym^\ast_{A} (A^k)$.
\end{example}

We end this section with an example showing that log derived de Rham cohomology may change under passage to the associated log structure in characteristic $0$; this pathology will not occur in characteristic $p$, as we will see later.

\begin{example}[Non-invariance of derived log de Rham cohomology under passing to log structure]
\label{ex:lddrlogstr}
Consider the map $f:(0 \to \Q) \to (0 \to \Q[x,x^{-1}])$ of prelog $\Q$-algebras. Let $f_a:(\Q^* \to \Q) \to (\Q^* \times x^\Z \to \Q[x,x^{-1}])$ be the associated map of log structures. Then there is a natural map
\[\dR_f \to \dR_{f_a}. \]
We will show this map is not an isomorphism. To see this, note that $\dR_f$ is strict, and hence $\dR_f \simeq \Q$ as explained in Proposition \ref{prop:logddrstrict}. On the other hand, calculating $\dR_{f_a}$ using the conjugate filtration gives
\[ \gr^\conj_0(\dR_{f_a}) \simeq \Q \quad \textrm{and} \quad \gr^\conj_1(\dR_{f_a}) \simeq \cok(\Q^* \to \Q^* \times x^\Z) \otimes_\Z \Q[-1] \simeq \Q[-1],\]
while the higher ones vanish. The map $\dR_f \to \dR_{f_a}$ maps the left hand side isomorphically onto $\gr^\conj_0(\dR_{f_a})$, and completely misses $\gr^\conj_1(\dR_{f_a})$, showing that $\dR_f \to \dR_{f_a}$ is not an isomorphism.
\end{example}

\section{Logarithmic derived de Rham cohomology modulo $p^n$}
\label{sec:ddr-mod-p-log}

This section is the logarithmic analog of  \S \ref{sec:ddr-mod-p}, and depends on the theory developed in \S \ref{sec:ddr-scr-log}, \S \ref{sec:logalg}, and \S \ref{sec:ddr-log}. More precisely, we will show in \S \ref{ss:log-cartier} that derived Cartier theory works equally well in the logarithmic context; this leads in \S \ref{ss:logddrlogcrys} to a strong connection between log derived de Rham cohomology and log crystalline cohomology. As a corollary, some of the characteristic $0$ pathologies of log derived de Rham cohomology (such as Example \ref{ex:lddrlogstr}) disappear in modulo $p^n$.

\subsection{Cartier theory}
\label{ss:log-cartier}

The key player in (logarithmic) Cartier theory is the Frobenius twist:

\begin{notation}[Frobenius twists]
\label{not:frobtwistlddr}
Let $(M \to A) \in \Log\Alg^\pre_{\F_p/}$. Then we define the Frobenius map $\Frob_{(M \to A)}:(M \to A) \to (M \to A)$ as the map which is multiplication by $p$ on $M$, and the usual Frobenius on $A$.  If $f:(M \to A) \to (N \to B)$ is a map in $\Log\Alg^\pre_{\F_p/}$, then the Frobenius twist $(N \to B)^{(1)}$ is defined as the homotopy pushout of $f$ along $\Frob_{(M \to A)}:(M \to A) \to (M \to A)$. There are natural maps $f^{(1)}:(M \to A) \to (N \to B)^{(1)}$ and $\Frob_f:(N \to B)^{(1)} \to (N \to B)$ defined as in Notation \ref{not:frobtwistddr}.
\end{notation}

The interaction between Frobenius twists on prelog rings and those on the underlying rings is quite strong:

\begin{lemma}
\label{lem:frobtwistldr}
Let $f:(M \to A) \to (N \to B)$ be a map in $\Log\Alg^\pre_{\F_p/}$. Then there are base change identifications $L_f \otimes_B B^{(1)} \simeq L_{f^{(1)}}$ and $\Frob_A^* \dR_f \simeq \dR_{f^{(1)}}.$ Moreover, $f^{(1)}_\Alg$ is homotopic to $(f_\Alg)^{(1)}$. If $\Frob_A:A \to A$ and $M \stackrel{p}{\to} M$ are flat, then the homotopy pushout $f^{(1)}$ is equivalent to the ordinary pushout. If $M = N$, then $f^{(1)}$ is equivalent to the non-logarithmic pushout (equipped with the log structure defined by $M$), and similarly for $\Frob_f$.
\end{lemma}
\begin{proof}
Let $P_\bullet \to (N \to B)$ be a projective resoltuon of $(N \to B)$ as an $(M \to A)$-algebra. Then $L_f \simeq \Omega^1_{P_\bullet/(M \to A)}$ and $\dR_f \simeq \Omega^\bullet_{P_\bullet/(M \to A)}$, while $L_{f^{(1)}} \simeq \Omega^1_{P_\bullet^{(1)}/(M \to A)}$ and $\dR_{f^{(1)}} \simeq \Omega^\bullet_{P_\bullet^{(1)}/(M \to A)}$, so the first claim follows from the base change properties for $\Omega^1$. Next, note that by Propositions \ref{prop:modelstrlogalg} and \ref{prop:modelstrlogalg2}, the functor $\Forget^{s\Log\Alg^\pre}_{s\Mon \times s\Alg}$ is both a left and right Quillen functor. Hence, $P_{\bullet,\Alg} \to B$ is a projective resolution of $B$ as an $A$-algebra, and similarly for $P^{(1)}_{\bullet,\Alg} \to B^{(1)}$, which immediately implies the second claim. For the third claim, it suffices to check that the base change $(N \to B)^{(1)}$ is discrete after applying $\Forget^{s\Log\Alg^\pre}_{s\Mon \times s\Alg}$, which follows from the assumptions on $M$ and $A$ because $\Forget^{s\Log\Alg^\pre}_{s\Mon \times s\Alg}$ is left and right Quillen. The last claim can be shown as in Proposition \ref{prop:logcotstrict}.
\end{proof}

Next, in preparation for the derived version, we first recall the logarithmic Cartier isomorphism (in the free case).

\begin{theorem}[Classical logarithmic Cartier isomorphism]
\label{thm:logcartier}
Fix sets $X$ and $Y$, and let $T(X,Y) = \Free^{\Set \times \Set}_{\Log\Alg^\pre_{(M \to A)/}}(X,Y) = (M \oplus \N^{(X)} \to A[X \sqcup Y])$ be a free $(M \to A)$-algebra. Then there is a natural equivalence of graded algebras
\[ C^{-1}: \oplus_{i \geq 0} \Omega^i_{f^{(1)}} [-i] \simeq \oplus_{i \geq 0} H^i(\Omega^\bullet_f)[-i]. \]
\end{theorem}
\begin{proof}
The map $f:(M \to A) \to T(X,Y)$ is a log smooth map of Cartier type, so this follows directly from Kato's logarithmic Cartier isomorphism \cite[Theorem 4.12 (1)]{KatoLogFI}. We briefly sketch the argument.
The $T(X,Y)^{(1)}_\Alg$-linear map $C^{-1}:\Omega^1_{T(X,Y)^{(1)}/(M \to A)} \to H^1(\Omega^\bullet_f)$ is characterised by the following condition: for $x \in X$, we have
\[ C^{-1}( \Frob_{(M \to A)}^* d\log(x)) = d\log(x) \]
while for $y \in Y$, we have
\[ C^{-1}( \Frob_{(M \to A)}^* dy ) = y^{p-1} dy. \]
Here $\Frob_{(M \to A)}$ is viewed as defining (via base change) the map $T(X,Y) \to T(X,Y)^{(1)}$. In particular, the logarithmic Cartier isomorphism is compatible with the usual one. To construct $C^{-1}$, set $S(X,Y) := \Free^{\Set \times \Set}_{\Log\Alg^\pre_{\Z/p^2/}}(X,Y)$, the corresponding free object over $\Z/p^2$. Then $f$ is obtained via base change from the map $g:(0 \to \Z/p^2) \to S(X,Y)$, and so it suffices to construct the Cartier isomorphism for the reduction modulo $p$ of $g$. Now we note that $g$ comes equipped with a lift of Frobenius (given by $\cdot p$ on $S(X,Y)_\Mon$, and sending variables in $X$ and $Y$ to their $p$-th powers in $S(X,Y)_\Alg$). The rest follows as in Theorem \ref{thm:classicalcartier}.
\end{proof}

As in the non-logarithmic case, one immediate deduces the derived version:

\begin{proposition}[Derived logarithmic Cartier theory]
\label{prop:conjsslddr}
Let $f:(M \to A) \to (N \to B)$ be a map of prelog $\F_p$-algebras. Then the conjugate filtration on $\dR_f$ is $B^{(1)}$-linear, and has graded pieces computed by
\[ \Cartier_i:\gr^\conj_i(\dR_f) \simeq \wedge^i L_{f^{(1)}} [-i]. \]
In particular, the conjugate spectral sequence takes the form
\[E_1^{p,q}: H_{2p+q}(\wedge^p L_{f^{(1)}}) \simeq H_{2p + q}(\Frob_A^* \wedge^p L_f) \to H_{p+q}(\dR_f). \]
\end{proposition}
\begin{proof}
This is proven like Proposition \ref{prop:conjss} using Theorem \ref{thm:logcartier} instead of Theorem \ref{thm:classicalcartier}.
\end{proof}

We now discuss applications. First, we show that pathologies discussed in Example \ref{ex:lddrlogstr} cannot occur modulo $p$:

\begin{corollary}
\label{cor:ddrinvlogstr}
Let $f:(M \to A) \to (N \to B)$ be a map in $\Log\Alg^\pre_{\F_p/}$. Assume that both $M$ and $N$ are integral. Let $f_a:(M_a \to A) \to (N_a \to B)$ be the induced map of log structures. Then the natural map
\[ \dR_f \to \dR_{f_a} \]
is an equivalence. 
\end{corollary}
\begin{proof}
By \cite[Proposition 1.2.2]{OgusLogBook}, the category of integral monoids is closed under pushouts provided one the involved terms is a group. In particular, the log structure associated to a prelog ring with an integral monoid is also integral. By comparing conjugate filtrations, we reduce to the analogous claim for the log cotangent complex which follows from \cite[Theorem 8.16]{OlssonLogCot}.
\end{proof}

Next, we prove an analogue of Corollary \ref{cor:ddrindepsmooth}.

\begin{corollary}
\label{cor:logddrcartiertype}
Let $f:(M \to A) \to (N \to B)$ be a map in $\Log\Alg^\pre_{\F_p/}$.  Assume that $f$ is log smooth and of Cartier type. Then the natural map $\dR_f \to \Omega^\bullet_f$ is an equivalence.
\end{corollary}

The log smoothness of $f$ implies that both $M$ and $N$ admit finitely generated charts, while the Cartier type assumption means that $M \to N$ is an integral map of integral monoids (so $M \to N$ is flat by Example \ref{ex:flatmonoid}), and that $\Frob_{f,\Mon}$ is an exact morphism of monoids. 

\begin{proof}
By \cite[Corollary 4.5]{KatoLogFI}, the map $A \to B$ is flat. Since $M \to N$ is flat as well, the homotopy pushout $f^{(1)}$ coincides with the usual one. Now by \cite[Corollary 8.29]{OlssonLogCot}, the cotangent complex $L_f$ coincides with the one coming from Olsson's theory, and hence $L_f$ is discrete with $\pi_0(L_f)$ projective and naturally isomorphic to Kato's $\Omega^1_f$. Twisting by Frobenius, we find that the same holds for $L_{f^{(1)}}$, and hence $\wedge^p L_{f^{(1)}} \simeq \Omega^p_{f^{(1)}}$. The claim now follows by comparing the conjugate filtrations on either side of the map $\dR_f \to \Omega^\bullet_f$ using \cite[Theorem 4.12 (1)]{KatoLogFI} (and the fact that $(X'',M'') = (X',M')$ in the notation of {\em loc. cit.} since $f$ is of Cartier type).
\end{proof}

We can also prove a connectivity estimate:

\begin{corollary}
\label{cor:logddrconnectivity}
Let $f:(M \to A) \to (N \to B)$ be a map in $s\Log\Alg^\pre_{ {\Z/p^n}/}$ for some $n \geq 1$. Assume that $\Omega^1_f$ is generated by $r$ elements for some $r \in \Z_{\geq 0}$. Then $\dR_f$ is $(-r-1)$-connected. 
\end{corollary}
\begin{proof}
This is proven exactly like  Corollary \ref{cor:connectivity}.
\end{proof}

Next, we address transitivity in log derived de Rham theory.

\begin{proposition}
\label{prop:conjfiltlddrcomposite}
Let $(M \to A) \stackrel{f}{\to} (N \to B) \stackrel{g}{\to} (P \to C)$ be a composite of maps of prelog $\F_p$-algebras. Then $\dR_{g \circ f}$ admits an increasing bounded below separated exhaustive filtration with graded pieces of the form
\[ \dR_f \otimes_{\Frob_A^* B} \Frob_A^* \big( \wedge^n L_g [-n] \big),\]
where the second factor on the right hand side is the base change of $\wedge^n L_g[-n]$, viewed as an $B$-module, along the map $\Frob_A:B \to \Frob_A^* B$.
\end{proposition}
\begin{proof}
This is proven like Proposition \ref{prop:conjfiltddrcomposite} using Proposition \ref{prop:conjsslddr} instead of Proposition \ref{prop:conjss}.
\end{proof}

To move further, we need a definition:

\begin{definition}
A map $f$ in $\Log\Alg^\pre_{\F_p/}$ is called {\em relatively perfect} if $\Frob_f$ is an isomorphism. A map $f$ in $\Log\Alg^\pre_{\Z/p^n/}$ is called {\em relatively perfect modulo $p$} if $f \otimes_{\Z} \F_p$ is relatively perfect; similarly for maps in $\Log\Alg^\pre_{\Z_p/}$.
\end{definition}

\begin{example}
\label{ex:relperfectlogalg}
Let $f:(M \to A) \to (N \to B)$ be a map of prelog $\F_p$-algebras. Assume that $M$ and $N$ are uniquely $p$-divisible, and that $A$ and $B$ are perfect. Then $f$ is relatively perfect. Indeed, the Frobenius $\Frob_{(M \to A)}$ is an isomorphism by the assumptions on $M$ and $A$, so the derived pushout $f^{(1)}$ coincides with the underived one, and $\Frob_f$ is natually identified with $\Frob_{(N \to B)}$; the latter is an isomorphism by the assumptions $N$ and $B$.
\end{example}

The basic result concerning relatively perfect maps is an analogue of Corollary \ref{cor:ddretaleinv}.

\begin{corollary}
\label{cor:relperfectldr}
Let $f:(M \to A) \to (N \to B)$ be a relatively perfect map in $\Log\Alg^\pre_{\F_p/}$. Then $L_f \simeq 0$, and $\dR_f \simeq B$.
\end{corollary}
\begin{proof}
This is proven like Corollary  \ref{cor:ddretaleinv}. 
\end{proof}

In Question \ref{ques:etaleinv}, we asked if the vanishing of the relative cotangent complex characterises relatively perfect maps of simplicial commutative rings. In the present logarithmic context, this question has a negative answer:

\begin{example}[A non-relatively perfect map $f$ with $L_f = 0$]
\label{ex:ltrivialnotwe}
 Let $k$ be a field of characteristic $p$, and consider $f:Y := (\N^2 \to k[x,y]) \to  X := (\N_a \to k[x,y,xy^{-1},yx^{-1}])$ where the first prelog ring is the usual one, and the second one is the log structure defined by the submonoid of $\Z^2$ generated by $\N^2$ and $\pm(1,-1)$ mapping to the algebra in the obvious way; this map is the first map in the exactification of $(\N^2 \to k[\N^2]) \to (\N \to k[\N])$ defined by the sum map $\N^2 \to \N$, and therefore is log \'etale in the sense of Kato. Since $Y$ is free over $k$, we see that $L_{Y/k} \simeq \Omega^1_{Y/k}$ is a free $k[x,y]$-module of rank $2$ with generators $d\log(x)$ and $d\log(y)$. Using \cite[Lemma 8.23]{OlssonLogCot}, one can compute that $L_{X/k}$ is also free of rank $2$ on generators $d\log(x)$ and $d(xy^{-1}) = (x y^{-1}) \big(d\log(x) + d\log(y)\big)$. Since $xy^{-1} \in X_\Alg$ is a unit, one easily sees that $L_{X/k} \otimes_{X_\Alg} Y_\Alg \to L_{Y/k}$ is an isomorphism. The transitivity triangle \cite[Theorem 8.14]{OlssonLogCot} then shows that $L_f \simeq 0$. However, the map $\Frob_f$ is not an isomorphism since $\Frob_{f,\Alg}$ is not so: the map $\Frob_{f,\Alg}$ is 
\[ X^{(1)}_\Alg := k[x^{\frac{1}{p}},y^{\frac{1}{p}},xy^{-1},yx^{-1}]  \to [x^{\frac{1}{p}},y^{\frac{1}{p}},(xy^{-1})^{\frac{1}{p}},(yx^{-1})^{\frac{1}{p}}] = X_\Alg \]
as $k[x^{\frac{1}{p}},y^{\frac{1}{p}}]$-algbera map, i.e., it is a non-trivial normalisation map. Thus, $f$ is a log \'etale map of prelog rings with a vanishing cotangent complex that is not relatively perfect. 
\end{example}

Next, we present some computations that will be useful in $p$-adic applications. First, we compute the log derived de Rham cohomology of the monoid algebra of a uniquely $p$-divisible monoid:

\begin{corollary}
\label{cor:invlogddrring}
The map $f:(0 \to \Z) \to (\Q_{\geq 0} \to \Z[\Q_{\geq 0}])$ is relatively perfect modulo $p$ and 
\[ \dR_f  \otimes_{\Z} \Z/p^n \simeq \Z/p^n[\Q_{\geq 0}]. \]
\end{corollary}
\begin{proof}
The first claim implies the second by devissage and Corollary \ref{cor:relperfectldr}, and can be proven using Example \ref{ex:relperfectlogalg}.
\end{proof}

\begin{remark}
Corollary \ref{cor:invlogddrring} is completely false in characteristic $0$: $\dR_f \otimes_{\Z} \Q$ is not even discrete. In fact, using Proposition \ref{prop:logddrchar0}, one can show that $\dR_f \otimes_{\Z} \Q \simeq \dR_{(\N \to \Q[\N])/(0 \to \Q)} \simeq \Q \oplus \Q[-1]$, with the non-trivial generator in degree $1$ corresponding to $d\log(``1")$ , where $``1" \in \Q_{\geq 0}$ is the evident element.
\end{remark}

Next, we study the effect of ``adding'' a uniquely $p$-divisible monoid to a prelog ring:

\begin{corollary}
\label{cor:invlogddrmon}
The map $f:(0 \to \Z[\Q_{\geq 0}]) \to (\Q_{\geq 0} \to \Z[\Q_{\geq 0}])$ is relatively perfect modulo $p$, and 
\[\dR_f \otimes_{\Z} \Z/p^n \simeq \Z/p^n[\Q_{\geq 0}]. \]
\end{corollary}
\begin{proof}
The first claim implies the second by devissage and Corollary \ref{cor:relperfectldr}, and can proven using Example \ref{ex:relperfectlogalg}.
\end{proof}

\begin{remark}
Corollary \ref{cor:invlogddrring} and \ref{cor:invlogddrmon} admit several generalisations. For example, the monoid $\Q_{\geq 0}$ may be replaced by any uniquely $p$-divisible monoid, and the algebras $\Z$ and $\Z[\Q_{\geq 0}]$ can be replaced by any algebras that are perfect modulo $p$; we leave such matters to the reader.
\end{remark}

We end this section by recording the presence of the Gauss-Manin connection on log derived de Rham cohomology.

\begin{proposition}
	\label{prop:logcrystrans}
	Let $(M \to A) \stackrel{f}{\to} (N \to B) \stackrel{g}{\to} (P \to C)$ be a composite of maps of $\Z/p^n$-algebras. Assume that $f$ is a free map. Then the $B$-module $\dR_g$ admits a flat connection relative to $f$ that is functorial in $g$.
\end{proposition}
\begin{proof}
	Let $Q_\bullet \to (P \to C)$ be a free resolution of $g$. Then $\Omega^\bullet_{Q_n/(N \to B)}$ is naturally a complex of $B$-modules that admits a flat connection relative to $f$; a direct way to see this is to use the isomorphism of $\Omega^\bullet_{Q_n/(N \to B)}$ with $\R\Gamma_\crys( Q_n/(N \to B), \calO_\crys)$ that is functorial in $Q_n$ by \cite[Theorem 6.4]{KatoLogFI}. Taking a homotopy colimit over $n \in \Delta^\opp$ then proves the desired statement.
\end{proof}

\subsection{Comparison with log crystalline cohomology}
\label{ss:logddrlogcrys}

Our goal in this section is to prove a reasonably general comparison result between log derived de Rham cohomology and log crystalline cohomology in the sense of Kato \cite[\S 5 - 6]{KatoLogFI}; since the proof follows the same steps as that of Theorem \ref{thm:ddrcryscomp}, we only sketch steps. First, we construct the comparison map in complete generality:

\begin{proposition}
\label{prop:lddrcryscompmap}
Let $f$ be a map in $\Log\Alg^\pre_{ {\Z/p^n}/}$ for some $n \geq 1$. Then there is a natural map of Hodge-filtered $E_\infty$-algebras
\[ \comp_f:\dR_f \to \R\Gamma(f_\crys,\calO_\crys).\]
\end{proposition}
\begin{proof}
Let $f:(M \to A) \to (N \to B)$ be the map under consideration, and let $P_\bullet \stackrel{b_\bullet}{\to} (N \to B)$ be a projective resolution in $s\Log\Alg^\pre_{(M \to A)/}$. For each $n$, the composition $(M \to A) \stackrel{a_n}{\to} P_n \stackrel{b_n}{\to} (N \to B)$ is a free map $a_n$ followed by a map $b_n$ that is an effective epimorphism, i.e., both $b_{n,\Alg}$ and $b_{n,\Mon}$ are surjective. Let $P_n \to D(b_n) \to (N \to B)$ be the logarithmic pd-envelope of $b_n$ in the sense of \cite[Definition 5.4]{KatoLogFI}; this is computed by first exactifying $b_n$ in the sense of \cite[Proposition 4.10 (1)]{KatoLogFI}, and then taking the pd-envelope of the resulting strict map. Since the formation of logarithmic pd-envelopes is functorial, we obtain a natural map of bicomplexes
\begin{equation}
\label{eq:comp-ldr-lcrys-bicomplex}
 \Omega^\bullet_{P_\bullet/(M \to A)} \to \Omega^\bullet_{P_\bullet/(M \to A)} \otimes_{P_{\bullet,\Alg}} D(b_\bullet)_{\Alg}
\end{equation}
Kato's theorem \cite[Theorem 6.4]{KatoLogFI} shows that 
\[ \R\Gamma(f_\crys,\calO_\crys) \simeq \Omega^\bullet_{a_n} \otimes_{P_{n,\Alg}} D(b_n)_{\Alg} \]
for each $n$, and so the right hand side of the map \eqref{eq:comp-ldr-lcrys-bicomplex} is quasi-isomorphic to the constant simplicial object on $\R\Gamma(f_\crys,\calO_\crys)$. More precisely, the natural map
 
\[ \Omega^\bullet_{a_0} \otimes_{P_{0,\Alg}} D(b_0)_{\Alg} \to |\Omega^\bullet_{P_\bullet/(M \to A)} \otimes_{P_{\bullet,\Alg}} D(b_\bullet)_{\Alg}| \]
is an equivalence with both sides computing logarithmic crystalline cohomology. Totalising the map \eqref{eq:comp-ldr-lcrys-bicomplex} then yields the desired map
\begin{equation*}
 \dR_f := |\Omega^\bullet_{P_\bullet/(M \to A)}| \to |\Omega^\bullet_{P_\bullet/(M \to A)} \otimes_{P_{\bullet,\Alg}} D(b_\bullet)_{\Alg}| \simeq \R\Gamma(f_\crys,\calO_\crys). \qedhere
\end{equation*}
\end{proof}

\begin{remark}
\label{rmk:lkeldrcrys}
Using Remark \ref{rmk:lkeldr}, one can give a direct construction of the map $\comp_f$ as in Remark \ref{rmk:lkedrcrys}.
\end{remark}

Next, we single out the class of maps we will prove the comparison theorem for:

\begin{definition}
A map $f:(M \to A) \to (N \to B)$ of prelog $(0 \to \Z/p^n)$-algebras is called a {\em $G$-lci map} if both $A$ and $B$ are $\Z/p^n$-flat, and $f$ can be factored as $(M \to A) \stackrel{a}{\to} (P \to F) \stackrel{b}{\to} (N \to B)$ with $a$ an inductive limit of maps which are log smooth and of Cartier type modulo $p$, and $b$ a strict effective epimorphism. 
\end{definition}

\begin{example}
Let $W$ be the ring of Witt vectors of a perfect field $k$ of characteristic $p$, and let $\calO_K$ be the ring of integers in a finite extension of $\Frac(W)$, and let $\overline{\calO_K}$ be an absolute integral closure of $\calO_K$. The primary examples of $G$-lci maps we will be interested in are the modulo $p^n$ reductions of: the map $(0 \to W) \to (\calO_K - \{0\} \to \calO_K)$, the map $(0 \to W) \to (\overline{\calO_K} - \{0\} \to \overline{\calO_K})$, and the map $(\calO_K - \{0\} \to \calO_K) \to (N \to B)$ obtained from an affine patch of a semistable $\calO_K$-variety.
\end{example}

Our main theorem here is:

\begin{theorem}
\label{thm:lddrcryscomp}
Let $f:(M \to A) \to (N \to B)$ be a $G$-lci map in $\Log\Alg^\pre_{ {\Z/p^n}/}$ for some $n \geq 1$. Then the map $\comp_f$ from Proposition \ref{prop:lddrcryscompmap} is an isomorphism.
\end{theorem}
\begin{proof}[Sketch of proof]
Let $(M \to A) \stackrel{a}{\to} (P \to F) \stackrel{b}{\to} (N \to B)$ be a factorisation with $a$ a log smooth map that is of Cartier type modulo $p$ (or an inductive limit of such maps), and $b$ a strict effective epimorphism. Then by Corollary \ref{cor:logddrcartiertype} and devissage, the map $\comp_a$ is an isomorphism. The map $\comp_b$ is an isomorphism by Theorem \ref{thm:ddrcryscomp} (or simply Corollary \ref{cor:ddrcryscompquot}). These two cases can be put together as in the proof of Theorem \ref{thm:ddrcryscomp} using Corollary \ref{prop:conjfiltlddrcomposite}; we leave the details to the reader.
\end{proof}

We give an example showing that the Cartier type assumption in Theorem \ref{thm:lddrcryscomp} cannot be dropped.

\begin{example} 
\label{ex:logddrnotlogcrys}
Let $k$ be a field of characteristic $p$, and let $f:Y := (\N^2 \to k[x,y]) \to  X := (\N_a \to k[x,y,xy^{-1},yx^{-1}])$ be the map considered in Example \ref{ex:ltrivialnotwe}. Since $L_f \simeq 0$, the complex $\dR_f$ is given by the ring $X^{(1)}_\Alg$ using the conjugate filtration. The crystalline cohomology $\R\Gamma(f_\crys,\calO)$, on the other hand, is given by the ring $X_\Alg$ thanks to Kato's theorem \cite[Theorem 6.4]{KatoLogFI} as $f$ is log \'etale. Since the map $X^{(1)}_\Alg \to X_\Alg$ is not an isomorphism, we see that log derived de Rham and log crystalline cohomologies do not necessarily agree. Note that the map $f$ in this example is not an integral map, and hence not of Cartier type.
\end{example}

We end this section by showing that the Frobenius action on log crystalline cohomology always lifts to one log derived de Rham cohomology.

\begin{proposition}
\label{prop:frobactionlddr}
Let $f:(0 \to \Z/p^n) \to (M \to A)$ be a map in $\Log\Alg^\pre$. Then $\dR_{f}$ has a natural Frobenius action compatible with $\comp_f$.
\end{proposition}
\begin{proof}
This is proven just like Proposition \ref{prop:frobactionddr}.
\end{proof}

\section{The derived de Rham complex for $p$-adic algebras}
\label{sec:ddr-p-adic}

In this section, we record $p$-adic limits of the results from \S \ref{sec:ddr-mod-p} and \S \ref{sec:ddr-mod-p-log}. The basic object of interest is completed derived de Rham cohomoogy:

\begin{definition} 
\label{defn:padicddrdefn}
Let $f:A \to B$ be a map in $s\Alg_{\Z_p/}$ (or in $s\Log\Alg^\pre_{\Z_p/}$). Then the $p$-adically completed derived de Rham cohomology of $f$ is defined as $\widehat{\dR_f} := \R \lim_n \big(\dR_f \otimes_{\Z_p} \Z/p^n\big)$, where the limit is derived. We let $d\log:B_\Mon \to \widehat{\dR_f}[1]$ denote the $p$-adic limit of the maps 
\[ d\log:B_\Mon \to \dR_f \otimes_{\Z_p} \Z/p^n[1] \simeq \dR_{f \otimes_{\Z_p} \Z/p^n}[1]\]
from Definition \ref{defn:logddr}.
\end{definition}

We recall our standing convention that $\widehat{K}$ always denotes the (derived) $p$-adic completion of a complex $K$ of abelian groups. A useful observation in working with these completions is:

\begin{lemma}
\label{lem:padiccompab}
Let $K$ be a complex of abelian groups. Then $\widehat{K} \simeq \widehat{\widehat{K}}$, and $K \otimes_{\Z} \Z/p^n \simeq \widehat{K} \otimes_{\Z} \Z/p^n$ for all $n$.
\end{lemma}
\begin{proof}
It clearly suffices to show the second claim. By devissage, we may assume $n = 1$. Since $\F_p$ is represented by a perfect complex of $\Z_p$-modules, the functor $- \otimes_{\Z_p} \F_p$ commutes with arbitrary limits, so $\widehat{K} \otimes_{\Z_p} \F_p \simeq \widehat{K \otimes_{\Z_p} \F_p}$. Hence, it suffices to show that $\widehat{L} \simeq L$ for a complex $L$ of $\F_p$-vector spaces viewed as a complex of abelian groups. Using the compatible sequence of resolutions $\Big(\Z_p \stackrel{p^n}{\to} \Z_p\Big) \simeq \Z/p^n$, one easily computes that $L \otimes_{\Z_p} \Z/p^n \simeq L \oplus L[1]$, with the transition maps given by the identity on the first summand, and $0$ on the second summand. Since an $\N^\opp$-indexed limit of $0$ maps is $0$, the claim follows.
\end{proof}

Next, we record some basic formal properties of $p$-adic derived de Rham cohomology.

\begin{lemma}
\label{lem:padicddrformalprop}
Let $A \to B$, $A \to C$, and $B \to D$ be maps in $s\Alg_{\Z_p/}$. Then we have:
\begin{enumerate}
\item The natural map $\widehat{\dR_{B/A}} \to \widehat{\widehat{\dR_{B/A}}}$ is an isomorphism.
\item The natural maps induce isomorphisms: $\widehat{\dR_{B/A}} \simeq \widehat{\dR_{\widehat{B}/\widehat{A}}} \simeq \widehat{\dR_{\widehat{B}/A}}$.
\item There is a Kunneth formula: $\widehat{\dR_{B \otimes_A C/A}} \simeq \widehat{\dR_{B/A}} \widehat{\otimes}_A \widehat{\dR_{C/A}}$.
\item There is a base change formula: $\widehat{\dR_{B/A}} \widehat{\otimes}_A C \simeq \widehat{\dR_{B \otimes_A C/C}}$.
\item If $A \to B$ is relatively perfect modulo $p$,  then $\widehat{L_{B/A}} \simeq 0$, and $\widehat{\dR_{D/A}} \simeq \widehat{\dR_{D/B}}$.
\end{enumerate}
\end{lemma}

All the assertions in Lemma \ref{lem:padicddrformalprop} are easily deduced from the corresponding modulo $p^n$ statement; the details are left to the reader. We also remark that each statement in Lemma \ref{lem:padicddrformalprop} admits a logarithmic analogue as well. The main $p$-adic theorem we want is the comparison between derived de Rham theory and crystalline cohomology:

\begin{theorem}
\label{thm:padicddrcryscomp}
Let $f:(M \to A) \to (N \to B)$ be a map of prelog $\Z_p$-algebras. Assume that $A$ and $B$ are $\Z_p$-flat, and that $f$ is $G$-lci modulo $p$. Then there is a natural isomorphism
\[ \widehat{\dR_f} \simeq \R\lim_n \R\Gamma_\crys( f \otimes_{\Z_p} \Z/p^n, \calO_\crys). \]
This isomorphism is compatible with the maps $d\log:N \to \widehat{\dR_f}[1]$ and $d\log:N \to \R\Gamma_\crys(f \otimes_{\Z_p} \Z/p^n,\calO_\crys)[1]$.
\end{theorem}

One of the advantages of derived de Rham theory over crystalline cohomology is that it automatically applies to derived rings. In practice, this extra flexibility allows one to compute derived de Rham cohomology of some maps of ordinary $\Z_p$-algebras without too many flatness constraints:

\begin{proposition}
\label{prop:ddrnonflat}
Let $A$ be a $\Z_p$-flat algebra, and let $B = A/(f_1,\dots,f_r)$ with $(f_i)$ a regular sequence. Then
\[ \widehat{\dR_{B/A}} \simeq \widehat{\otimes_i} \Big( \widehat{A \langle x \rangle} \stackrel{x - f_i}{\to} \widehat{A \langle x \rangle}\Big).  \] 
\end{proposition}

We do {\em not} assume that the sequence $f_1,\dots,f_r$ is regular modulo $p$, so that $f_i = p$ for some $i$ is permissible.

\begin{proof}
We can write $B$ as the derived tensor product $ \otimes_i A/(f_i)$. Each map $A \to A/(f_i)$ can be obtained via derived base change from the map $\Z_p[x] \stackrel{x \mapsto 0}{\to} \Z_p$ along $x \mapsto f_i$. By Theorem \ref{thm:padicddrcryscomp}, we know that
\[ \widehat{\dR_{\Z_p/\Z_p[x]}} \simeq \widehat{\Z_p\langle x \rangle}. \]
The map $\Z_p[x] \to A$ defined by $x \mapsto f_i$ admits a flat resolution $\Big(A[x] \stackrel{x-f_i}{\to} A[x]\Big)$ in the category of $\Z_p[x]$-modules, where $A[x]$ is viewed as a $\Z_p[x]$-module via $x \mapsto x$. Base change  and Kunneth then show that
\[ \widehat{\dR_{B/A}} \simeq \widehat{\otimes_i} \widehat{\dR_{(A/(f_i))/A}} \simeq \widehat{\otimes_i} \Big( \widehat{A \langle x \rangle} \stackrel{x - f_i}{\to} \widehat{A \langle x \rangle}\Big),  \] 
as desired.
\end{proof}

As a corollary, we can relate the $\Z_p$-derived de Rham cohomology of an $\F_p$-algebra to geometric invariants:

\begin{corollary}
\label{cor:ddrdrwsmooth}
Let $A_0$ be an $\F_p$-algebra. If $A$ is a $p$-adically complete $\Z_p$-flat algebra $A$ lifting $A_0$, then
 \[ \widehat{\dR_{A_0/\Z_p}} \simeq \widehat{\dR_{A/\Z_p}} \oplus T,\]
where $T$ is the completion of a complex of torsion abelian groups. If $A_0$ is perfect, then $\widehat{\dR_{A_0/\Z_p}} \simeq W(A) \oplus T$ where $W(A)$ is the ring of Witt vectors of $A$, and $T$ is as before.
\end{corollary}
\begin{proof}
The second assertion follows from the first as $W(A_0)$ is $\Z_p$-flat algebra lifting lifting $A_0$ when $A_0$ is perfect (and using that $\widehat{\dR_{W(A_0)/\Z_p}} \simeq W(A_0)$ as $\Z_p \to W(A_0)$ is relatively perfect modulo $p$). To see the first assertion, note that the formula $A_0 = A \otimes_{\Z_p} \F_p$ (coupled with Kunneth) immediately show that
\[ \widehat{\dR_{A_0/\Z_p}} \simeq \widehat{\dR_{A/\Z_p}} \widehat{\otimes_{\Z_p}} \widehat{\dR_{\F_p/\Z_p}},\]
so it suffices to show the assertion for $A_0 = \F_p$. Since $p \in \Z_p$ is a regular element, Proposition \ref{prop:ddrnonflat} shows that
\[ \widehat{\dR_{\F_p/\Z_p}} \simeq \Big(\widehat{\Z_p \langle x \rangle} \stackrel{x-p}{\to} \widehat{\Z_p \langle x\rangle}\Big). \]
To compute the above complex, first observe that the decompleted object can be written as
\[ \Big( \Z_p\langle x \rangle \stackrel{x-p}{\to} \Z_p \langle x\rangle  \Big) \simeq \Z_p \oplus \big(\oplus_{j \in \Z_{> 0}} \Z_p/j\big)\]
where the summand $\Z_p/j$ on the right is defined by the image of $\gamma_j(x-p)$. Completing then gives
\[ \widehat{\dR_{\F_p/\Z_p}} \simeq \Z_p \oplus \big( \widehat{\oplus}_{j \in \Z_{> 0}} \Z_p/j \big),\]
as desired.
\end{proof}

\begin{remark} The completed direct sum appearing at the end of the proof of Corollary \ref{cor:ddrdrwsmooth} need not be torsion; for example, the element of the direct product that is $p^{n-1}$ in the $\Z_p/p^n$ summand for all $n$ (and $0$ in the other slots) is naturally a non-torsion element in the completed direct sum. Nevertheless, Corollary \ref{cor:ddrdrwsmooth} does show that when $A_0$ is perfect, the ring $W(A_0)$ may be obtained as the largest separated torsion-free quotient of $\widehat{\dR_{A_0/\Z_p}}$. 
\end{remark}

\begin{remark}
The idea of using the Frobenius action on the cotangent complex of an $\F_p$-algebra to produce liftings to characteristic $0$ is not new. It occurs in \cite[\S 1.2]{deJongCrystallineFormal} and, more recently, in Scholze's work \cite{ScholzePerfectoid1}.  The resulting interpretation of $W(A_0)$ as a formal deformation of $A_0$ is quite useful in practice. For example, for $A_0$ perfect, the Teichmuller lift $[\cdot]:A_0 \to W(A_0)$ arises by repeatedly applying the following simple observation to $N = A_0$ and $M$ the multiplicative monoid underlying an infinitesimal thickening of $A_0$: if $V$ is an abelian $p^n$-torsion group, $N$ is a uniquely $p$-divisible commutative monoid, and $\pi:M \to N$ is a surjection of commutative monoids with kernel $V$, then there is a unique section of $\pi$ (as the multiplication by $p^n$ map on $M$ factors through $\pi$).
\end{remark}

\begin{warning}
It is tempting to guess that the derived de Rham cohomology of $\Z/p^n \to \F_p$ is simply $\Z/p^n$. However, this is false for $n \geq 2$. If it were true, then the derived de Rham cohomology of $\F_p \to \F_p \otimes_{\Z/p^n} \F_p =: R$ would be $\F_p$ by base change. Now $S = \Sym_{\F_p}(\F_p[1])$ is a retract of $R$ (the map $R \to S$ is a Postnikov trunctation, while the section $S \to R$ is defined by choosing a generator of $\pi_1(R) \simeq \F_p$), so $\dR_{S/\F_p} \simeq \F_p$ as well. However, this is a contradiction as $\dR_{S/\F_p} \simeq \F_p \langle x \rangle \otimes_{\F_p[x]} \F_p$ (via pushout from $\dR_{\F_p/\F_p[x]} \simeq \F_p \langle x \rangle$  along $\F_p[x] \stackrel{x \mapsto 0}{\to} \F_p$) has infinite dimensional $\pi_0$ and $\pi_1$. In fact, a simplicial enhancement of Proposition \ref{prop:derivedcartier} shows $\dR_{S/\F_p} \simeq S \langle x \rangle$.
\end{warning}

\section{Period rings via derived de Rham cohomology}
\label{sec:ddr-period-rings}

In this section, we give derived de Rham interpretations for various period rings (with their finer structure) that occur in the $p$-adic comparison theorems. We begin with notation that will be used through the rest of this paper.

\begin{notation}
\label{not:periodring}
Let $k$ be a perfect field of characteristic $p$ with ring of Witt vectors $W$. Let $K_0 = \Frac(W)$, and fix a finite extension $K/K_0$ of degree $e$ with ring of integers $\calO_K$, and a uniformiser $\pi$ with minimal (Eisenstein) monic polynomial $E(x) \in W[x]$. We fix an algebraic closure $\overline{K}$ of $K$, which gives us access to the absolute integral closure $\overline{\calO_K}$ of $\calO_K$,  its $p$-adic completion $\widehat{\overline{\calO_K}}$, and the Galois group $G_K$. For an $\F_p$-algebra $R$, let $R^\perf$ and $R_\perf$ to denote the $\lim$ and $\colim$ perfections of $R$, respectively. We follow the convention that $(R,M)$ refers to a prelog ring where $R$ is a ring, $\alpha:M \to R$ is a prelog structure; when $M = \N$ (resp. $\Q_{\geq 0}$) with $\alpha(1) = f$ (resp. with $\alpha(1) = f$ for an element $f$ with specified rational powers), then we also write $(R,f)$ (resp. $(R,f^{\Q_{\geq 0}})$). For an $\calO_K$-algebra $A$, we let $(A,\can)$ denote the log ring defined by the open subset $\Spec(A[1/p]) \subset \Spec(A)$  (unless otherwise specified). 
\end{notation}

We start by recalling a construction of Fontaine that lies at the heart of the theory of period rings.

\begin{construction}
\label{cons:ainf}
We define $A_\inf = W( (\overline{\calO_K}/p)^\perf)$. Given a sequence $\{r_n \in \overline{\calO_K}\}$ of $p$-power compatible roots (i.e., $r_n^p = r_{n-1}$), we use $[\underline{r}] \in A_\inf$ to denote the Teichmuller lift of the evident element $\underline{r} = \lim_n r_n$ of $(\overline{\calO_K}/p)^\perf$. By functoriality, there is a $G_K$-action on $A_\inf$.
\end{construction} 

Construction \ref{cons:ainf} interacts extremely well with de Rham theory; the highlights of this interaction are:

\begin{proposition} With notation as above, one has:
\label{prop:ainfprop}
\begin{enumerate}
\item The ring $(\overline{\calO_K}/p)^\perf$ is a perfect rank $1$ complete valuation ring.
\item The cotangent complex $\widehat{L_{A_\inf/W}}$ vanishes.
\item There exists a unique $G_K$-equivariant ring homomorphism $\theta:A_\inf = W( (\overline{\calO_K}/p)^\perf ) \to \widehat{\overline{\calO_K}}$ that modulo $p$ reduces to the defining map $(\overline{\calO_K}/p)^\perf \to \overline{\calO_K}/p \simeq \widehat{\overline{\calO_K}}/p$. This map $\theta$ is surjective and satisfies $\theta([\underline{r}]) = r_0$ for any $p$-power compatible sequence $\{r_n \in \widehat{\overline{\calO_K}}\}$. 
\item The kernel $\ker(\theta)$ is principal and generated by a regular element. If a compatible sequence $\{\pi^{\frac{1}{p^n}} \in \overline{\calO_K}\}$ of $p$-power roots of $\pi$ has been chosen, then $E([\underline{\pi}])$ is a generator for $\ker(\theta)$.
\item The transitivity triangle for $W \to A_\inf \to \widehat{\overline{\calO_K}}$ induces an isomorphism 
\[ \psi:\ker(\theta)/\ker(\theta)^2 \simeq \widehat{L_{\widehat{\overline{\calO_K}}/W}}[-1] \simeq \widehat{\Omega^1_{\overline{\calO_K}/W}}[-1].\] 

\end{enumerate}
\end{proposition}
\begin{proof}
These results (due to Fontaine) are well-known, but we sketch arguments to show that they are easy to prove.
\begin{enumerate}
\item The perfectness is clear. An elementary argument \cite[\S 1.2.2]{Fontainepadicperiod} shows that there is a multiplicative bijection of sets $(\overline{\calO_K}/p)^\perf \simeq \lim_{x \mapsto x^p} \widehat{\overline{\calO_K}}$ defined by the obvious map from the right to the left. This allows one to define a rank $1$ semi-valuation on $(\overline{\calO_K}/p)^\perf$ via the valuation on the first component of the inverse limit on the right. One checks directly that this semi-valuation has no kernel, so it defines a rank $1$ valuation; the completeness is automatic as the displayed inverse limit has continuous transition maps and complete terms.
\item This follows from the vanishing of $\widehat{L_{A_\inf/W}} \otimes_W W/p \simeq L_{ (\overline{\calO_K}/p)^\perf/(W/p) }$ which follows from perfectness.
\item This follows directly from the cotangent complex vanishing. Indeed, as the rings $W$, $A_\inf$, and $\widehat{\overline{\calO_K}}$ are all $p$-adically complete, the surjective map $A_\inf/p \to \overline{\calO_K}/p$ admits a unique lift to a surjective map $\theta_n:A_\inf/p^n \to \overline{\calO_K}/p^n$ by the vanishing of $L_{(A_\inf/p^n)/(W/p^n)}$, so $\theta = \lim_n \theta_n$ does the job.
\item As the source and target of $\theta$ are $W$-flat and $p$-adically complete, it suffices to show that $\ker(\theta)$ is principal and generated by a regular element modulo $p$. By (1), the kernel of $A_\inf/p = (\overline{\calO_K}/p)^\perf \to \overline{\calO_K}/p$ is the set of elements with valuation $\geq 1$, which is certainly a principal non-zero ideal and can be generated by any element of valuation exactly $1$.  The same reasoning shows that any element of $\ker(\theta)$ whose reduction modulo $p$ has valuation $1$ is a generator. One has $\theta(E([\underline{\pi}])) = E(\pi) = 0$, so $E([\underline{\pi}]) \in \ker(\theta)$. On the other hand, $E([\underline{\pi}]) = [\underline{\pi}]^e \mod p$, which has valuation $e \cdot \val_p(\pi) = 1$, so the claim follows.
\item This follows from (2), (4), and the isomorphism $\widehat{L_{\widehat{\overline{\calO_K}}/W}} \simeq \widehat{L_{\overline{\calO_K}/W}} \simeq \widehat{\Omega^1_{\overline{\calO_K}/W}}$, where the last one comes from the ind-lci nature of $W \to \overline{\calO_K}$, see Proposition \ref{prop:diffformszpbar}; an explicit formula is given in Remark \ref{rmk:explicitformd}.\qedhere
\end{enumerate}
\end{proof}

\begin{remark}
	\label{rmk:explicitformd}
	Continuing the discussion of Proposition \ref{prop:ainfprop}, one can describe the isomorphism $\psi$ from (5) explicitly. The transitivity triangle for $W \to A_\inf \to \widehat{\overline{\calO_K}}$ is an exact triangle
	\[ \ker(\theta)/\ker(\theta)^2 \to L_{A_\inf/W} \to L_{\widehat{\overline{\calO_K}}/W}. \]
	Multiplication by $p^n$ is injective on the left term (as it is a free rank $1$ $\widehat{\overline{\calO_K}}$-module), a quasi-isomorphism on the middle term (by (2)), and surjective on $\pi_0$ of the right term (since one can extract $p^n$-th roots in $\widehat{\overline{\calO_K}}$). Hence, the cone of multiplication by $p^n$ on this exact triangle gives a coboundary isomorphism.
	\[ \psi_n:\ker(\theta)/(\ker(\theta^2),p^n) \simeq L_{\widehat{\overline{\calO_K}}/W} \otimes_{\Z} \Z/p^n[-1] \simeq \Omega^1_{\overline{\calO_K}/W}[p^n],\]
	which is simply the reduction modulo $p^n$ of the map $\psi$ from (5). Unwrapping definitions, this map is given by
	\[ \overline{f} \mapsto q^{-1}(\theta_*(\frac{1}{p^n}(df))).\]
	Here $f \in \ker(\theta)$ is a lift of $\overline{f} \in \ker(\theta)/(\ker(\theta)^2,p^n)$; its derivative $df$ is viewed as a map $\Z \stackrel{f}{\to} \ker(\theta) \stackrel{d}{\to}  L_{A_\inf/W}$. The map $\frac{1}{p^n}(df):\Z \to L_{A_\inf/W}$ is the composition of $df$ with the inverse of $p^n:L_{A_\inf/W} \to L_{A_\inf/W}$ (which is a quasi-isomorphism by (2)). The map $\theta_*:L_{A_\inf/W} \to L_{\widehat{\overline{\calO_K}}/W}$ is the usual map. Since $df$ comes from $\ker(\theta)$, one has a specified null homotopy of $p^n \cdot \theta_*(\frac{1}{p^n}(df))$, so $\theta_*(\frac{1}{p^n}(df))$ can be viewed as map $\Z \to L_{\widehat{\overline{\calO_K}}/W} \otimes_{\Z} \Z/p^n[-1]$. Finally, $q^{-1}$ is the quasi-isomorphism $L_{\widehat{\overline{\calO_K}}/W} \otimes_{\Z} \Z/p^n[-1] \simeq L_{\overline{\calO_K}/W} \otimes_{\Z} \Z/p^n[-1] \simeq \Omega^1_{\overline{\calO_K}/W}[p^n]$. For example, given a compatible sequence $\{p^{\frac{1}{p^n}} \in \overline{\calO_K} \}$ of $p$-power roots of $p$, one has 
\[ \psi_n([\underline{p}] - p)  = \big(p^\frac{1}{p^n}\big)^{p^n - 1} \cdot d(p^{\frac{1}{p^n}}) \in \Omega^1_{\overline{\calO_K}/W}[p^n].\]
	Essentially the same computation also shows
\[ \psi_n([\underline{\epsilon}]) = \epsilon_n^{p^n - 1} d(\epsilon_n) = \epsilon_n^{-1} d(\epsilon_n) =: d\log(\epsilon_n) \in \Omega^1_{\overline{\calO_K}/W}[p^n]\]
where $\{\epsilon_n \in \overline{\calO_K} \}$ is a compatible system of $p$-power roots of $1$.

\end{remark}

\begin{remark}
In the notation of Proposition \ref{prop:ainfprop}, it is also true that $(\overline{\calO_K}/p)^\perf$ has an algebraically closed fraction field; we do not prove that here as we do not need it.
\end{remark}

\begin{remark}
\label{rmk:ainfperfectoid}
Let $A$ be an integral perfectoid $\widehat{\overline{\calO_K}}$-algebra in the sense of Scholze \cite{ScholzePerfectoid1}, i.e., $A$ is a $p$-adically complete flat $\widehat{\overline{\calO_K}}$-algebra such that $\Frob:A/(p^{\frac{1}{p}}) \to A/p$ is an isomorphism (one can also replace $\widehat{\overline{\calO_K}}$ by any other pefectoid base ring). Most of Proposition \ref{prop:ainfprop} generalises effortlessly when we replace $\widehat{\overline{\calO_K}}$ with $A$. In fact, the map $\widehat{\overline{\calO_K}} \to A$ is relatively perfect modulo $p$ by definition, so the results for $\widehat{\overline{\calO_K}}$ imply those for $A$ by deformation theory. In particular, there exists a unique ($p$-adic formal) deformation $A_\inf(A)$ of $A$ along $A_\inf \to \widehat{\overline{\calO_K}}$; moreover, $A_\inf(A)$ is perfect modulo $p$ (as it is relatively perfect over $A_\inf$;  in fact, one has $A_\inf(A) = W((A/p)^\perf)$), and the structure map $\theta_A:A_\inf(A) \to A$ has kernel $\ker(\theta_A) = \ker(\theta_{\widehat{\overline{\calO_K}}}) \otimes_{A_\inf} A_\inf(A)$. 
\end{remark}

Next, we introduce the period ring $A_\crys$ by a derived de Rham definition.

\begin{definition}
\label{defn:acrys}
The ring $A_\crys$ of crystalline periods is defined as $\widehat{\dR_{\overline{\calO_K}/W}}$.
\end{definition}

\begin{remark}
The ring $A_\crys$ comes equipped with a Hodge filtration and a Frobenius action (by Proposition \ref{prop:frobactionddr}).
\end{remark}

We show next that the preceding definition of $A_\crys$ coincides with the classical one:

\begin{proposition}
\label{prop:acrysdefn}
The ring $A_\crys$ can be identified with the $p$-adic completion of the pd-envelope $D_{A_\inf}(\ker(\theta))$. Under this isomorphism, the Hodge filtration on $A_\crys$ corresponds to the filtration by divided-powers of $\ker(\theta)$.
\end{proposition}
\begin{proof}
By Lemma \ref{lem:padicddrformalprop}, we have $A_\crys \simeq \widehat{\dR_{\widehat{\overline{\calO_K}}/W}}$. Now the map $W \to \widehat{\overline{\calO_K}}$ factors as a composite $W \stackrel{a}{\to} A_\inf \stackrel{b}{\to} \widehat{\overline{\calO_K}}$. The map $a$ is relatively perfect modulo $p$ since $W/p$ and $A_\inf/p$ are perfect.  The map $b$ is a quotient by the regular element by Proposition \ref{prop:ainfprop}. Hence, by Lemma \ref{lem:padicddrformalprop} (5) and Theorem \ref{thm:padicddrcryscomp} (or simply Corollary \ref{cor:ddrcryscompquot}), we have
\[ \widehat{\dR_{\overline{\calO_K}/W}} \simeq \widehat{\dR_{\widehat{\overline{\calO_K}}/W}} \simeq \widehat{\dR_{\widehat{\overline{\calO_K}}/A_\inf}} \simeq \widehat{D_{A_\inf}(\ker(\theta))}. \]
The assertion about the Hodge filtration is immedate.
\end{proof}

\begin{remark}
\label{rmk:acrysperfectoid}
Continuing Remark \ref{rmk:ainfperfectoid}, Proposition \ref{prop:acrysdefn} generalises directly to the case where $\widehat{\overline{\calO_K}}$ is replaced by any integral perfectoid $\widehat{\overline{\calO_K}}$-algebra $A$, i.e.,  $A_\crys(A) := \widehat{\dR_{A/\Z_p}}$ can be identified with the $p$-adic completion of $D_{A_\inf(A)}(\ker(\theta_A))$.  This observation can be used to define a comparison map using \cite{ScholzePerfectoid1}.
\end{remark}

The ring $A_\crys$ is also natural from the point of view of {\em log} derived de Rham cohomology. In fact, addition of the (uniquely divisible) canonical log structure to $\overline{\calO_K}$ does nothing at all to de Rham cohomology:

\begin{proposition}
\label{prop:ddrlogddrzpbar}
Let $(\overline{\calO_K},\can)$ denote the ring $\overline{\calO_K}$ endowed with the log structure $\overline{\calO_K} - \{0\} \to \overline{\calO_K}$. Then we have
\[ \widehat{L_{\overline{\calO_K}/W}} \simeq \widehat{L_{(\overline{\calO_K},\can)/W}} \quad \textrm{and} \quad A_\crys \simeq \widehat{\dR_{(\overline{\calO_K},\can)/W}}.\]
\end{proposition}
\begin{proof}
Using the natural map, devissage, and the conjugate filtration modulo $p$, it suffices to prove the assertion about cotangent complexes. Also, we may pass to $p$-adic completions of rings using Lemma \ref{lem:padicddrformalprop}. We fix once and for all a collection $\{\pi^{\frac{a}{b}} \in \overline{\calO_K} \}$ of power-compatible positive rational powers of $\pi$; this choice allows us to define compatible powers $[\underline{\pi}]^{\frac{a}{b}} \in A_\inf$ for any $\frac{a}{b} \in \Q_{\geq 0}$, and hence defines a commutative diagram
\[ \xymatrix{ W \ar[r]^-a &  A_\inf \ar[r]^-b \ar[d]^-{c = \theta} & (A_\inf,[\underline{\pi}]^{\Q_{\geq 0}}) \ar[d]^-d \\
					  & \widehat{\overline{\calO_K}} \ar[r]^-e &  (\widehat{\overline{\calO_K}},\can)}, \]
with the square on the right being a pushout, up to passage from prelog structures to log structures. Here $(\widehat{\overline{\calO_K}},\can)$ denotes the prelog ring $\overline{\calO_K} - \{0\} \to \widehat{\overline{\calO_K}}$. Since the map $a$ is relatively perfect modulo $p$, it suffices show that $\widehat{L_c} \to \widehat{L_{d \circ b}}$ is an equivalence. By the Kunneth formula for the square, it suffices to show that $\widehat{L_b} \simeq 0$. This follows from Corollary \ref{cor:invlogddrmon} by base change along the flat map $\Z[\Q_{\geq 0}] \to A_\inf$ defined by $t^{\frac{a}{b}} \to [\underline{\pi}]^{\frac{a}{b}}$, where $t = ``1" \in \Q_{\geq 0}$ is the co-ordinate on $\Z[\Q_{\geq 0}]$.
\end{proof}

\begin{remark}
The proof of Proposition \ref{prop:ddrlogddrzpbar} ``cheated'' by using that $\overline{\calO_K} - \{0\} \stackrel{\val}{\to} \Q_{\geq 0}$ has a section (given by the choice $\{\pi^{\frac{a}{b}}\}$ of roots of $\pi$). A better proof would go through the following statement (which can be shown): if $A \subset M$ is an inclusion of integral commutative monoids with $A$ a group and $M/A$ a uniquely $p$-divisible monoid, then  
\[ L_{(M \to R)/(A \to R)} \simeq 0.\]
for any prelog $\F_p$-algebra $(M \to R)$.
\end{remark}

Next, we want to study some finer structures on the period ring $A_\crys$. For this, we briefly recall the structure of the cotangent complex of $W \to \overline{\calO_K}$, discovered by Fontaine; our exposition follows that of Beilinson \cite[\S 1.3]{Beilinsonpadic}.

\begin{proposition}[{Fontaine \cite[Theorem 1 (ii)]{FontaineDiffFormAbVar}}]

\label{prop:diffformszpbar}
The map $W \to \overline{\calO_K}$ has a discrete cotangent complex, i.e., $L_{\overline{\calO_K}/W} \simeq \Omega^1_{\overline{\calO_K}/W}$. Moreover, the map $\mu_{p^\infty} \to L_{\overline{\calO_K}/W}$ defined by $\zeta \mapsto d\log{\zeta}$ induces an exact sequence
\begin{equation}
\label{eq:omegaonezpbar}
1 \to \big(\fraa/\overline{\calO_K})(1) \to \big(\overline{K}/\overline{\calO_K}\big)(1) \simeq \mu_{p^\infty} \otimes_{\Z_p} \overline{\calO_K} \stackrel{d\log}{\to} \Omega^1_{\overline{\calO_K}/W} \to 1,
\end{equation}
where $\fraa \subset \overline{\calO_K}$ is the fractional ideal comprising all elements of valuation $\geq -\frac{1}{p-1}$ (so $\fraa = \overline{\calO_K} \cdot p^{-\frac{1}{p-1}} \subset \overline{K}$).
\end{proposition}

All tensor products appearing below take place over $\Z_p$ unless otherwise specified. The following fact will be used implicitly: if $L/K$ is a finite extension, then $L_{\calO_L/W} \simeq \Omega^1_{\calO_L/\calO_K}$ is a cyclic torsion $\calO_K$-module.

\begin{proof}
The transitivity triangle shows that for any extension $K_0 \to L \to \overline{K}$, the map $\Omega^1_{\calO_L/W} \otimes_{\calO_L} \overline{\calO_K} \to \Omega^1_{\overline{\calO_K}/W}$ is injective, and the filtered colimit over these maps as $L$ varies spans the target. Since $L_{\calO_L/W} \simeq \Omega^1_{\calO_L/W}$, it follows that the same is true in the limit, proving the first assertion. For the second claim, one first observes that $\ker(d\log) \subset \overline{\calO_K} \otimes \mu_p \subset \overline{\calO_K} \otimes \mu_{p^\infty}$ as the set of all $\overline{\calO_K}$-submodules of $\overline{\calO_K} \otimes \mu_{p^\infty}$ is totally ordered under inclusion (and because $d\log(\overline{\calO_K} \otimes \mu_p) \neq 0$). This gives a commutative diagram
\[ \xymatrix{ W[\mu_p] \otimes \mu_p \ar[r]^-{a} \ar[d]^-{\can}& \Omega^1_{W[\mu_p]/W} \ar[d]^-{\can} \\
			  \overline{\calO_K} \otimes \mu_p \ar[r]^-{b} \ar@{=}[d] & \overline{\calO_K} \otimes_{W[\mu_p]} \Omega^1_{W[\mu_p]/W} \ar[d]^-{c} \\
			  \overline{\calO_K} \otimes \mu_p \ar[r]^{d} & \Omega^1_{\overline{\calO_K}/W}, } \]
where the first square is a flat base change along $W[\mu_p] \to \overline{\calO_K}$. Since $c$ is injective, it follows that $\ker(d) = \ker(b) = \ker(a) \otimes_{W[\mu_p]} \overline{\calO_K}$. If $\zeta \in \mu_p$ denotes a fixed primitive $p$-th root of $1$, then $\ker(a) = \Ann(d\log(\zeta)) \otimes \mu_p \subset W[\mu_p] \otimes \mu_p$. Now $\Ann(d\log(\zeta))$ has valuation $1 - \frac{1}{p-1}$ (by computing the derivative of $1 + X + \dots + X^{p-1}$ evaluated at $\zeta$, for example), and this implies the claim about $\ker(d\log)$. For surjectivity of $d\log$, it suffices to show that for any finite extension $L/K_0$, one has $\Omega^1_{\calO_L/W} \subset \overline{\calO_K} \cdot d\log(\mu_{p^n}) \subset \Omega^1_{\overline{\calO_K}/W}$ for some large $n$. If $p^d$ kills $\Omega^1_{\calO_L/W}$, then $\Omega^1_{\calO_L/W} \subset \Omega^1_{\calO_{L[\mu_{p^n}]}/W}$ generates a submodule killed by $p^d$, for any $n$. The set of all submodules of $\Omega^1_{\calO_{L[\mu_{p^n}]}/W}$ is totally ordered under inclusion, and it is clear that $\calO_{L[\mu_{p^n}]} \cdot d\log(\mu_{p^n}) \subset \Omega^1_{\calO_{L[\mu_{p^n}]}/W}$ is a submodule not killed by $p^d$, for $n$ sufficiently large: we simply need $L[\mu_{p^n}]$ to have a different (relative to $K_0$) with valuation $> d$. It follows that $\Omega^1_{\calO_L/W} \subset \calO_{L[\mu_{p^n}]} \cdot d\log(\mu_{p^n}) \subset \Omega^1_{\calO_{L[\mu_{p^n}]}/W}$ as desired.
\end{proof}

\begin{remark}
An alternative argument for the surjectivity of $d\log$ from Proposition \ref{prop:diffformszpbar} runs as follows: the map $f:W[\mu_{p^\infty}] \to \overline{\calO_K}$ is relatively perfect modulo $p$, so $\widehat{L_f} \simeq 0$ by Lemma \ref{lem:padicddrformalprop} (5). By the transitivity triangle, the map 
\[ f_*:\Omega^1_{W[\mu_{p^\infty}]/W} \otimes_{W[\mu_{p^\infty}]} \overline{\calO_K} \to \Omega^1_{\overline{\calO_K}/W}\]
induces an isomorphism after $p$-adic completion. Then $f_*$ is also an isomorphism as $p$-adic completion is conservative on $p$-torsion $p$-divisible $W$-modules; surjectivity of $d\log$ now follows from the analogous claim for $W[\mu_{p^\infty}]$.
\end{remark}

Our next goal is to use derived de Rham formalism to construct a $G_K$-equivariant map $\Z_p(1) \to A_\crys$, and show that this coincides with a map defined by Fontaine. We first construct the map:

\begin{construction}
\label{cons:acryschern}
The $d\log$ maps in logarithmic derived de Rham cohomology define maps
\[ d\log: \mu_{p^\infty} \subset \overline{\calO_K}^* \to \dR_{(\overline{\calO_K},\can)/W}[1], \]
where $(\overline{\calO_K},\can)$ denotes the prelog ring from Proposition \ref{prop:ddrlogddrzpbar}. Taking $p$-adic completions gives a map
\[ \widehat{d\log}:\Z_p(1)[1] \to \widehat{\dR_{(\overline{\calO_K},\can)/W}}[1] \simeq A_\crys[1].\]
Applying $\pi_1$, we obtain
\[ \beta := \pi_1(\widehat{d\log}):\Z_p(1) \to A_\crys.\]
This map is $G_K$-equivariant, and has image contained in $\Fil^1_H(A_\crys)$.
\end{construction}

The map $\beta$ defined in Construction \ref{cons:acryschern} coincides with maps defined by Fontaine:

\begin{proposition}
\label{prop:acrysbeta}
Let $(\epsilon_n) \in \Z_p(1)$ denote a typical element.
\begin{enumerate}
\item The element $[\epsilon] - 1 \in A_\crys$ lies in $\Fil^1_H(A_\crys) = \ker(\theta)$, hence $\log([\epsilon]) \in A_\crys$ makes sense.
\item The image of $[\epsilon]-1$ in $\gr^1_H(A_\crys) \simeq \widehat{L_{\overline{\calO_K}/W}}[-1]$ has positive valuation.
\item The map $\beta:\Z_p(1) \to A_\crys$ coincides with Fontaine's  map $(\epsilon_n) \mapsto \log([\underline{\epsilon}])$.
\end{enumerate}
\end{proposition}

\begin{proof}[Proof sketch]
We follow the notation of the proof of Proposition \ref{prop:ainfprop}. 
\begin{enumerate}
\item This is clear because $\theta([\underline{r}]) = r_0$ for any $p$-power compatible system of elements $r_n \in \widehat{\overline{\calO_K}}$.
\item We may assume that $\epsilon_1$ is a primitive $p$-th root of $1$, so $(\epsilon_n) \in \Z_p(1)$ is a generator.  It suffices to check that $[\underline{\epsilon}] - 1$ does not generate the kernel of $A_\inf/p \to \overline{\calO_K}/p$, i.e., that $\frac{[\underline{\epsilon}] - 1}{[\underline{p}]} \in A_\inf/p$ has positive valuation. Twisting by Frobenius, it suffices to show that $\val_p(\epsilon_1 - 1) > \val_p(p^{\frac{1}{p}})$, but this is clear: $\val_p(\epsilon_1 - 1) = \frac{1}{p-1}$, and $\val_p(p^{\frac{1}{p}}) = \frac{1}{p}$. 
\item Let $\beta':\Z_p(1) \to A_\crys$ denote the map $(\epsilon_n) \mapsto \log([\epsilon])$, which makes sense by (1). It is clear that this map is $G_K$-equivariant, and has image contained in $\Fil^1_H(A_\crys)$.  To show $\beta = \beta'$, assume first that the induced maps
\[ \gr^1_H(\beta),\gr^1_H(\beta'):\Z_p(1) \to \gr^1_H(A_\crys) \simeq \widehat{L_{\overline{\calO_K}/W}}[-1]\]
are equal, where the last isomorphism comes from Proposition \ref{prop:acrysdefn}. Then $\beta - \beta'$ defines a $G_K$-equivariant map $\Z_p(1) \to \Fil^2_H(A_\crys)$. As the only such map is $0$ (by mapping to $B_\dR$ and using Tate's theorem \cite[Theorem 3.3.2]{Tatepdivgrps}, see Remark \ref{rmk:bdr}), one sees $\beta = \beta'$. It remains to show that $\gr^1_H(\beta) = \gr^1_H(\beta')$. For this, note that applying the $p$-adic completion functor to the exact sequence \eqref{eq:omegaonezpbar} from Proposition \ref{prop:diffformszpbar} gives an exact sequence
\[ 1 \to \widehat{\overline{\calO_K}}(1) \stackrel{a}{\to} \widehat{L_{\overline{\calO_K}/W}}[-1] \to Q \to 1\]
where the cokernel $Q$ is a cyclic $\widehat{\overline{\calO_K}}$-module killed (exactly) by $p^{\frac{1}{p-1}}$. By Construction \ref{cons:acryschern}, it is immediate that the map $\gr^1_H(\beta)$ is given by the composite
\[ \Z_p(1) \stackrel{\can}{\to} \widehat{\overline{\calO_K}}(1) \stackrel{a}{\to} \widehat{L_{\overline{\calO_K}/W}}[-1]. \]
The map $\gr^1_H(\beta')$ is 
\[ \Z_p(1) \stackrel{ (\epsilon_n) \mapsto [\underline{\epsilon}] - 1}{\to} \ker(\theta)/\ker(\theta)^2 \simeq \widehat{L_{\overline{\calO_K}/W}}[-1], \]
where the last isomorphism comes from Proposition \ref{prop:ainfprop}. Showing $\gr^1_H(\beta) = \gr^1_H(\beta')$ amounts to showing that the elements $[\underline{\epsilon}] - 1 \in \ker(\theta)/\ker(\theta)^2$ and $ \lim_n d\log(\epsilon_n) \in \widehat{L_{\overline{\calO_K}/W}}[-1]$ agree under the isomorphism $\ker(\theta)/\ker(\theta)^2 \simeq \widehat{L_{\overline{\calO_K}/W}}[-1]$. This was shown in Remark \ref{rmk:explicitformd}, and is also stated in \cite[\S 1.5.4]{Fontainepadicperiod}. \qedhere
\end{enumerate}
\end{proof}

\begin{remark}
\label{rmk:bdr}
	Proposition \ref{prop:acrysbeta} used properties of the ring $B_\dR$, so we briefly recall the definition; see \cite{Fontainepadicperiod} for more. The ring $B_\dR^+$ is the completion of $A_\inf[1/p]$ along the ideal $\ker(\theta)[1/p]$. It is a complete discrete valuation ring with residue field $\widehat{\overline{K}}$ and an action of $\Gal(\overline{K}/K)$; the ring $B_\dR$ is simply the fraction field of $B_\dR^+$. Powers of the maximal ideal define a complete filtration of $B_\dR$, and the graded pieces of this filtration are $\gr^k(B_\dR^+) \simeq \widehat{\overline{K}}(k)$ for any $k \in \Z$. Tate's theorem \cite[Theorem 3.3.2]{Tatepdivgrps} implies that $H^0(\Gal(\overline{K}/K),\widehat{\overline{K}}(k))$ is trivial if $k \neq 0$, and  $K$ for $k = 0$. The completeness of the filtration then implies that $H^0(\Gal(\overline{K}/K), B_\dR) = K$, and $H^0(\Gal(\overline{K}/K),\Fil^k B_\dR) = 0$ for $k > 0$. To relate this to $A_\crys$, observe that the image of $\ker(\theta) \subset A_\inf$ in $B_\dR^+$ lies in $\Fil^1 B_\dR^+$, and hence has topologically nilpotent divided powers. The defining map $A_\inf \to B_\dR^+$ then extends to a filtered map $A_\crys \to B_\dR^+$ which can be checked to be injective by checking it on graded pieces as the filtration on $A_\crys$ is separated.
\end{remark}

Next, we discuss some extensions. Fontaine defined a certain natural non-zero element of $H^1(G_K,A_\crys)$; we construct it as logarithmic Chern class:

\begin{construction}
\label{cons:logchernclasszpbar}
Fix a uniformiser $\pi \in \calO_K$. Then the $d\log$ maps in logarithmic derived de Rham cohomology define additive maps
\[ \st^\Mon_\pi:\pi^\N  \subset \overline{\calO_K} - \{0\} \stackrel{d\log}{\to} \dR_{(\overline{\calO_K},\can)/W}[1]. \]
Applying the $p$-adic completion functor, using Proposition \ref{prop:ddrlogddrzpbar}, and passing to group completions on the source defines an additive map
\[ \st_\pi:\Z \simeq (\pi^\N)^\grp  \to  A_\crys[1], \]
i.e., a $G_K$-equivariant extension of $\Z$ by $A_\crys$, depending on the choice of $\pi$.  We let $\cl(\st_\pi) \in H^1(G_K,A_\crys)$ denote the class of the corresponding extension.  
\end{construction}

To interpret the class $\cl(\st_\pi)$ geometrically, we need an auxilliary ring $R_{\calO_K}$, the so-called Faltings-Breuil ring.

\begin{lemma}
\label{lem:faltingsbreuil}
Let $\pi \in \calO_K$ be a fixed uniformiser with minimal (Eisenstein) polynomial $E(x) \in W[x]$. Let $(W[x],x) \to (\calO_K,\can)$ denote the unique $W$-linear map of prelog rings defined by $x \mapsto \pi$. Then
\[ \widehat{\dR_{(\calO_K,\can)/(W[x],x)}} \simeq \widehat{W[x] \langle E(x) \rangle} =: R_{\calO_K}.\]
The Hodge filtration on the left coincides with the pd-filtration on the right.
\end{lemma}
\begin{proof}
The map $(W[x],x) \to (\calO_K,\can)$ is strict, up to passage to associated log structure, and has kernel generated by a single regular element $E(x)$. Hence, the claim follows immediately from Theorem \ref{thm:padicddrcryscomp}.
\end{proof}

The promised geometric interpretation of $\st_\pi$ is:

\begin{proposition} 
\label{prop:stpigeom}
Let notation be as in Lemma \ref{lem:faltingsbreuil}. Then
\begin{enumerate}
\item The class $\cl(\st_\pi)$ is the obstruction to factoring the natural map $\widehat{\dR_{(\calO_K,\can)/W}} \to \widehat{\dR_{(\overline{\calO_K},\can)/W}} \simeq A_\crys$ in a $G_K$-equivariant manner through the projection $\widehat{\dR_{(\calO_K,\can)/W}} \to \widehat{\dR_{(\calO_K,\can)/(W[x],x)}} \simeq R_{\calO_K}$. 
\item This obstruction vanishes after adjoining $p$-power roots of $\pi$, i.e., if $K_\infty = \cup_n K(\pi^{\frac{1}{p^n}})$ for a chosen compatible sequence of $p$-power roots of $\pi$, then $\cl(\st_\pi)$ maps to $0$ under $H^1(G_K,A_\crys) \to H^1(G_{K_\infty},A_\crys)$. In particular, there is a canonical $G_{K_\infty}$-equivariant map $R_{\calO_K} \to A_\crys$.
\end{enumerate}
\end{proposition}
\begin{proof}
We freely use the identification between derived de Rham and crystalline cohomology (Theorem \ref{thm:padicddrcryscomp}) to compute derived de Rham in terms of explicit de Rham complexes.
\begin{enumerate}
\item  The factorisation $W \to (W[x],x) \to (\calO_K,\can)$ lets us compute $\widehat{\dR_{(\calO_K,\can)/W}}$ as the complex
\[ R_{\calO_K} \otimes_{W[x]} \Big(W[x] \stackrel{d}{\to} W[x] \cdot \frac{dx}{x} \Big) \simeq \Big(R_{\calO_K} \stackrel{d}{\to} R_{\calO_K} \cdot \frac{dx}{x}\Big). \]
The map $\Z \simeq (\pi^\N)^\grp \to (\calO_K - \{0\})^\grp \stackrel{d\log}{\to} \widehat{\dR_{(\calO_K,\can)/W}}[1]$ can then be identified as the map $\Z \to \widehat{\dR_{(\calO_K,\can)/W}}[1]$ determined by $\frac{dx}{x}$ in the complex above. On the other hand, composing this map with $\widehat{\dR_{(\calO_K,\can)/W}} \to \widehat{\dR_{(\overline{\calO_K},\can)/W}} \simeq A_\crys$ defines $\st_\pi$. The claim follows by chasing triangles.

\item A choice of a compatible sequence of $p$-power roots of $\pi$ determines a $G_{K_\infty}$-equivariant map
\[ c:\Z[1/p] \simeq \Big(\pi^{\N[\frac{1}{p}]}\Big)^\grp \subset (\overline{\calO_K} - \{0\})^\grp \stackrel{d\log}{\to} \dR_{(\overline{\calO_K},\can)/W}[1]. \]
Restricting to $\Z \cdot 1 \subset \Z[1/p]$ followed by $p$-adic completion on the target recovers the map $\st_\pi$. However, $p$-adically completing $\Z[1/p]$ produces $0$, so the $p$-adic completion of $c$ is $G_{K_\infty}$-equivariantly nullhomotopic. It follows that the same is true for $\st_{\pi}$, proving the claim. \qedhere
\end{enumerate}
\end{proof}

\begin{remark}
\label{rmk:stpiexplicit}
The proof of the second part of Proposition \ref{prop:stpigeom} gives an explicit identification of the class $\cl(\st_\pi)$ as follows. Fix a compatible sequence $\kappa = \{ \pi^{\frac{1}{p^n}} \in \overline{\calO_K} \}$ of $p$-power roots of $\pi$. Then this choice $\kappa$ determines a nullhomotopy $H_\kappa$ of the map $\st_\pi:\Z \to A_\crys[1]$ by the recipe of the proof. This nullhomotopy is $G_{K_\infty}$-equivariant, and its failure to be $G_K$-equivariant is tautologically codified by the map $G_K \to A_\crys$ determined by
\[ \sigma \mapsto H_{\sigma(\kappa)} - H_\kappa.\]
Unravelling definitions, this is simply the map
\[ \sigma \mapsto \log(\frac{\sigma([\underline{\pi}])}{[\underline{\pi}]}), \] 
which is the usual formula for Fontaine's extension. In particular, Proposition \ref{prop:acrysbeta} shows that $\cl(\st_\pi)$ actually comes from the Kummer torsor $\overline{\kappa} \in H^1(G_K,\Z_p(1))$ (determined by $\kappa$) by pushforward along $\beta:\Z_p(1) \to A_\crys$.
\end{remark}

Next, we discuss Kato's semistable ring $\widehat{A_\st}$, and its connection with the class $\cl(\st_\pi)$. We give a direct definition first; a derived de Rham interpretation is given in the proof of Proposition \ref{prop:stpiast}.

\begin{definition}
\label{defn:ast}
Fix a uniformiser $\pi \in \calO_K$, and a sequence $\{\pi^{\frac{1}{p^n}} \in \overline{\calO_K}\}$ of $p$-power roots of $\pi$. The ring $\widehat{A_\st}$ is defined as $\widehat{A_\crys \langle X \rangle}$, the free $p$-adically complete pd-polynomial ring in one variable $X$. This ring is endowed with a $G_K$-action extending the one on $A_\crys$ given by
\[ \sigma(X + 1) = \frac{[\underline{\pi}]}{\sigma([\underline{\pi}])} \cdot (X + 1).\]
We equip $\widehat{A_\st}$ with the minimal pd-multiplicative Hodge filtration extending the one on $A_\crys$ and satisfying $X \in \Fil^1_H(\widehat{A_\st})$. We define $\phi:\widehat{A_\st} \to \widehat{A_\st}$ to be the unique extension of $\phi$ on $A_\crys$ which satisfies $\phi(X+1) = (X+1)^p$. Finally, we define a continuous $A_\crys$-linear pd-derivation $N:\widehat{A_\st} \to \widehat{A_\st}$ via $N(1 + X) = 1 + X$.
\end{definition}

\begin{remark}
The construction of $\widehat{A_\st}$ given in Definition \ref{defn:ast} relied not just on $\pi \in \calO_K$, but also on a choice a compatible sequence of $p$-power roots of $\pi$. However, one can show that the resulting ring (with its extra structure) is independent of this last choice, up to a transitive system of isomorphisms; see \cite[\S 5]{BreuilMessing}. In fact, Kato disocvered $\widehat{A_\st}$ as the log crystalline cohomology of a certain map which depends only on $\pi$; see Remark \ref{rmk:astddrdefn}.
\end{remark}

\begin{proposition}
\label{prop:stpiast}
The class $\cl(\st_\pi)$ maps to $0$ under $A_\crys \to \widehat{A_\st}$.
\end{proposition}
\begin{proof}
One may prove this assertion directly using Remark \ref{rmk:stpiexplicit}. However, we give a ``pure thought'' argument: the ring $\widehat{A_\st}$ will be realised as the derived de Rham cohomology of a map, and a commutative diagram will force $\cl(\st_\pi)$ to vanish when pushed to $\widehat{A_\st}$. For convenience, we fix a compatible system of {\em all} rational powers of $\pi$. Let 
\[ C := (A_\inf[y,y^{-1}], [\underline{\pi}]^{\Q_{\geq 0}} \cdot y^{\Z})\]
be the displayed prelog ring (defined using the choice of rational powers of $\pi$), with a $G_K$-action defined by 
\[ \sigma(y) = \frac{[\underline{\pi}] }{\sigma([\underline{\pi}])}\cdot y, \]
extending the usual action on $A_\inf$. The map $y \mapsto 1$, coupled with the usual map $A_\inf \to \widehat{\overline{\calO_K}}$, defines a map
\[ C \to (\widehat{\overline{\calO_K}},\can) \]
which is essentially strict, i.e., the associated map on log rings is strict, and has kernel $(E([\underline{\pi}]),(y-1))$, which is a regular sequence. We define
\[ \widehat{A_\st}' := \widehat{\dR_{(\widehat{\overline{\calO_K}},\can)/C}}. \]
The composite $W \to C \to (\widehat{\overline{\calO_K}},\can)$ gives $\widehat{A_\st}'$ the structure of an $A_\crys$-algebra. Moreover, one computes that 
\[ \widehat{A_\st}' = \widehat{A_\inf[y,y^{-1}] \langle E([\underline{\pi}]), y - 1 \rangle} = \widehat{A_\crys[y,y^{-1}]\langle y-1 \rangle} \simeq \widehat{A_\crys\langle y-1 \rangle}. \]
Hence, the association $X \mapsto y-1$ identifies $\widehat{A_\st}$ with $\widehat{A_\st}'$ in a $G_K$-equivariant manner, and we use this isomorphism without comment for the rest of the proof. The map $x \mapsto [\underline{\pi}] \cdot y$ defines a $G_K$-equivariant diagram of rings
\[ \xymatrix{ W \ar[r] & (W[x],x) \ar[d] \ar[r]^-{x \mapsto [\underline{\pi}] \cdot y} & C \ar[d]^-{y \mapsto 1} \\
				&	(\calO_K,\can) \ar[r] & (\widehat{\overline{\calO_K}},\can). } \]
Passing to de Rham cohomology gives us a $G_K$-equivariant commutative diagram
\[ \xymatrix{ \dR_{(\calO_K,\can)/W} \ar[r] \ar[d] & \dR_{(\widehat{\overline{\calO_K}},\can)/W} \ar[d] \\
			  \dR_{(\calO_K,\can)/(W[x],x)} \ar[r] & \dR_{(\widehat{\overline{\calO_K}},\can)/C}. } \]
Taking $p$-adic completions then gives a $G_K$-equivariant commutative diagram
\[ \xymatrix{ \widehat{\dR_{(\calO_K,\can)/W}} \ar[r] \ar[d] & A_\crys \ar[d] \\
					R_{\calO_K} \ar[r] & \widehat{A_\st}. } \]
In other words, the natural map $\widehat{\dR_{(\calO_K,\can)/W}} \to A_\crys \to \widehat{A_\st}$ factors $G_K$-equivariantly through the projection $\widehat{\dR_{(\calO_K,\can)/W}} \to R_{\calO_K}$. The vanishing claim now follows from Proposition \ref{prop:stpigeom}.
\end{proof}

\begin{remark}
\label{rmk:astddrdefn}
Let $(W[x],x) \stackrel{a}{\to} C \stackrel{b}{\to} (\widehat{\overline{\calO_K}},\can)$ be the factorisation appearing in the proof of Proposition \ref{prop:stpiast}. One can check that this factorisation is an {\em exactification} of the composite. In particular, the log crystalline cohomology of $b \circ a$ and that of $b$ coincide, and they both recover the ring $\widehat{A_\st}$; this is Kato's conceptual definition of $\widehat{A_\st}$. The argument above shows that $\widehat{\dR_b} \simeq \widehat{A_\st}$, which gives a derived de Rham definition for $\widehat{A_\st}$. However, unlike in the crystalline theory, since the map $a$ is log \'etale but not of Cartier type modulo $p$, the map $\widehat{\dR_{b \circ a}} \to \widehat{\dR_b}$ fails to be an isomorphism (for roughly the same reason as Example \ref{ex:logddrnotlogcrys}). This explains why we cannot define $\widehat{A_\st}$ as $\widehat{\dR_{b \circ a}}$, a much easier ring to contemplate than $\widehat{\dR_b}$. Note, however, that the proof given above also applies to show that $\cl(\st_\pi)$ trivialises under $A_\crys \to \widehat{\dR_{b \circ a}}$. This leads to a comparison theorem over the ring $\widehat{\dR_{b \circ a}}$, which is smaller than $\widehat{A_\st}$.
\end{remark}

\begin{remark}
	\label{rmk:astddrdefn2}
	We continue the notation of Remark \ref{rmk:astddrdefn}. The map $\widehat{\dR_{b \circ a}} \to \widehat{A_\st}$ is simply the map $\widehat{\comp_{b \circ a}}$ from Proposition \ref{prop:lddrcryscompmap}. In particular, the $(W[x],x)$-algebras $\widehat{\dR_{b \circ a}}$ and $\widehat{A_\st}$ come equipped with the Gauss-Manin connection (the former by Proposition \ref{prop:logcrystrans}, and the latter by Kato's theorem) relative to $W$, and the map $\widehat{\comp_{b \circ a}}$ is equivariant for the connection. In fact, the Gauss-Manin connection on $\widehat{A_\st}$ can be identified with the derivation $N$ introduced in Definition \ref{defn:ast}: the isomorphism $\widehat{L_{(W[x],x)/W}} \simeq \widehat{L_{C/W}}$ sends $d\log(x)$ to $d\log([\underline{\pi}] y) = d\log([\underline{\pi}]) + d\log(y) = d\log(y)$ (since $d\log([\underline{\pi}])$ is infinitely $p$-divisible and hence $0$ $p$-adically).
\end{remark}

\section{The  semistable comparison theorem}
\label{sec:cst}

Our goal now is to use the theory developed earlier in the paper to prove the Fontaine-Jannsen $C_\st$ conjecture following the outline of \cite{Beilinsonpadic}. Inspired by the complex analytic case, we construct in \S \ref{ss:sspairs} a topology on $p$-adic schemes for which derived de Rham cohomology sheafifies to $0$ $p$-adically; this can be viewed as a $p$-adic Poincare lemma, and is the key conceptual theorem. The comparison map is constructed in \S \ref{ss:cst} using the Poincare lemma.

\subsection{The site of pairs and the Poincare lemma}
\label{ss:sspairs}

We preserve the notation from \S \ref{sec:ddr-period-rings}, and introduce the geometric categories of interest.

\begin{notation}
Let $\Var_K$ and $\Var_{\overline{K}}$ be the categories of reduced and separated finite type schemes over the corresponding fields. These categories are viewed as sites via the $h$-topology, the coarsest topology finer than the Zariski and proper topologies; see \cite[\S 2]{Beilinsonpadic}.
\end{notation}

Next, we define the category $\calP_K$ of pairs. Roughly speaking, an object of this site is a variety $U \in \Var_K$ together with a compactification $\overline{U}$ relative to  $\calO_K$; the compactification $\overline{U}$ will help relate the de Rham cohomology of $U$ to mixed characteristic phenomena. 

\begin{definition}
\label{ss:sitelocal}
The site $\calP_K$ of pairs over $K$ has as objects pairs $(U,\overline{U})$ with $U$ a reduced and separated finite type $K$-scheme, and $\overline{U}$ a proper flat reduced $\calO_K$-scheme containing $U$ as a dense open subscheme. The morphisms are defined in the evident way. We often write maps in $\calP_K$ as $\overline{f}:(U,\overline{U}) \to (V,\overline{V})$ with underlying map $f:U \to V$.
\end{definition}

The pair $(\Spec(K),\Spec(\calO_K))$ is the final object of $\calP_K$. Moreover, each pair $(U,\overline{U}) \in \calP_K$ defines a log scheme $(\overline{U},\can)$ where $\can:\calO_{\overline{U}} \cap \calO_U^* \to \calO_{\overline{U}}$ is the log structure defined by the open susbet $U \subset \overline{U}$. Forgetting $\overline{U}$ defines a faithful functor $\calP_K \to \Var_K$, and the $h$-topology on $\calP_K$ is defined to be the pullback of the $h$-topology from $\Var_K$ under this functor. An essential observation \cite[\S 2.5]{Beilinsonpadic} is that $\calP_K$ is a particularly convenient basis for $\Var_K$:

\begin{proposition}
\label{prop:pairsvars}
The functor $\calP_K \to \Var_K$ is continuous and induces an equivalence of associated topoi.
\end{proposition}
\begin{proof}[Proof sketch]
First, one checks directly that every map $f:U \to V$ in $\Var_K$ extends to a map $\overline{f}:(U,\overline{U}) \to (V,\overline{V})$ between suitable pairs; similarly for covers. Next, essentially by blowing up, one shows that given pairs $(U,\overline{U})$ and $(V,\overline{V})$ and a map $f:U \to V$, there is an $h$-cover $\overline{\pi}:(U',\overline{U'}) \to (U,\overline{U})$ and a map $\overline{g}:(U',\overline{U'}) \to (V,\overline{V})$ of pairs such that $\pi \circ f = g$. Using the faithfulness of $\calP_K \to \Var_K$, it follows formally that $\Shv_h(\calP_K) \simeq \Shv_h(\Var_K)$.
\end{proof}

From now on, we will freely identify $h$-sheaves on $\calP_K$ with $h$-sheaves on $\Var_K$. In particular, to specify an $h$-sheaf on $\Var_K$, it will suffice to give an $h$-sheaf on $\calP_K$. Our ultimate goal is to relate de Rham cohomology to \'etale cohomology. The following result ensures that the $h$-topology on $\calP_K$ is good enough to compute \'etale cohomology:

\begin{corollary}[Deligne]
\label{cor:hetale}
Let $\underline{A}$ be a constant torsion abelian sheaf on $\Var_K$ with value $A$, and let $(U,\overline{U}) \in \calP_K$. Then we have a canonical equivalence
\[ \R\Gamma_{\calP_K}( (U,\overline{U}), \underline{A}) \simeq \R\Gamma(U_\et,\underline{A}).\]
\end{corollary}
\begin{proof}
This follows from Proposition \ref{prop:pairsvars} and Deligne's theorem \cite[Proposition 4.3.2, Expose V bis]{SGA4_2} that \'etale cohomology with constant torsion coefficients can be computed in the $h$-topology.
\end{proof}

In the definition of $\Var_K$ and $\calP_K$, we impose no restrictions on fields of definition of the objects. For applications, it is convenient to work with objects are defined over a fixed base. Working with finite extensions of $K$ is not possible as we cannot control the field of definition along $h$-covers; instead, we define a variant of these categories over $\overline{K}$:

\begin{definition}
The site $\calP_{\overline{K}}$ of geometric pairs has as objects pairs $(V,\overline{V})$ where $V$ is a reduced and separated finite type $\overline{K}$-scheme, and $\overline{V}$ is a proper flat reduced $\overline{\calO_K}$-scheme containing $V$ as a dense open subscheme. The morphisms are defined in the evident way.
\end{definition}

For any finite subextension $K \subset L \subset \overline{K}$, there is a base change functor $\calP_L \to \calP_{\overline{K}}$ defined by $(U,\overline{U}) \mapsto (U \otimes_L \overline{K}, (\overline{U} \otimes_{\calO_L} \overline{\calO_K})_\red)$. By \cite[Theorem 3.4.6]{RaynaudGruson},  every finite type flat $\overline{\calO_K}$-scheme is automatically finitely presented, so each pair $(U,\overline{U}) \in \calP_{\overline{K}}$ comes from $\calP_L$ for some $L$.  In particular, ``geometric'' techniques (such as \cite[\S 4]{dJAlt} and \cite{Bhattmixedcharpdiv}) can be applied to such pairs by limit arguments. The logarithmic and topological remarks concerning $\calP_K$ also apply to $\calP_{\overline{K}}$. In particular, $(\Spec(\overline{K}),\Spec(\overline{\calO_K}))$ is the final object of this category. 

\begin{remark}
\label{rmk:geompairscohomology}
The forgetful functor $\calP_{\overline{K}} \to \Var_{\overline{K}}$ lets us define an $h$-topology on $\calP_{\overline{K}}$. The analogues of Proposition \ref{prop:pairsvars} and Corollary \ref{cor:hetale} obtained by replacing $\calP_K$ (resp. $\Var_K$) with $\calP_{\overline{K}}$ (resp. $\Var_{\overline{K}}$) are also true, and proven in exactly the same way. In fact, there is a pro-\'etale morphism $\phi:\Shv_h(\calP_{\overline{K}}) \to \Shv_h(\calP_K)$ of topoi with $\phi^{-1}$ being the pullback functor introduced above.
\end{remark}

\begin{remark}
\label{rmk:sspair}
By \cite[Theorem 4.5]{dJAlt}, every pair $(U,\overline{U}) \in \calP_K$ admits an $h$-cover $\overline{\pi}:(V,\overline{V}) \to (U,\overline{U})$ with $(V,\overline{V})$ a {\em semistable} pair, i.e., $\overline{V}$ is regular, $\overline{V} - V$ is a simple normal crossings divisor on $\overline{V}$, and the fibres of $\overline{V} \to \Spec(\calO(\overline{V}))$ are reduced. As in Proposition \ref{prop:pairsvars}, this observation can be improved to say: the collection of all such semistable pairs forms a subcategory $\calP_K^\ss \subset \calP_K$ such that $\Shv_h(\calP_K) \simeq \Shv_h(\Var_K) \simeq \Shv_h(\calP_K^\ss)$ via the evident functors (and similarly for a variant category $\calP_{\overline{K}}^\ss$ of semistable pairs over $\overline{K}$). Hence, at the expense of keeping track of more conditions, we could work consistently with the more ``geometric'' category of semistable pairs (as opposed to arbitrary pairs) in this paper without changing the arguments seriously.
\end{remark}

Our main theorem is a Poincare lemma relating two natural sheaves on $\Var_{\overline{K}}$: one computes \'etale cohomology, while the other is closely related to de Rham cohomology. These sheaves are:

\begin{construction}
\label{ss:periodsheaves}
There are two {\em presheaves} $\fraa^c_{\crys}$ and $\fraa_{\crys}$ on $\calP_K$ defined by
\[  \fraa^c_{\crys}(U,\overline{U}) = \dR_{(\calO(\overline{U}),\can)/W} \]
and
\[ \fraa_{\crys}(U,\overline{U}) = \R\Gamma(\overline{U},\dR_{(\overline{U},\can)/W}). \]
The object on the right in the preceding formula is the hypercohomology in the Zariski topology of $\overline{U}$ of displayed complex. Both these presheaves are presheaves of cochain complexes with an algebra structure, and we view them as living in an appropriate (symmetric monoidal) stable $\infty$-category of presheaves. Let $\calA^c_{\crys}$ and $\calA_{\crys}$ denote the $h$-sheafifications of $\fraa^c_{\crys}$ and $\fraa_{\crys}$ respectively. Pullback of forms induces natural maps $\fraa^c_{\crys} \to \fraa_{\crys}$ and $\calA^c_{\crys} \to \calA_{\crys}$. We denote the corresponding objects of $\calP_{\overline{K}}$ by the same notation.
\end{construction}

The cohomology of the sheaf $\calA^c_\crys$ is essentially \'etale cohomology:

\begin{proposition}
Fix an object $(U,\overline{U}) \in \calP_{\overline{K}}$. Then one has:
\[ \R\Gamma_{\calP_{\overline{K}}}( (U,\overline{U}),\calA^c_{\crys} \otimes_{\Z} \Z/p^n) \simeq \R\Gamma(U_\et,\Z/p^n) \otimes_{\Z/p^n} A_\crys/p^n.\]
\end{proposition}
\begin{proof}
Note first that $\calA^c_{\crys}$ is a constant sheaf on $\calP_{\overline{K}}$ since $\calO(\overline{U}) \simeq \overline{\calO_K}^{\#\pi_0(\overline{U})}$ for any $(U,\overline{U}) \in \calP_{\overline{K}}$. Moreover, Proposition \ref{prop:acrysdefn} shows that
\[ \calA^c_{\crys}(\ast) \otimes_{\Z} \Z/p^n \simeq \dR_{(\overline{\calO_K},\can)/W} \otimes_\Z \Z/p^n =: A_\crys/p^n.\]
The claim now follows from Corollary \ref{cor:hetale} (and Remark \ref{rmk:geompairscohomology}).
\end{proof}

The Poincare lemma asserts that $\calA^c_\crys$ and $\calA_\crys$ are $p$-adically isomorphic. To prove this, we first prove a theorem showing that the difference is $p$-adically small, at least $h$-locally; this is the key geometric ingredient in this paper.

\begin{theorem}
\label{thm:mixedcharpdiv}
For any pair $(U,\overline{U}) \in \calP_{\overline{K}}$, there exists an $h$-cover $\pi:(V,\overline{V}) \to (U,\overline{U})$ such that 
\begin{enumerate}
\item The induced map
\[ \tau_{\geq 1} \R\Gamma(\overline{U},\calO_{\overline{U}}) \to \tau_{\geq 1} \R\Gamma(\overline{V},\calO_{\overline{U}}) \]
is divisible by $p$ as a map in the derived category.
\item For $i > 0$, the induced map
\[ \R\Gamma(\overline{U},\Omega^i_{(\overline{U},\can)/(\overline{\calO_K},\can)}) \to \R\Gamma(\overline{V},\Omega^i_{(\overline{V},\can)/(\overline{\calO_K},\can)})\]
is divisible by $p$ as a map in the derived category.
\end{enumerate}
\end{theorem}
\begin{proof}
The first claim is \cite[Remark 2.10]{Bhattmixedcharpdiv}, while the second claim follows from Lemma \ref{lem:pdivomega1}. More precisely, both references ensure the relevant $p$-divisibility at the level of cohomology groups. To obtain $p$-divisibility at the level of complexes, one simply iterates the relevant construction $(\dim(\overline{U})-1)$-times (see \cite[Lemma 3.2]{ddscposchar}).
\end{proof}

\begin{remark}
	We do not know if the conclusion of Theorem \ref{thm:mixedcharpdiv} holds if we replace the base $\overline{\calO_K}$ with a higher-dimensional ring; this seems to be an obstacle in extending the present approach to the comparison theorems to the relative setting. The geometric question amounts to:  given an affine scheme $\Spec(A)$ and a proper map $f:X \to \Spec(A)$, can one find proper covers $\pi:Y \to X$ such that the induced map $\tau_{\geq 1} \R\Gamma(X,\calO_X) \to \tau_{\geq 1} \R\Gamma(Y,\calO_Y)$ is divisible by $p$? We can prove such divisibility for $\tau_{\geq 2}$ (and perhaps for $\dim(A) \leq 2$), but not generally.
\end{remark}

The following lemma was used in the proof of Theorem \ref{thm:mixedcharpdiv}.

\begin{lemma}
\label{lem:pdivomega1}
Let $(X,\overline{X}) \in \calP_{\overline{K}}$. Then there exist $(Y,\overline{Y}) \in \calP_{\overline{K}}$ and an $h$-cover $\pi:(Y,\overline{Y}) \to (X,\overline{X})$ such that $\pi^* \Omega^1_{(\overline{X},\can)/(\overline{\calO_K},\can)} \to \Omega^1_{(\overline{Y},\can)/(\overline{\calO_K},\can)}$ is divisible by $p$ as a map. In particular, $\Omega^1_{(\overline{X},\can)/(\overline{\calO_K},\can)} \to \R \pi_* \Omega^1_{(\overline{Y},\can)/(\overline{\calO_K},\can)}$ is divisible by $p$ as a map.
\end{lemma}
\begin{proof}
It suffices to construct a proper surjective map $\pi:\overline{Y} \to \overline{X}$ of schemes such that the desired $p$-divisibility holds for the pullback log structure on $\overline{Y}$. First, we claim that it suffices to solve this problem locally on $\overline{X}$. Indeed, assume that there exists a Zariski cover $\{U_i \subset \overline{X}\}$ and proper surjective maps $\pi_i:V_i \to U_i$ such that $\pi^* \Omega^1_{(U_i,\can)/(\overline{\calO_K},\can)} \to \Omega^1_{(V_i,\can)/(\overline{\calO_K},\can)}$ is divisible by $p$; here all log structures are defined by pullback from the given log structure on $\overline{X}$. Then, by Nagata, we can find a single proper surjective $\pi:\overline{Y} \to \overline{X}$ which factors through $\pi_i$ over $U_i$. Moreover, by Remark \ref{rmk:sspair}, we may ensure that $(\overline{Y},\can)$ defines a semistable pair, where $\can$ denotes the log structure defined by $\pi^{-1}(X)$. In particular, $\Omega^1_{(\overline{Y},\can)/(\overline{\calO_K},\can)}$ is $\Z_p$-flat. Now the pullback map $\pi^* \Omega^1_{(\overline{X},\can)/(\overline{\calO_K},\can)} \to \Omega^1_{(\overline{Y},\can)/(\overline{\calO_K},\can)}$ is divisible by $p$ over each $U_i$, and hence globally divisible by $p$ by flatness. Hence, we have now reduced to solving the local problem. Again, we may assume that $(X,\overline{X})$ is a semistable pair. Now pick an affine open cover $\{U_i \subset \overline{X}\}$ such that each $U_i$ is \'etale over $\Spec(\overline{\calO_K}[t_1,\dots,t_d]/(\prod_{i=1}^r t_i - \pi))$ where $\pi \in \overline{\calO_K}$, and the log structure is defined by $t_1,\dots,t_k$ for $r \leq k \leq d$. In this case, extracting $p$-th roots of each $t_i$ can be seen to solve the problem.
\end{proof}

We now prove the promised Poincare lemma relating $\calA^c_\crys$ and $\calA_\crys$.

\begin{theorem} 
\label{thm:padicpoincare}
The map $\calA^c_{\crys} \otimes_{\Z} \Z/p^n \to \calA_{\crys} \otimes_{\Z} \Z/p^n$ is an equivalence of sheaves on $\calP_{\overline{K}}$ for all $n$.
\end{theorem}
 
\begin{proof}
By devissage, it suffices to show the case $n = 1$. Thus, we must show that $\fraa^c_{\crys} \otimes \Z/p \to \fraa_{\crys} \otimes \Z/p$ is an isomorphism after $h$-sheafification. Given a pair $(U,\overline{U}) \in \calP_{\overline{K}}$ with $\overline{U}$ normal (it suffices to restrict to such pairs since every pair is covered by such a pair), we have
\[ (\fraa^c_{\crys} \otimes \Z/p)(U,\overline{U}) = \dR_{(\overline{\calO_K}/p,\can)/k}  \]
and
\[ (\fraa_{\crys} \otimes \Z/p)(U,\overline{U}) = \R\Gamma(\overline{U},\dR_{(\overline{U}/p,\can)/k}). \]
By Proposition \ref{prop:conjfiltddrcomposite},  in the stable $\infty$-category of {\em presheaves} of cochain complexes on $\calP_{\overline{K}}$, the presheaf $\fraa_{\crys} \otimes \Z/p$ admits an increasing bounded below separated exhaustive filtration with graded pieces $\calG_i$ (starting at $i = 0$) given by
\[ \calG_i(U,\overline{U}) \simeq \dR_{(\overline{\calO_K}/p,\can)/k} \otimes_{\Frob_k^* \overline{\calO_K}/p} \Frob_k^* \big(\R\Gamma(\overline{U},\Omega^i_{(\overline{U}/p,\can)/(\overline{\calO_K}/p,\can)})[-i]\big) \]
Moreover, it is easily checked that the map $\fraa^c_{\crys} \otimes \Z/p \to \fraa_{\crys} \otimes \Z/p$ factors through the structure map $\calG_0 \to \fraa_{\crys} \otimes \Z/p$. As sheafification commutes with colimits,  it suffices to show the following:
\begin{enumerate}
\item The $h$-sheafification of the map $\fraa^c_{\crys} \otimes \Z/p \to \calG_0$ is an equivalence.
\item The $h$-sheafification of $\calG_i$ is $0$ for $i > 0$.
\end{enumerate}
Both claims follow the $p$-divisibility results of Theorem \ref{thm:mixedcharpdiv} and base change.
\end{proof} 

\begin{remark}
	\label{rmk:poincarelemmaalternative}
	The proof of Theorem \ref{thm:padicpoincare} shows that one does not really need to work relative to the base $W$: one can define analogs of $\calA^c_\crys$ and $\calA_\crys$ by replacing $W$ with any prelog ring mapping to $(\overline{\calO_K},\can)$ without affecting the conclusion of the theorem. In particular, using $x \mapsto \pi$, if one defines presheaves $\fraa^c_\st$ and $\fraa_\st$ via
	\[ \fraa^c_\st(U,\overline{U}) = \dR_{(\calO(\overline{U}),\can)/(W[x],x)} \quad \mathrm{and} \quad \fraa_\st(U,\overline{U}) = \R\Gamma(\overline{U},\dR_{(\overline{U},\can)/(W[x],x)}),  \]
	then the associated $h$-sheaves $\calA^c_\st$ and $\calA_\st$ will be isomorphic modulo $p^n$ via the natural map $\calA^c_\st \to \calA_\st$.
\end{remark}

\begin{remark}
\label{rmk:padiccompapproaches}
An essential feature of most known approaches to the $p$-adic comparison theorems is the construction of certain well-chosen towers of covers of mixed characteristic schemes, together with a good understanding of of cohomology (either \'etale, or de Rham) as one moves in these towers. In Faltings' method of almost \'etale extensions, the key technical result is the almost purity theorem (see \cite[page 182, Theorem]{FaltingsAEE}) which controls flatness properties of the normalisation of a mixed characteristic ring $R$ in a tower of finite \'etale covers of $R[1/p]$; in the end, this lets one compute {\em \'etale} cohomology of $R[1/p]$ in terms of mixed characteristic data (see \cite[page 242, Theorem]{FaltingsAEE}). In contrast, in the approach of \cite{Beilinsonpadic} as well as this paper, one constructs towers of $h$-covers of mixed characteristic schemes that make {\em de Rham} cohomology classes highly $p$-divisible (see the proof of Theorem \ref{thm:padicpoincare}); that such covers suffice for applications is entirely due to Corollary \ref{cor:hetale}. The fineness of the $h$-topology over the \'etale topology lets one construct the required covers rather easily, while completely eschewing delicate algebraic considerations encountered, for example, in \cite[page 196]{FaltingsAEE}; this is the main reason for the relative simplicity of the present proof. An intermediate between the two methods just described is Scholze's approach (unpublished): he works \'etale locally on the underlying rigid analytic space, which, roughly speaking, amounts to working with $h$-covers of a mixed characteristic formal model that are \'etale over the generic fibre (by Raynaud's theory \cite{RaynaudRigAnGeom}).
\end{remark}

\subsection{The semistable comparison map}
\label{ss:cst}

Using the results of \S \ref{ss:sspairs}, we will construct the promised comparison map, and show it is an isomorphism. For technical reasons pertaining to monodromy, we introduce some  notation first:

\begin{notation}
We continue using Notation \ref{not:periodring}. In particular, we fix once and for all a uniformiser $\pi \in \calO_K$ and, unless otherwise specified, the ring $\calO_K$ (and hence all $\calO_K$-schemes) are viewed as $W[x]$-algebras via $x \mapsto \pi$. The Faltings-Breuil ring $R_{\calO_K}$ is defined as $\widehat{\dR_{(\calO_K,\can)/(W[x],x)}}$ or, equivalently, as the $p$-adic completion of the pd-envelope of $W[x] \to \calO_K$ (by Lemma \ref{lem:faltingsbreuil}).
\end{notation}

We now come to the main theorem:

\begin{theorem} 
\label{thm:compcst}
Let $(X,\overline{X}) \in \calP_K$ be a semistable pair with $X = \overline{X}[1/p]$ and $\calO(\overline{X}) = \calO_K$. Then there is an $\widehat{A_\st}$-linear comparison map
\[ \comp^\st_\et:\widehat{\R\Gamma_\crys( (\overline{X},\can)/(W[x],x), \calO_\crys)} \otimes_{R_{\calO_K}} \widehat{A_\st} \to \R\Gamma(X_{\overline{K},\et},\Z_p) \otimes_{\Z_p} \widehat{A_\st} \]
that preserves filtrations, $G_K$-actions, Frobenius actions, Chern classes of vector bundles, and monodromy operators. Moreover, $\comp^\st_\et$ admits an inverse up to $\beta^d$, where $\beta \in A_\crys$ is Fontaine's element from Proposition \ref{prop:acrysbeta}, and $d = \dim(X)$. In particular, Fontaine's $C_\st$-conjecture is true.
\end{theorem}

We refer to Remark \ref{rmk:sspair} for the definition of a semistable pair. The left hand side above is defined via
\[ \widehat{\R\Gamma_\crys( (\overline{X},\can)/(W[x],x),\calO_\crys)} := \R\lim_n \R\Gamma_\crys( (\overline{X},\can)/(W[x],x) \otimes \Z/p^n,\calO_\crys).\]
This is a module over $R_{\calO_K}$, and agrees with the crystalline cohomology groups $H^*(X/R_V)$ of \cite{FaltingsAEE}. The groups in \cite[\S 2]{FaltingsIntCrys} are slightly different because the ring $R_V$ there is complete for the Hodge filtration. Informally, we may think of $\widehat{\R\Gamma_\crys( (\overline{X},\can)/(W[x],x), \calO_\crys)}$ as the de Rham cohomology over $(R_{\calO_K},x)$ of a deformation of $(\overline{X},\can)$ across $(R_{\calO_K},x) \to (\calO_K,\can)$; as such deformations might not exist globally on $\overline{X}$, one has to proceed using cohomological descent. In the sequel, we will often write $\comp$ instead of $\comp_\et^\st$ when the meaning is clear.

\begin{remark}
The Chern classes mentioned in Theorem \ref{thm:compcst} live in crystalline (and \'etale) cohomology. In the spirit of the present paper, a more natural operation would be to define Chern classes in derived de Rham cohomology that lift crystalline Chern classes via the comparison maps of propositions \ref{prop:ddrcryscompmap} and \ref{prop:lddrcryscompmap}. A natural solution to this last problem is to develop a theory of derived de Rham cohomology for algebraic stacks over some base $S$, and construct universal Chern classes in $\R\Gamma(B(\GL_n),\dR_{B(\GL_n)/S})$. This can indeed be done over $S = \Spec(\Z/p^n)$, and will be discussed in \cite{BhattFlatDescentCC}. We simply remark here that our definition proceeds by cohomological descent instead of imitating Illusie's definition of derived de Rham cohomology; the latter is problematic to implement for Artin stacks as it is not clear how to define wedge powers of a complex that is supported in both positive and negative degrees.
\end{remark}

\begin{proof}[Construction of the map]
We first explain the idea informally. The sheafification adjunction gves a natural map $\fraa_\crys( (X,\overline{X})/W) \to \calA_\crys( (X,\overline{X}) \otimes_W \overline{\calO_K})$. Up to completion, the right hand side is the $p$-adic \'etale cohomology of $X_{\overline{K}}$, by the $p$-adic Poincare lemma. The left hand side is closely related to the left hand side of the desired map $\comp^\st_\et$: the latter is the de Rham cohomology of $(\overline{X},\can)$ relative to $(W[x],x)$, while the former is the de Rham cohomology of $(\overline{X},\can)$ over $W$ (up to completions).  Hence, adjunction almost gives a map of the desired form. To move from de Rham cohomology relative to $W$ to that relative to $(W[x],x)$, we extend scalars to $\widehat{A_\st}$.

Now for the details. First consider the map
\[ \comp_n:\fraa_\crys(X,\overline{X}) \otimes_\Z \Z/p^n \to \fraa_\crys( (X,\overline{X}) \otimes_{\calO_K} \overline{\calO_K}) \otimes_\Z \Z/p^n \to \R\Gamma_{\calP_{\overline{K}}}((X,\overline{X}) \otimes \overline{\calO_K}, \calA_{\crys} \otimes_\Z \Z/p^n).\]
The left hand side is computed as
\[ \fraa_{\crys}(X,\overline{X}) \otimes_\Z \Z/p^n = \R\Gamma(\overline{X},\dR_{(\overline{X},\can)/W} \otimes_W W/p^n)  = \R\Gamma( (\overline{X},\can)/W \otimes_W W/p^n,\calO_\crys) . \]
while the right hand side is computed by the Poincare lemma to be
\[ \R\Gamma_{\calP_{\overline{K}}}( (X,\overline{X}) \otimes \overline{\calO_K},\calA_{\crys} \otimes_\Z \Z/p^n) \simeq \R\Gamma_{\calP_{\overline{K}}}( (X,\overline{X}) \otimes \overline{\calO_K},\calA^c_{\crys} \otimes_\Z \Z/p^n) \simeq \R\Gamma(X_{\overline{K},\et},\Z/p^n) \otimes_{\Z/p^n} \widehat{A_{\crys}}/p^n. \]
Taking $p$-adic limits then shows that $\R \lim_n \comp_n$ gives a map
\[ \widehat{\R\Gamma_\crys( (\overline{X},\can)/W, \calO_\crys)}  \to \R\Gamma(X_{\overline{K},\et},\Z_p) \otimes_{\Z_p} A_\crys.\]
This map is linear over the algebra map 
\[ \widehat{\fraa_\crys( (\calO_K,\can)/W)} \to \widehat{\fraa_\crys( (\overline{\calO_K},\can)/W)} \simeq A_\crys. \]
Hence, linearisation gives an $A_\crys$-linear map
\begin{equation}
\label{eq:basiccompmap}
\widehat{\R\Gamma_\crys( (\overline{X},\can)/W, \calO_\crys)} \otimes_{\widehat{\fraa_\crys( (\calO_K,\can)/W)}} A_\crys \to \R\Gamma(X_{\overline{K},\et},\Z_p) \otimes_{\Z_p} A_\crys.
\end{equation}
We base change this along $A_\crys \to \widehat{A_\st}$ to get a map
\[ \widehat{\R\Gamma_\crys( (\overline{X},\can)/W, \calO_\crys)} \otimes_{\widehat{\fraa_\crys( (\calO_K,\can)/W)}} \widehat{A_\st} \to \R\Gamma(X_{\overline{K},\et},\Z_p) \otimes_{\Z_p} \widehat{A_\st}.\]
Proposition \ref{prop:stpiast} shows that the map $\widehat{\fraa((\calO_K,\can)/W)} \to \widehat{A_\st}$ factors $G_K$-equivariantly through the natural map $\widehat{\fraa((\calO_K,\can)/W)} \to R_{\calO_K}$.  Hence, one can rewrite the preceding map as
\[ \Big(\widehat{\R\Gamma_\crys( (\overline{X},\can)/W, \calO_\crys)} \otimes_{\widehat{\fraa_\crys( (\calO_K,\can)/W)}} R_{\calO_K}\Big) \otimes_{R_{\calO_K}} \widehat{A_\st} \to \R\Gamma(X_{\overline{K},\et},\Z_p) \otimes_{\Z_p} \widehat{A_\st}.\]
The parenthesized term on the left can be identified with $\widehat{\R\Gamma_\crys( (\overline{X},\can)/(W[x],x),\calO_\crys)}$: comparison with the crystalline theory identifies $\widehat{\R\Gamma_\crys( (\overline{X},\can)/W, \calO_\crys)}$ with the complex  
\[ \Big(\widehat{\R\Gamma_\crys( (\overline{X},\can)/(W[x],x),\calO_\crys)} \stackrel{d}{\to} \widehat{\R\Gamma_\crys( (\overline{X},\can)/(W[x],x),\calO_\crys)} \cdot \frac{dx}{x}\Big),\] 
where the differential is defined using the Gauss-Manin connection, while the complex $\widehat{\fraa( (\calO_K,\can)/W)}$ is identified with the complex
\[ R_{\calO_K} \otimes_{W[x]} \Big(W[x] \stackrel{d}{\to} W[x] \cdot \frac{dx}{x}\Big) \simeq \Big(R_{\calO_K} \stackrel{d}{\to} R_{\calO_K} \cdot \frac{dx}{x}\Big). \]
Thus we obtain the promised map
\[ \comp^\st_\et:\widehat{\R\Gamma_\crys( (\overline{X},\can)/(W[x],x), \calO_\crys)} \otimes_{R_{\calO_K}} \widehat{A_\st} \to \R\Gamma(X_{\overline{K},\et},\Z_p) \otimes_{\Z_p} \widehat{A_\st}. \qedhere \]
\end{proof}

\begin{remark}
It is clear from the construction that the only reason to base change up to $\widehat{A_\st}$ from $A_\crys$ is to ensure that the map $\widehat{\fraa((\calO_K,\can)/W)} \to \widehat{A_\st}$ factors $G_K$-equivariantly through the natural map $\widehat{\fraa((\calO_K,\can)/W)} \to R_{\calO_K}$. If we are prepared to work only $G_{K_\infty}$-invariantly (with notation as in Proposition \ref{prop:stpigeom}), then this base change is unnecessary by the same proposition, i.e., the map \eqref{eq:basiccompmap} above can be identified with a $G_{K_\infty}$-invariant comparison map
\[ \widehat{\R\Gamma_\crys( (\overline{X},\can)/(W[x],x), \calO_\crys)} \otimes_{R_{\calO_K}} A_\crys  \to \R\Gamma(X_{\overline{K},\et},\Z_p) \otimes_{\Z_p} A_\crys. \]
This is the form of the comparison map in \cite{FaltingsAEE}.
\end{remark}

\begin{remark}
Let $X$ be a smooth $K$-variety, and fix a hypercovering $(Y_\bullet,\overline{Y_\bullet})$ in $\calP_{\overline{K}}$ resolving $X_{\overline{K}}$, i.e., $(Y_\bullet,\overline{Y_\bullet}) \in \calP_{\overline{K}}$ is a simplicial object equipped with an equivalence $|Y_\bullet| \simeq X_{\overline{K}}$. Applying Theorem \ref{thm:padicpoincare} and following the arguments above gives a map $\fraa_\crys( (Y_\bullet,\overline{Y_\bullet})/W)[1/p] \to R\Gamma(X_{\overline{K},\et},\Z_p) \otimes A_\crys[1/p]$. It is tempting to identify the left hand side with $R\Gamma_\dR(X/K) \otimes_K \overline{K}$ as de Rham cohomology satisfies $h$-descent in characteristic $0$; this would give a de Rham comparison isomorphism over $A_\crys[1/p]$. However, this reasoning is flawed: $\fraa_\crys( (Y_\bullet,\overline{Y_\bullet})/W)[1/p]$ computes {\em derived} de Rham cohomology in characteristic $0$, so it is quite degenerate (see Corollary \ref{cor:ddrchar0}). To recover the de Rham comparison theorem $C_\dR$, one must work with the Hodge-completed picture so as to make the characteristic $0$ theory non-degenerate. When implemented, this strategy leads to the proof in \cite{Beilinsonpadic}.
\end{remark}

\begin{proof}[Deducing consequences from \cite{Beilinsonpadic}]
Construction shows that the comparison map defined above is compatible with the one defined in \cite{Beilinsonpadic}: the map defined in {\em loc. cit.} uses the same site $\calP_{\overline{K}}$, and is defined using the Hodge-completed versions of the sheaves used above. This means that we can deduce consequences for the map defined above from those proven in \cite{Beilinsonpadic}, provided we work in suitable torsion free contexts. In particular, if $(\G_m,\overline{\G_m})$ denotes the usual semistable compactification of $\G_m$ relative to $\calO_K$ and $t$ is the co-ordinate on $\G_m$, then the map
\[ \widehat{A_\st} \cdot \frac{dt}{t} := \widehat{H^1_\crys( (\G_m,\overline{\G_m})/(W[x],x), \calO_\crys)} \otimes_{R_{\calO_K}} \widehat{A_\st} \to H^1_\et(\G_m \otimes \overline{K},\Z_p) \otimes \widehat{A_\st} \simeq H^1_\et(\G_m \otimes \overline{K},\Z_p(1)) \otimes \widehat{A_\st}(-1) \]
sends the generator $\frac{dt}{t}$ to the element $\kappa \otimes \beta$; here 
\[ \kappa \in H^1_\et(\G_m \otimes \overline{K},\Z_p(1))\] 
denotes the generator of that group determined by the compatible system of the Kummer torsors (i.e., the $p^n$-power map on $\G_m$), and 
\[ \beta \in A_\crys(-1) \subset \widehat{A_\st}(-1) \simeq \Hom(\Z_p(1),\widehat{A_\st})\] 
is Fontaine's map $(\epsilon) \mapsto \log([\epsilon])$. Given this compatibility, one formally deduces many more. For example, since the comparison map commutes with Mayer-Vietoris sequences, we deduce that the following diagram commutes:
\[ \xymatrix{ \widehat{A_\st}\{-1\} \ar[r]^\beta \ar[d]^-{c_1^\crys(\calO(1))} & \widehat{A_\st}(-1) \ar[d]^{c_1^\et(\calO(1))} \\
\widehat{H^2( \P^1_{\calO_K} /(W[x],x),\calO_\crys )} \otimes_{R_{\calO_K}} \widehat{A_\st} \ar[r]^-\comp & H^2(\P^1_{\overline{K},\et},\Z_p) \otimes_{\Z_p} \widehat{A_\st}. }\]
Here $\widehat{A_\st}\{-1\}$ denotes the ring $\widehat{A_\st}$ with the filtration shifted by $1$, and the fact used above is that $\frac{dt}{t}$ and $\kappa$ correspond to the Chern classes of $\calO(1)$ in the de Rham and the \'etale theories under the Mayer-Vietoris identification of $H^1(\G_m)$ with $H^2(\P^1)$. Compatibility with cup-products combined with compatibility with restriction along a hyperplane $\P^{n-1} \subset \P^n$ leads to similar diagrams as above for $H^*(\P^n)$. In particular, the comparison map commutes with Chern classes of ample line bundles, up to the appropriate power of $\beta$, i.e., for a semistable pair $(X,\overline{X}) \in \calP_K$ and an ample line bundle $\calL \in \Pic(\overline{X})$, we have a commutative diagram for all $d$
\[ \xymatrix{ \widehat{A_\st}\{-d\} \ar[r]^{\beta^d} \ar[d]^-{(c_1^\crys(\calO(1)))^d} & \widehat{A_\st}(-d) \ar[d]^{(c_1^\et(\calO(1)))^d} \\
\widehat{H^{2d}( (\overline{X},\can) /(W[x],x),\calO_\crys )} \otimes_{R_{\calO_K}} \widehat{A_\st} \ar[r]^-\comp & H^{2d}(X_{\overline{K},\et},\Z_p) \otimes_{\Z_p} \widehat{A_\st}. }\]
Since any line bundle can be written as a difference of ample line bundles on a projective scheme, we deduce the same for arbitrary line bundles. Passing to the flag variety then proves the same statement for arbitrary vector bundles.
\end{proof}

\begin{proof}[Gysin compatibility]
Fix proper smooth geometrically connected $K$-schemes $X$ and $Y$. Assume that there exist semistable pairs $(X,\overline{X})$ and $(Y,\overline{Y})$ in $\calP_K$ extending $X$ and $Y$, and a morphism $i:(Y,\overline{Y}) \to (X,\overline{X})$ of pairs such that $\overline{Y} \to \overline{X}$ is a closed immersion of codimension $c$ that is transverse to all the strata, i.e., \'etale locally on $\overline{X}$, we have an isomorphism $(\overline{X},\overline{Y}) \simeq X^* \times (\A^c,\{0\})$ for some semistable scheme $X^*$ (see \cite[Theorem 2, page 252]{FaltingsAEE}). Then Poincare dulaity (see \cite[page 248]{FaltingsAEE}) gives an adjoint pushforward map
\[ i^\crys_* : \widehat{\R\Gamma_\crys( (\overline{Y},\can)/(W[x],x),\calO_\crys)}\{-c\}[-2c] \to \widehat{\R\Gamma( (\overline{X},\can)/(W[x],x),\calO_\crys)}. \]
Similarly, by \cite[Theorem XVII.3.2.5]{SGA4.5}, we also have a pushforward
\[ i^\et_\ast:\R\Gamma(Y_{\overline{K},\et},\Z_p)(-c)[-2c] \to \R\Gamma(X_{\overline{K},\et},\Z_p) \]
that is Poincare dual to the pullback. We claim that these commute with $\comp$ up to $\beta^c$. This is proven via a deformation to the normal cone argument which reduces considerations to the case where $\overline{X} = \P(\calN^\vee \oplus \calO_{\overline{Y}})$ for some vector bundle $\calN^\vee$ on $\overline{Y}$ with $i$ being the $0$ section of $\calN^\vee$. Instead of repeating the argument here, we simply refer to \cite[Proposition 14.7]{OlssonFaltingsAEE}; the setup there assumes that $\overline{X}$ and $\overline{Y}$ are smooth, but this is not necessary for the proof as long as $i$ is transverse as above (use \cite[bottom of page 249]{FaltingsAEE} to commute $i^\crys_\ast$ with transverse pullbacks).
\end{proof}

\begin{proof}[Verification of Chern class behaviour]
For the reader's convenience, we recall the Chern class compatibility of $\comp^\st_\et$; this discussion is simply a version of \cite[\S 3.6]{Beilinsonpadic} in the present context. Note that all objects involved --- $\G_m$, $\overline{\G_m}$, $\frac{dt}{t}$, $\kappa$, and the comparison map --- are defined over $W$. Since the comparison maps are compatible with change of base field, we can assume that $\calO_K = W$. We will show the desired compatibility modulo $p^n$ for all $n$. 
	
	Fix an integer $n \geq 0$, and let $f_n:(T_n,\overline{T_n}) \to (\G_m,\overline{\G_m})$ be the semistable compactification of the  $p^n$-power map on $\G_m$ obtained by taking $(T_n,\overline{T_n}) = (\G_m,\overline{\G_m})$, with the map being the $p^n$-th power map on $\overline{\G_m}$. The map $f_n$ is $\mu_{p^n}$-equivariant for the standard $\mu_{p^n}$-action on the source, and so we have a pullback map
\[ f_n^*: \fraa_{\crys}(\G_m,\overline{\G_m})/p^n \to \Big( \fraa_{\crys}( (T_n,\overline{T_n}) \otimes_W \overline{\calO_K})/p^n \Big)^{h\mu_{p^n}} \simeq \Big(\fraa_{\crys}(T_n,\overline{T_n})/p^n \otimes_{W/p^n} A_\crys/p^n \Big)^{h\mu_{p^n}}. \]
Here the right hand side is the homotopy-fixed points of the $\mu_{p^n}(\overline{\calO_K})$-action on the displayed complex, and can be computed via group cohomology.  We will identify the image of $\frac{dt}{t}$ under $f_n^*$. We need some notation first. Let $t_n$ be the co-ordinate on $T_n$ satisfying  $t = t_n^{p^n}$. The formula $\zeta \mapsto \frac{d\zeta}{\zeta}$ will be viewed as defining a map
\[ c:\mu_{p^n}(\overline{\calO_K}) \to A_\crys/p^n\]
obtained from the first Chern class map $c_1:\mu_{p^\infty} \to A_\crys[1]$ of Construction \ref{cons:acryschern} by the formula $c = \pi_1(\widehat{c_1})/p^n$. This is simply the reduction of Fontaine's map $\beta$ modulo $p^n$ by Proposition \ref{prop:acrysbeta}.

\begin{claim}
The image of 
\[ \frac{dt}{t} \in \pi_{-1}(\fraa_{\crys}(\G_m,\overline{\G_m})/p^n)\] 
under $\pi_{-1}(f_n^*)$ coincides with the class defined by the $1$-cocycle in group cohomology of $\mu_{p^n}(\overline{\calO_K})$ (computed using the standard complex) determined by the map $\mu_{p^n}(\overline{\calO_K}) \stackrel{c}{\to} A_\crys/p^n \to \fraa_{\crys}(T_n,\overline{T_n})/p^n \otimes_{W/p^n} A_\crys/p^n$.
\end{claim}

\begin{proof}[Proof sketch]
The element $\frac{dt}{t}$ maps to 
\[ 0 = p^n \cdot \frac{dt_n}{t_n} = \frac{d (t_n^{p^n})}{t_n^{p^n}} \in \pi_{-1}(\fraa_{\log}( (T_n,\overline{T_n}) \otimes_W \overline{\calO_K})/p^n),\]
so $\pi_{-1}(f_n^*)(\frac{dt}{t})$ is the obstruction to $\frac{dt_n}{t_n}$ being $\mu_{p^n}$-invariant, but this obstruction is tautologically the map 
\[ \zeta \mapsto \frac{d(\zeta t_n)}{\zeta t_n} -  \frac{dt_n}{t_n} = \frac{d\zeta}{\zeta}. \qedhere\]
\end{proof}

Now consider the diagram
\[ \xymatrix{ 
K_1 := \fraa_{\crys}(\G_m,\overline{\G_m})/p^n \ar[r] & \Big(\fraa_{\crys}( (T_n,\overline{T_n}) \otimes_W \overline{\calO_K})/p^n\Big)^{h\mu_{p^n}} \ar[dd] \\
	& \\
\Big(\calA^c_{\crys}( (T_n,\overline{T_n}) \otimes_W \overline{\calO_K})/p^n\Big)^{h\mu_{p^n}} \ar[r]_-{\simeq}^c &  \Big(\calA_{\crys}( (T_n,\overline{T_n}) \otimes_W \overline{\calO_K})/p^n\Big)^{h\mu_{p^n}} =: K_2\\
	& \\
K_3 := \Big(\R\Gamma_\et(T_n \otimes \overline{K},\Z/p^n(1)) \otimes_{\Z/p^n} A_\crys/p^n(-1) \Big)^{h \mu_{p^n}} \ar[uu]_-{\simeq}^-b & \R\Gamma_\et(\G_m \otimes_W \overline{K},\Z/p^n(1)) \otimes_{\Z/p^n} A_\crys/p^n(-1) =: K_4.  \ar[l]_-{\simeq}^-a   }\]
Here all maps are the natural ones, the map $a$ is an isomorphism by \'etale descent, $b$ is an isomorphism by the computation of the cohomology of constant sheaves in the $h$-topology, and $c$ is an isomorphism by the Poincare lemma.

The earlier computation shows that the  $\frac{dt}{t} \in \pi_{-1}(K_1)$  maps to the class in $\pi_{-1}(K_2)$ determined by the cocycle $\zeta \mapsto 1 \otimes \frac{d\zeta}{\zeta}$. On the other hand, since the torsor $T_n \to \G_m$ is precisely the torsor determined by $\kappa$ modulo $p^n$, the class $\kappa \otimes \beta \in \pi_{-1}(K_4)$ maps under $a$ to the cocycle determined by $\id \otimes \beta$ in $\pi_{-1}(K_3)$ (computed by group cohomology). One then chases definitions to show that the image of $\id \otimes \beta$ under $c \circ b$ in $\pi_{-1}(K_2)$ coincides with the earlier map.
\end{proof}

\begin{proof}[Proof of Theorem \ref{thm:compcst}]
We have already constructed the map $\comp$ and shown that it respects pullbacks, cup products, Chern classes of vector bundles and Gysin maps. As the map $\calA^c_\crys/p^n \to \calA_\crys/p^n$ occurring in Theorem \ref{thm:padicpoincare} respects Frobenius actions (with actions defined using Theorem \ref{prop:frobactionddr}), so does $\comp$. 

For monodromy compatibility, consider the map
\[ \comp':\widehat{\R\Gamma( (\overline{X},\can)/(W[x],x),\calO_\crys)} \to \R\Gamma_\et(X_{\overline{K},\et},\Z_p) \otimes_{\Z_p} \widehat{A_\st} \]
whose $\widehat{A_\st}$-linearisation yields $\comp$. As explained in Remark \ref{rmk:astddrdefn}, we can identify
\[ \widehat{A_\st} \simeq \widehat{\R\Gamma_\crys(f,\calO_\crys)}\]
where $f:(W[x]x) \to (\overline{\calO_K},\can)$ is the map defined by $f(x) = \pi$. Thus, the $(W[x],x)$-modules occurring on both sides of $\comp'$ acquire a connection relative to $W$ by the Gauss-Manin connection on crystalline cohomology. We will prove the desired monodromy compatibility of $\comp$ by showing that $\comp'$ is equivariant for this connection.  Replacing Theorem \ref{thm:padicpoincare} with the modified version from Remark \ref{rmk:poincarelemmaalternative} in the construction of $\comp$ leads to a map
\[ \comp'':\widehat{\R\Gamma( (\overline{X},\can)/(W[x],x),\calO_\crys)} \to \R\Gamma_\et(X_{\overline{K},\et},\Z_p) \otimes_{\Z_p} \widehat{\dR_f} \]
where $f:(W[x],x) \to (\overline{\calO_K},\can)$ is the map defined by $x \mapsto \pi$. The map $\comp'$ is obtained from $\comp''$ by composition with $\comp_f:\widehat{\dR_f} \to  \widehat{\R\Gamma_\crys(f,\calO_\crys)}$ from Remark \ref{rmk:astddrdefn}. Since $\comp_f$ is equivariant for the natural connection (by Remark \ref{rmk:astddrdefn2}), it suffices to show that $\comp''$ is equivariant for the connection. This follows from the connection-equivariance of the map $\fraa^c_\st \to \fraa_\st$ from Remark \ref{rmk:poincarelemmaalternative}, which is obvious. Hence, $\comp'$ (and thus $\comp$) are equivariant for the Gauss-Manin connection, as desired.

To see that $\comp$ admits an inverse up to $\beta^d$, note that both the source and target of $\comp$ satisfy Poincare duality (by \cite[page 248]{FaltingsAEE}). A formal argument (using the regularisation of the diagonal defined in \cite[pp 238-239]{FaltingsAEE},  and the Gysin and Chern class compatibility of $\comp$) then implies that $\comp$ admits an inverse up to $\beta^d$. 
\end{proof}

\begin{remark}
The method employed above can be used to show a comparison result between de Rham and \'etale cohomology over {\em global fields}, as we now sketch in case of $\Q$. Fix an algebraic closure $\overline{\Q}$ of $\Q$, and let $\overline{\Z}$ be the integral closure of $\Z$ in $\overline{\Q}$. Set $A_\ddR$ to be the derived projective limit of $\dR_{\overline{\Z}/\Z} \otimes_{\Z} \Z/n$ as $n$ varies through all integers. Then one can show that $A_\ddR$ is a filtered (ordinary) flat $\widehat{\Z}$-algebra equipped with a $\Gal(\overline{\Q}/\Q)$-action, and an endomorphism $\phi_p$ for each prime number $p$. Moreover, the methods used above can be massaged to give the following (loosely formulated) analog of $C_\st$:

\begin{theorem*}
Let $X$ be a semistable proper variety over $\Q$. Then log de Rham cohomology of a semistable model for $X$ is isomorphic to the $\widehat{\Z}$-\'etale cohomology of $X_{\overline{\Q}}$ once both sides are base changed to a localisation of $A_\ddR$ (while preserving all natural structures on either side). 
\end{theorem*}

Essentially by Proposition \ref{prop:frobactionddr}, the log de Rham cohomology of a semistable model for $X$ carries a Frobenius operator $\phi_p$ for each prime $p$; the analog of the monodromy operator is the $\widehat{\dR_{(\Z,\can)/\Z}}$-module structure. It is very conceivable that the global theorem can deduced from the $p$-adic ones by an induction procedure; this question, the preceding theorem, and related matters will be investigated elsewhere.
\end{remark}

\bibliography{padicddr}

\end{document}